\documentclass[11pt]{amsart}
\usepackage{amscd}
\usepackage{amsfonts}
\usepackage{amsmath}
\usepackage{amssymb}
\usepackage{amsthm}
\usepackage{array}
\usepackage[justification=centering]{caption}
\usepackage{color}
\usepackage{comment}
\usepackage{enumitem}
\usepackage{euscript}
\usepackage{fullpage}
\usepackage{graphicx}
\usepackage{hyperref}
\usepackage{latexsym}
\usepackage{longtable}
\usepackage{mathrsfs}
\usepackage{multirow}
\usepackage[abs]{overpic}
\usepackage{stmaryrd}
\usepackage{tabularx}
\usepackage[all,pdf]{xy}

\renewcommand{\arraystretch}{1.2}

\newcolumntype{L}{>{\raggedright\arraybackslash}m{.311\textwidth}}
\newcolumntype{M}{>{\raggedright\arraybackslash}m{.4775\textwidth}}
\newcolumntype{W}{>{\centering\arraybackslash}m{.4775\linewidth}}
\newcolumntype{H}{>{\centering\arraybackslash}m{.311\textwidth}}

\newtheorem{thm}{Theorem}[section]
\newtheorem*{mainthm}{Theorem~\ref{thm:main_thm}}
\newtheorem{cor}[thm]{Corollary}
\newtheorem{lem}[thm]{Lemma}
\newtheorem{prop}[thm]{Proposition}
\newtheorem{conj}[thm]{Conjecture}

\theoremstyle{definition}

\newtheorem{defn}[thm]{Definition}

\newtheorem{rem}[thm]{Remark}

\newcommand{\bbA}{\mathbb{A}}
\newcommand{\bbC}{\mathbb{C}}
\newcommand{\bbF}{\mathbb{F}}
\newcommand{\bbN}{\mathbb{N}}
\newcommand{\bbP}{\mathbb{P}}
\newcommand{\bbQ}{\mathbb{Q}}
\newcommand{\bbR}{\mathbb{R}}
\newcommand{\bbZ}{\mathbb{Z}}

\newcommand{\calD}{\mathcal{D}}

\newcommand{\calK}{\mathcal{K}}

\newcommand{\calO}{\mathcal{O}}

\newcommand{\calS}{\mathcal{S}}

\newcommand{\kbar}{\overline{k}}

\newcommand{\QQbar}{\overline{\bbQ}}

\renewcommand{\hbar}{\overline{h}}

\newcommand{\ord}{\operatorname{ord}}
\newcommand{\Gal}{\operatorname{Gal}}
\newcommand{\id}{\operatorname{id}}

\newcommand{\tors}{\operatorname{tors}}

\newcommand{\Res}{\operatorname{Res}}

\newcommand{\PrePer}{\operatorname{PrePer}}

\newcommand{\Aut}{\operatorname{Aut}}
\newcommand{\rk}{\operatorname{rk}}

\newcommand{\Ell}{\operatorname{ell}}
\newcommand{\Xell}{X^{\Ell}}
\newcommand{\Yell}{Y^{\Ell}}
\newcommand{\Supp}{\operatorname{Supp}}

\renewcommand{\tilde}{\widetilde}
\renewcommand{\phi}{\varphi}

\newcommand{\pic}{\includegraphics[scale=.6]}

\renewcommand{\labelenumi}{(\Alph{enumi})}

\usepackage{verbatim}

\numberwithin{equation}{section}

\title[Preperiodic points over quadratic fields]{Preperiodic points for quadratic polynomials with small cycles over quadratic fields}
\author{John R. Doyle}
\address{Department of Mathematics \\
University of Rochester \\
Rochester, NY 14627}
\curraddr{Mathematics \& Statistics Department\\
Louisiana Tech University\\
Ruston, LA 71272}
\email{jdoyle@latech.edu}

\begin{document}

\begin{abstract}
Given a number field $K$ and a polynomial $f(z) \in K[z]$, one can naturally construct a finite directed graph $G(f,K)$ whose vertices are the $K$-rational preperiodic points of $f$, with an edge $\alpha \to \beta$ if and only if $f(\alpha) = \beta$. The dynamical uniform boundedness conjecture of Morton and Silverman suggests that, for fixed integers $n \ge 1$ and $d \ge 2$, there are only finitely many isomorphism classes of directed graphs $G(f,K)$ as one ranges over all number fields $K$ of degree $n$ and polynomials $f(z) \in K[z]$ of degree $d$. In the case $(n,d) = (1,2)$, Poonen has given a complete classification of all directed graphs which may be realized as $G(f,\bbQ)$ for some quadratic polynomial $f(z) \in \bbQ[z]$, under the assumption that $f$ does not admit rational points of large period. The purpose of the present article is to continue the work begun by the author, Faber, and Krumm on the case $(n,d) = (2,2)$. By combining the results of the previous article with a number of new results, we arrive at a partial result toward a theorem like Poonen's --- with a similar assumption on points of large period ---  but over all quadratic extensions of $\bbQ$. 
\end{abstract}



\maketitle


\section{Introduction}\label{sec:intro}

Throughout this article, $K$ will be a number field. Let $\phi(z) \in K(z)$ be a rational map of degree $d \ge 2$, thought of as a self-map of $\bbP^1(K)$. For $n \ge 0$, let $\phi^n$ denote the $n$-fold composition of $\phi$; that is, $\phi^0$ is the identity map, and $\phi^n = \phi \circ \phi^{n-1}$ for each $n \ge 1$. A point $\alpha \in \bbP^1(K)$ is \textbf{preperiodic} for $\phi$ if there exist integers $m \ge 1$ and $n \ge 0$ for which $\phi^{m+n}(\alpha) = \phi^n(\alpha)$. In this case, the minimal such $n$ is called the \textbf{preperiod} of $\alpha$ and the minimal such $m$ is called the \textbf{eventual period} of $\alpha$; following the terminology and notation of \cite{poonen:1998}, we say that $\alpha$ has \textbf{preperiodic type} (or simply \textbf{type}) $m_n$ under $\phi$. If the preperiod $n$ is zero, then we say that $\alpha$ is \textbf{periodic} of \textbf{(exact) period} $m$. We define
\[
	\PrePer(\phi,K) := \{\alpha \in \bbP^1(K) : \alpha \text{ is preperiodic for $\phi$} \}.
\]
The set $\PrePer(\phi,K)$ comes naturally equipped with the structure of a directed graph, which we denote by $G(\phi,K)$. The vertices of $G(\phi,K)$ are the elements of $\PrePer(\phi,K)$, and there is a directed edge from $\alpha$ to $\beta$ if and only if $\phi(\alpha) = \beta$.

In 1950, Northcott \cite[Thm. 3]{northcott:1950} proved that the set $\PrePer(\phi,K)$ is finite. Drawing an analogy between preperiodic points for rational maps and torsion points on elliptic curves, Morton and Silverman have made the following conjecture, a dynamical analogue of the strong uniform boundedness conjecture (now Merel's theorem \cite{merel:1996}) for elliptic curves:

\begin{conj}[Morton-Silverman \cite{morton/silverman:1994}] \label{conj:ubc}
There is a constant $C(n,d)$ such that for any number field $K/\bbQ$ of degree $n$, and for any rational map $\phi \in K(z)$ of degree $d \ge 2$,
	\[ \# \PrePer(\phi,K) \le C(n,d). \]
\end{conj}

It is still unknown whether such a constant $C(n,d)$ exists for any pair of integers $(n,d)$. In fact, Conjecture~\ref{conj:ubc} remains unsolved even in the simplest case of quadratic \emph{polynomials} over $\bbQ$, though a significant amount of work has been done in this setting. The main result in this direction is the following result due to Poonen.

\begin{thm}[{Poonen \cite[Cor. 1]{poonen:1998}}]\label{thm:poonen_bound}
Let $f \in \bbQ[z]$ be a quadratic polynomial. If $f$ does not admit rational points of period greater than $3$, then $G(f,\bbQ)$ is isomorphic to one of the following twelve directed graphs, which appear in Appendix~\ref{app:graphs}:
	\[ 0 , \ 2(1) , \ 3(1,1) , \ 3(2) , \ 4(1,1) , \ 4(2) , \ 5(1,1)a , \ 6(1,1) , \ 6(2) , \ 6(3) , \ 8(2,1,1) , \ 8(3) . \]
In particular, $\#\PrePer(f,\bbQ) \le 9$.
\end{thm}

\begin{rem}
The upper bound given by Poonen is nine, while the largest graphs appearing in his classification have only eight vertices. This is due to the fact that the point at infinity, which is a fixed point for every polynomial, has been omitted from the graphs. In this article, we follow Poonen's convention of excluding the fixed point at infinity from $G(f,K)$.
\end{rem}

Concerning the assumption in Theorem~\ref{thm:poonen_bound}, it was conjectured by Flynn, Poonen, and Schaefer \cite{flynn/poonen/schaefer:1997} that if $f \in \bbQ[z]$ is quadratic, then $f$ does not admit rational points of period greater than three. Hutz and Ingram \cite{hutz/ingram:2013} have provided substantial computational evidence to support this conjecture. It is known that $f$ cannot have rational points of period four (Morton \cite[Thm. 4]{morton:1998}), period five (Flynn-Poonen-Schaefer \cite[Thm. 1]{flynn/poonen/schaefer:1997}), or, if one assumes certain standard conjectures on $L$-series of curves, period six (Stoll \cite[Thm. 7]{stoll:2006}). However, ruling out periods greater than six currently seems to be out of reach.

The main strategy for proving these results is as follows: Given a directed graph $G$, one constructs an algebraic curve $Y_1(G)$ defined over $\bbQ$ whose rational points parametrize (equivalence classes of) quadratic polynomial maps $f$ together with a collection of marked points which ``generate" a subgraph of $G(f,\bbQ)$ isomorphic to $G$. Then, one determines the full set of rational points on $Y_1(G)$; if $Y_1(G)$ has no rational points (outside a finite set of dynamical ``cusps"), then for no quadratic polynomial $f \in \bbQ[z]$ can the graph $G(f,\bbQ)$ have a subgraph isomorphic to $G$. To prove their results, Morton \cite{morton:1998}; Flynn, Poonen, and Schaefer \cite{flynn/poonen/schaefer:1997}; and Stoll \cite{stoll:2006} each had to determine the full set of rational points on one such algebraic curve. On the other hand, to prove Theorem~\ref{thm:poonen_bound}, Poonen had to consider the curves $Y_1(G)$ for several different graphs $G$. In the current article, as well as in previous work with Faber and Krumm \cite{doyle/faber/krumm:2014}, we adopt the same general strategy in order to attempt a classification like Theorem~\ref{thm:poonen_bound}, but over arbitrary quadratic extensions of $\bbQ$.

Over quadratic fields, cycles of length 1, 2, 3, and 4 occur infinitely often. (This follows, for example, from Propositions~\ref{prop:1or2or3} and \ref{prop:4curve}.) However, only one example (up to equivalence) of a 6-cycle is known (\cite{doyle/faber/krumm:2014,flynn/poonen/schaefer:1997,hutz/ingram:2013,stoll:2006}), and there are no known examples of 5-cycles. Experimental evidence in \cite{hutz/ingram:2013} suggests that if $K/\bbQ$ is a quadratic field and $f \in K[z]$ is a quadratic polynomial, then $f$ admits no $K$-rational points of period greater than six.

While the computations in \cite{hutz/ingram:2013} dealt specifically with the possible periods of periodic points, in order to formulate an appropriate version of Theorem~\ref{thm:poonen_bound} for quadratic extensions of $\bbQ$ we require the data of not just cycle lengths, but full preperiodic graphs $G(f,K)$. In \cite{doyle/faber/krumm:2014}, we gathered a substantial amount of data to this end: For each of the first two hundred quadratic fields $K$ (ordered by the absolute values of their discriminants), we chose a large collection of (equivalence classes of) quadratic polynomials $f \in K[z]$ and, for each such $f$, determined the full graph $G(f,K)$. The complete list of graphs (up to isomorphism) found in our search appears in Appendix~\ref{app:graphs}. The largest graphs that we found had fourteen vertices. Taking into account the additional fixed point at infinity, we make the following conjecture:

\begin{conj}\label{conj:quad_bound}
Let $K/\bbQ$ be a quadratic field, and let $f \in K[z]$ be a quadratic polynomial. Then
	\[ \# \PrePer(f,K) \le 15. \]
Moreover, the only graphs which may occur (up to isomorphism) as $G(f,K)$ for some quadratic field $K$ and quadratic polynomial $f \in K[z]$ are the $46$ graphs appearing in Appendix~\ref{app:graphs}.
\end{conj}

Our approach to Conjecture~\ref{conj:quad_bound} is twofold and follows the basic structure of Poonen's proof of Theorem~\ref{thm:poonen_bound}. First, for each of the 46 graphs $G$ appearing in Appendix~\ref{app:graphs}, we try to explicitly determine all realizations of $G$ as the graph $G(f,K)$ for some quadratic field $K$ and quadratic polynomial $f \in K[z]$. In particular, if there are only finitely many such instances (up to equivalence; see below) for a particular graph $G$, we would like to find all of them (or show that we already have), typically by determining the full set of quadratic points on a certain algebraic curve $Y_1(G)$ as described above. This is the direction we pursued in \cite{doyle/faber/krumm:2014}.

However, since in \cite{doyle/faber/krumm:2014} we dealt specifically with the 46 graphs we \emph{did} find in our initial search, our analysis yielded little evidence to suggest that there could not be additional graphs that our search \emph{did not} find. In the current article, which is the second part of this twofold approach, we turn our attention to those graphs that we did not find in the search in \cite{doyle/faber/krumm:2014}. We consider certain directed graphs $G$ which are ``minimal" among graphs not found in Appendix~\ref{app:graphs}, and we attempt to prove that such graphs $G$ cannot be realized as $G(f,K)$ for any quadratic field $K$ and quadratic polynomial $f \in K[z]$. The eventual goal is to narrow down the possible preperiodic graphs $G(f,K)$ until all that remains is the collection of 46 graphs appearing in Appendix~\ref{app:graphs}. This goal has not been fully attained; however, our main theorem (Theorem~\ref{thm:main_thm}) is a partial result in this direction.

In order to state the main result of this paper, we must first set some notation. Two polynomials $f,g \in K[z]$ are considered \emph{dynamically equivalent} if there exists a linear polynomial $\ell \in K[z]$ such that $g = \ell^{-1} \circ f \circ \ell$, since iteration commutes with conjugation. In particular, $\ell$ induces a directed graph isomorphism $G(g,K) \overset{\sim}{\longrightarrow} G(f,K)$, so for our purposes it suffices to consider a single quadratic polynomial from each equivalence class. Every quadratic polynomial over a given number field $K$ is linearly conjugate to a unique polynomial of the form
	\[ f_c(z) := z^2 + c \]
with $c \in K$; see \cite[p. 156]{silverman:2007}, for example. We therefore restrict our attention to the one-parameter family of maps of the form $f_c$.

Combining the results from \cite{doyle/faber/krumm:2014} with those in the current article yields our main theorem:

\begin{thm}\label{thm:main_thm}
Let $K$ be a quadratic field, and let $c \in K$. Suppose $f_c$ does not admit $K$-rational points of period greater than four, and suppose that $G(f_c,K)$ is not isomorphic to one of the 46 graphs shown in Appendix~\ref{app:graphs}. Then $G(f_c,K)$	\begin{enumerate}
		\item \emph{properly} contains one of the following graphs from Appendix~\ref{app:graphs}:
			\[ \text{10(1,1)b, 10(2), 10(3)a, 10(3)b, 12(2,1,1)b, 12(4), 12(4,2); or} \]
		\item contains one of the following graphs (from Figures \ref{fig:1and4}, \ref{fig:graph12_2_minus2pts}, and \ref{fig:1and2_3}, respectively):
			\[ \text{$G_0$, $G_2$, $G_4$.} \]
	\end{enumerate}
Moreover:
	\renewcommand{\labelenumi}{(\alph{enumi})}
	\begin{enumerate}
		\item There are at most four pairs $(K,c)$ for which $G(f_c,K)$ properly contains the graph 12(2,1,1)b, at most five pairs for which $G(f_c,K)$ properly contains 12(4), and at most one pair for which $G(f_c,K)$ properly contains 12(4,2). For each such pair we must have $c \in \bbQ$.
		\item There is at most one pair $(K,c)$ for which $G(f_c,K)$ contains $G_0$, at most four pairs for which $G(f_c,K)$ contains $G_2$, and at most three pairs for which $G(f_c,K)$ contains $G_4$. For each such pair we must have $c \in \bbQ$.
	\end{enumerate}
	\renewcommand{\labelenumi}{(\Alph{enumi})}
\end{thm}

Though Theorem~\ref{thm:main_thm} is not a complete solution to Conjecture~\ref{conj:quad_bound}, it offers a precise description of what must be done to prove an analogue of Theorem~\ref{thm:poonen_bound} for quadratic fields. To prove Conjecture~\ref{conj:quad_bound} \emph{under the assumption that $f$ admits no quadratic points of period greater than four\footnote{A full proof of Conjecture~\ref{conj:quad_bound} would also require showing that there are no points of period greater than four, with the exception of the one known example of a point of period six. Even over $\bbQ$, however, proving non-existence of points of large period will likely require vastly different techniques, owing to the increasing complexity of the corresponding curves.}}, it suffices to completely determine the sets of quadratic points on the curves $Y_1(G)$ for the ten graphs $G$ appearing in the statement of Theorem~\ref{thm:main_thm}. The necessary analysis of these curves, which requires different, more technical methods than those used in the present article, is ongoing joint work with Krumm and Wetherell and will appear in \cite{doyle/krumm/wetherell}.

The following terminology will be used throughout the remainder of this paper: If $K$ is a quadratic field and $c \in K$, then we refer to the pair $(K,c)$ as a \textbf{quadratic pair}. For the sake of brevity, we will often ascribe to a quadratic pair $(K,c)$ the properties of the dynamical system $f_c$ over $K$. For example, we may say ``quadratic pairs $(K,c)$ that admit points of period $N$" rather than ``quadratic fields $K$ and elements $c \in K$ for which $f_c$ admits $K$-rational points of period $N$."

We briefly outline the rest of this article. In \textsection \ref{sec:DMC}, we discuss the dynamical modular curves $Y_1(G)$ for the family of maps $f_c$, where $G$ is an ``admissible graph" --- a finite directed graph with certain additional structure dictated by the dynamics of quadratic polynomials. We record in \textsection \ref{sec:curves} a collection of results concerning rational and quadratic points on algebraic curves. We then apply these results to the study of various dynamical modular curves in sections \ref{sec:periodic} and \ref{sec:preper_pts}, concluding with the proof of Theorem~\ref{thm:main_thm} in \textsection \ref{sec:main_thm}.

\subsection*{Acknowledgments}
Much of this paper appears, in some form, in my dissertation \cite{doyle:thesis}. I would like to thank Pete Clark, Dino Lorenzini, Robert Rumely, and Robert Varley for many helpful discussions while I was working on my thesis. I would also like to thank David Krumm for his suggestions on some of the quadratic point computations, and Tom Tucker for comments on an earlier draft.


\section[Dynamical modular curves]{Dynamical modular curves}\label{sec:DMC}

Given a finite directed graph $G$, we define a \emph{dynamical modular curve} whose $K$-rational points parametrize quadratic maps $f_c \in K[z]$ together with a collection of marked points that ``generate" a subgraph of $G(f_c,K)$ isomorphic to $G$. Throughout this article, we will use the notation $X_\bullet$, $Y_\bullet$, and $U_\bullet$ exclusively to represent various dynamical modular curves. When we need to refer to a \emph{classical} modular curve, which parametrizes elliptic curves together with certain level structure, we will heed the advice of \cite[p. 163]{silverman:2007} and write $\Xell_\bullet$ and $\Yell_\bullet$ to avoid confusion.

In this paper, we give only an informal description of these dynamical modular curves. A more formal treatment appears in \cite{doyle:dmc}.

\subsection{First case: periodic points}\label{sub:per_dyn}

Let $N$ be any positive integer. If $x$ is a point of period $N$ for $f_c$, then we have $f_c^N(x) - x = 0$. However, this equation is also satisfied if $x$ has period equal to a proper divisor of $N$. We therefore define the \textbf{$N$th dynatomic polynomial} to be the polynomial
	\[
		\Phi_N(x,c) := \prod_{n \mid N} \left(f_c^n(x) - x\right)^{\mu(N/n)} \in \bbZ[x,c],
	\]
which has the property that
	\begin{equation}\label{eq:Ncycle}
		f_c^N(x) - x = \prod_{n \mid N} \Phi_n(x,c)
	\end{equation}
for all $N \in \bbN$ --- see \cite[p. 571]{morton/vivaldi:1995}. (Here $\mu$ is the M\"{o}bius function.) If $(x,c) \in K^2$ satisfies $\Phi_N(x,c) = 0$, we say that $x$ has \textbf{formal period} $N$ for $f_c$. Every point of exact period $N$ has formal period $N$, but in some cases a point of formal period $N$ may have exact period $n$ a proper divisor of $N$. The fact that $\Phi_N(x,c)$ is a polynomial is shown in \cite[Thm. 4.5]{silverman:2007}.

Since $\Phi_N(x,c)$ has coefficients in $\bbZ$, the equation $\Phi_N(x,c) = 0$ defines an affine curve $Y_1(N) \subset \bbA^2$ over $K$, and this curve was shown to be irreducible over $\bbC$ by Bousch \cite[\textsection 3, Thm. 1]{bousch:1992}. We define $U_1(N)$ to be the Zariski open subset of $Y_1(N)$ on which $\Phi_n(x,c) \ne 0$ for each proper divisor $n$ of $N$. In other words, $(x,c)$ lies on $Y_1(N)$ (resp., $U_1(N)$) if and only if $x$ has formal (resp., exact) period $N$ for $f_c$. We denote by $X_1(N)$ the normalization of the projective closure of $Y_1(N)$.

There is a natural order-$N$ automorphism $\sigma$ on $U_1(N)$ given by $(x,c) \mapsto (f_c(x),c)$. We denote by $U_0(N)$ the quotient of $U_1(N)$ by $\sigma$. The $K$-rational points on the curve $U_0(N)$ parametrize quadratic maps $f_c \in K[z]$ together with $K$-rational \emph{cycles} of length $N$ for $f_c$. We define $X_0(N)$ and $Y_0(N)$ similarly.

For each $N \in \bbN$, define 
\[
	d(N) := \deg_x \Phi_N(x,c) = \sum_{n \mid N} \mu(N/n)2^n,
\]
and let $r(N) := d(N)/N$. The quantity $d(N)$ (resp., $r(N)$) represents the number of points of period $N$ (resp., periodic cycles of length $N$) for a general quadratic polynomial $f_c$ over $\QQbar$, excluding the fixed point at infinity in the case $N = 1$. The maps $X_1(N) \to \bbP^1$ and $X_0(N) \to \bbP^1$ obtained by projection onto the $c$-coordinate have degrees $d(N)$ and $r(N)$, respectively. The first few values of $d(N)$ and $r(N)$ are shown in Table~\ref{tab:d_and_r}.

\begin{table}
\centering
\renewcommand{\arraystretch}{1.5}
\caption{Values of $d(N)$ and $r(N)$ for small values of $N$}
\label{tab:d_and_r}
\begin{tabular}{|c||c|c|c|c|c|c|c|c|}
\hline
	$N$ & 1 & 2 & 3 & 4 & 5 & 6 & 7 & 8 \\
\hline
	$d(N)$ & 2 & 2 & 6 & 12 & 30 & 54 & 126 & 240 \\
\hline
	$r(N)$ & 2 & 1 & 2 & 3 & 6 & 9 & 18 & 30 \\
\hline
\end{tabular}
\end{table}

More generally, let $\{N_1,\ldots,N_m\}$ be a finite sequence of (not necessarily distinct) positive integers in nondecreasing order, and suppose that any integer $N$ occurring in the sequence appears with multiplicity at most $r(N)$. We define $U_1(N_1,\ldots,N_m)(K)$ \emph{as a set} to be
	\begin{align*}
	\{(x_1,\ldots,x_m,c) \in \bbA^{m+1}(K) : x_i &\text{ has period $N_i$ for $f_c$, and $x_i$ and $x_j$ }\\
			&\text{ have disjoint orbits whenever $i \ne j$} \}.
	\end{align*}
Note that $U_1(N_1,\ldots,N_m)(K)$ is contained in the set of $K$-rational points on the affine curve 
	\[
		\{\Phi_{N_i}(x_i,c) = 0 : i \in \{1,\ldots,m\} \} \subset \bbA^{m+1},
	\]
which is the fiber product of the curves $Y_1(N_i)$ relative to projection onto the $c$-coordinate. We let $X_1(N_1,\ldots,N_m)$ denote the normalization of the projective closure of $U_1(N_1,\ldots,N_m)$. The action of $f_c$ on each coordinate $x_i$ induces a natural automorphism of $U_1(N_1,\ldots,N_m)$; the group $\Gamma$ of such automorphisms is isomorphic to $\bigoplus_{i=1}^m (\bbZ/N_i\bbZ)$. We denote by $U_0(N_1,\ldots,N_m)$ the quotient of $U_1(N_1,\ldots,N_m)$ by $\Gamma$, and we similarly define $X_0(N_1,\ldots,N_m)$.

\subsection{General case: admissible graphs}\label{sub:admissible}
Given a number field $K$ and a parameter $c \in K$, the graph $G(f_c,K)$ necessarily has a great deal of structure and symmetry dictated by the dynamics of quadratic polynomial maps. For this reason, we restrict our attention to finite directed graphs $G$ that possess this additional structure.

\begin{defn}\label{defn:admissible}
A finite directed graph $G$ is \textbf{admissible} if it has the following properties:

\renewcommand{\labelenumi}{(\alph{enumi})}
\begin{enumerate}
\item Every vertex of $G$ has out-degree 1 and in-degree either 0 or 2.
\item For each $N \ge 2$, $G$ contains at most $r(N)$ $N$-cycles.
\end{enumerate}

We say that $G$ is \textbf{strongly admissible} if it satisfies the following additional condition:
\begin{enumerate}
\setcounter{enumi}{2}
\item If $G$ contains a fixed point (i.e., a vertex with a self-loop), then $G$ contains exactly two such vertices.
\end{enumerate}
\renewcommand{\labelenumi}{(\Alph{enumi})}
\end{defn}

Given an admissible graph $G$, we define the \textbf{cycle structure} of $G$ to be the nondecreasing list of lengths of disjoint cycles occurring in $G$. We will say that $G$ \emph{contains} the cycle structure $\tau = (N_1,\ldots,N_m)$ if $\tau$ is a subsequence of the cycle structure of $G$; that is, if $G$ has an admissible subgraph with cycle structure $\tau$.

Admissibility is a property shared by nearly all preperiodic graphs $G(f_c,K)$; more precisely, we have the following:

\begin{lem}[{\cite[Lem. 2.4 \& Cor. 2.6]{doyle:dmc}}]\label{lem:admissible}
Let $K$ be a number field, and let $c \in K$. The graph $G(f_c,K)$ is \textit{admissible} if and only if $0 \notin \PrePer(f_c,K)$ and is \textit{strongly admissible} if and only if $0 \notin \PrePer(f_c,K)$ and $c \ne 1/4$. Moreover, the set of parameters $c \in K$ for which $G(f_c,K)$ is not strongly admissible is finite.
\end{lem}

Before discussing the dynamical modular curves $U_1(G)$ for admissible graphs $G$, we require one more definition.

\begin{defn}\label{defn:generate}
Let $G$ be an admissible graph, and let $\{P_1,\ldots,P_n\}$ be a set of vertices of $G$. Let $H$ be the smallest admissible subgraph of $G$ containing all of the vertices $P_1,\ldots,P_n$. We say that $\{P_1,\ldots,P_n\}$ is a \textbf{generating set} for $H$. If any other generating set for $G$ contains at least $n$ vertices, then we call $\{P_1,\ldots,P_n\}$ a \textbf{minimal generating set} for $G$.
\end{defn}

The dynamical modular curves $U_1(G)$ are formally defined in \cite{doyle:dmc}, where it is shown that $U_1(G)$ is always an algebraic curve that is irreducible over $\bbC$. For the purposes of this article, however, given an admissible graph $G$ minimally generated by $n$ vertices, we will be content to describe $U_1(G)(K)$ simply as a subset of $\bbA^{n+1}(K)$ for all number fields $K$. In each case, $U_1(G)(K)$ will have the property that if $G(f_c,K)$ contains a subgraph isomorphic to $G$, then there is a point $(x_1,\ldots,x_n,c) \in U_1(G)(K)$ such that $\{x_1,\ldots,x_n\}$ minimally generates the subgraph $G \subseteq G(f_c,K)$. We denote by $X_1(G)$ the normalization of the projective closure of $U_1(G)$.

\begin{rem}\label{rem:redundant}
If the cycle structure of a strongly admissible graph contains a ``1," then it must contain exactly two by definition. This is reflected in the fact that the curves $X_1(1)$ and $X_1(1,1)$ are isomorphic: The fixed points for $f_c$ are the roots of $x^2 - x + c$, so if $\alpha$ is one fixed point, then the second is given by $\alpha' = 1 - \alpha$. More generally, suppose $G$ is a strongly admissible graph minimally generated by $\{P_1,P_2,\ldots,P_n\}$, where $P_1$ is a fixed point. In this case, if $G'$ is the (weakly) admissible subgraph generated by $\{P_2,\ldots,P_n\}$, then we have an isomorphism $X_1(G) \cong X_1(G')$.
\end{rem}

\begin{figure}
\centering
\begin{overpic}[scale=.5]{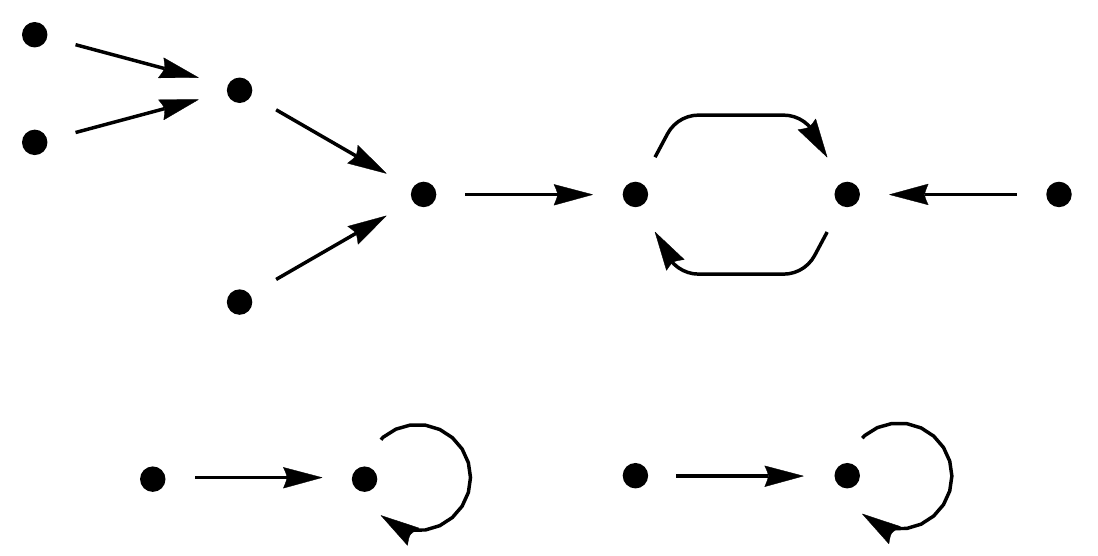}
	\put(44,0){$\alpha$}
	\put(-7,52){$\beta$}
	\put(110,0){$\alpha'$}
\end{overpic}
\caption{An admissible graph $G$}
\label{fig:U1_ex}
\end{figure}

We end this section with an example. Let $G$ be the graph appearing in Figure~\ref{fig:U1_ex}, which we further consider in \textsection \ref{sub:1and2_3}. The graph $G$ is minimally generated by the vertices $\alpha$, $\alpha'$, and $\beta$. Therefore, for each number field $K$, we define
	\[
		U_1(G)(K) = \{(\alpha,\alpha',\beta,c) \in \bbA^4(K) : \text{ $\alpha$ and $\alpha'$ are distinct fixed points and $\beta$ has type $2_3$ for $f_c$} \}.
	\]
By Remark~\ref{rem:redundant}, if $G'$ is the subgraph of $G$ generated by $\alpha$ and $\beta$, and if we write
	\[
		U_1(G')(K) = \{(\alpha,\beta,c) \in \bbA^3(K) : \text{ $\alpha$ has period $1$ and $\beta$ has type $2_3$ for $f_c$} \},
	\]
then the curves $U_1(G)$ and $U_1(G')$ are birational.


\section{Quadratic points on algebraic curves}\label{sec:curves}

Let $k$ be a number field, and let $X$ be an algebraic curve defined over $k$. We say that $P \in X(\kbar)$ is \textbf{quadratic over $k$} if the field of definition of $P$, denoted $k(P)$, is a quadratic extension of $k$. We will mostly be working in the situation that $k = \bbQ$, in which case we will simply say that $P$ is quadratic. We will denote by $X(k,2)$ the set of points on $X$ that are rational or quadratic over $k$; i.e.,
	\[
		X(k,2) := \{P \in X(\kbar) : [k(P) : k] \le 2\}.
	\]
If $X$ is an affine curve, then the \textbf{genus} of $X$ will be understood to be the geometric genus; i.e., the genus of the nonsingular projective curve birational to $X$.

Many of the curves we consider in this paper are hyperelliptic. Recall that a hyperelliptic curve defined over $k$ is a nonsingular projective curve $X$ of genus $g \ge 2$ for which there exists a degree two morphism $X \to \bbP^1$. A curve $X$ is hyperelliptic if and only if it has an affine model of the form $y^2 = f(x)$ for some polynomial $f(x) \in k[x]$ of degree $2g + 1$ or $2g + 2$ with no repeated roots. If $\deg f$ is odd, then the curve $X$ has a single $k$-rational point at infinity; if $\deg f$ is even, then the curve $X$ has two points at infinity, and they are $k$-rational if and only if the leading coefficient of $f$ is a square in $k$. A \textbf{Weierstrass point} on $X$ is a ramification point for the double cover of $\bbP^1$ given by $(x,y) \mapsto x$. In other words, $P$ is a Weierstrass point on $X$ if and only if $P = (x,0)$ or if $\deg f$ is odd and $P$ is the point at infinity.

If $X$ is a hyperelliptic curve defined over $k$, then there is a simple way to generate infinitely many points on $X$ that are quadratic over $k$: if $x \in k$ is such that $f(x)$ is not a $k$-rational square, then the point $(x,\sqrt{f(x)} )$ is quadratic over $k$. The fact that there are infinitely many such $x \in k$ is a consequence of Hilbert's irreducibility theorem. Following the terminology introduced in \cite{doyle/faber/krumm:2014}, quadratic points $(x,y)$ with $x \in k$ (resp., $x \not \in k$) will be called \textbf{obvious} (resp., \textbf{non-obvious}) quadratic points. We denote by $X(k,2)^o$ and $X(k,2)^n$ the sets of obvious and non-obvious quadratic points, respectively, on $X$. We will use the same terminology for genus one curves of the form $y^2 = f(x)$ with $\deg f \in \{3,4\}$, in which case we have the following result:

\begin{lem}[{\cite[Lem. 4.5.3]{krumm:2013}, \cite[Lem. 2.2]{doyle/faber/krumm:2014}}]\label{lem:ECquad_pts}
Let $k$ be a number field, let $E$ be an elliptic curve defined over $k$, and fix a model for $E$ of the form
	\[ y^2 = ax^3 + bx^2 + cx + d \]
with $a,b,c,d \in k$. Let $(x,y)$ be a quadratic point on $E$ with $x \not \in k$. Then there exist a point $(x_0,y_0) \in E(k)$ and an element $v \in k$ such that $y = y_0 + v(x - x_0)$ and
	\[ x^2 + \frac{ax_0 - v^2 + b}{a} x + \frac{ax_0^2 + v^2x_0 + bx_0 - 2y_0v + c}{a} = 0. \]
\end{lem}

We now consider curves of genus two, and much of what follows may be found in \cite{cassels/flynn:1996}. Let $X$ be a curve of genus two defined over a number field $k$. Since every genus two curve $X$ is hyperelliptic, $X$ has an affine model of the form $y^2 = f(x)$ with $f(x) \in k[x]$ of degree $d \in \{5,6\}$ having no repeated roots. We denote by $\iota$ the hyperelliptic involution of $X$:
	\[ \iota(x,y) = (x,-y). \]
Let $J$ be the Jacobian of $X$. If we assume that $X$ has a $k$-rational point (this is guaranteed if $d = 5$), then we may identify the Mordell-Weil group $J(k)$ with the group of $k$-rational degree zero divisors of $X$ modulo linear equivalence (see \cite[p. 39]{cassels/flynn:1996} and \cite[p. 168]{milne:1986}, for example). For $n \ge 2$, we will denote by $J(k)[n]$ the set of $k$-rational $n$-torsion points on $J$, and we denote by $J(k)_{\tors}$ the full torsion subgroup of $J(k)$.

If we let $\infty^+$ and $\infty^-$ be the points on $X$ at infinity, with $\infty^+ = \infty^-$ if $f$ has degree five, then the divisor $\infty^+ + \infty^-$ is a $k$-rational divisor on $X$, and the divisor class $\calK$ containing $\infty^+ + \infty^-$ is the canonical divisor class. By the Riemann-Roch theorem, every degree two divisor class $\calD$ contains an effective divisor, and this effective divisor is unique if and only if $\calD \ne \calK$. The effective divisors in the canonical class $\calK$ are precisely those divisors of the form $P + \iota P$. We may therefore represent every nontrivial element of $J(k)$ \emph{uniquely} by a divisor of the form $P + Q - \infty^+ - \infty^-$, up to reordering of $P$ and $Q$, where $P + Q$ is a $k$-rational divisor on $X$ (either $P$ and $Q$ are both $k$-rational points on $X$ or $P$ and $Q$ are Galois conjugate quadratic points on $X$). We therefore represent points of $J(k)$ as unordered pairs $\{P,Q\}$, with the identification
	\[ \{P,Q\} = \left[ P + Q - \infty^+ - \infty^- \right], \]
where $[D]$ denotes the divisor class of the divisor $D$. Note that $\{P,Q\} = \calO$ if and only if $[P + Q] = \calK$; that is, $\{P,Q\}$ is trivial if and only if $Q = \iota(P)$.

The identification of points on $J(k)$ as pairs $\{P,Q\}$ is useful for describing the non-obvious quadratic points on the genus two curve $X$. Let $P$ be a quadratic point on $X$, and let $\sigma$ be the nontrivial automorphism of $k(P)/k$. If $P \in X(k,2)^o$, then we can write $P = (x,\sqrt{z})$ for some $x,z \in k$ with $z$ not a square in $k$. Then $\sigma(P) = (\sigma(x) , \sigma(\sqrt{z})) = (x,-\sqrt{z}) = \iota(P)$, so that the point $\calD = \{P,\sigma(P)\} \in J(k)$ is equal to $\{P,\iota(P)\} = \calO$. On the other hand, if $P = (x,y) \in X(k,2)^n$, then $x \not \in k$, hence $\sigma(x) \ne x$. It follows that $\sigma(P)$ is equal to neither $P$ nor $\iota(P)$, so $\{P,\sigma(P)\}$ is a nontrivial element of $J(k)$.

This distinction between obvious and non-obvious quadratic points $P \in X(k,2)$ in terms of the corresponding points $\{P,\sigma(P)\} \in J(k)$ is not new; see \cite{bosman/etc:2013} for an application involving classical modular curves. However, we record here a useful consequence. A similar result is stated in \cite[Lem. 2.3]{doyle/faber/krumm:2014}, though we require a slightly more explicit version in the present article.

\begin{lem}\label{lem:rank0jac}
Let $X$ be a genus two curve defined over a number field $k$, and fix an affine model $y^2 = f(x)$ for $X$, where $f(x) \in k[x]$ has degree five or six and has no repeated roots. Suppose $X$ has a $k$-rational point. Let
	\[
		 \{ P_1,Q_1\}, \{P_2,Q_2\}, \{P_3,Q_3\}, \ldots 
	\]
be a (possibly finite) enumeration of the nonzero points of $J(k)$, and let $\Supp J(k)$ denote the set $\{P_1,Q_1,P_2,Q_2,\ldots,\} \subseteq X(k,2)$.

\begin{enumerate}
	\item If $X$ has only one $k$-rational point, then
		\[
			\Supp J(k) = X(k,2)^n.
		\]
	\item If $X$ has at least two $k$-rational points, then
		\[
			\Supp J(k) = X(k,2)^n \cup X(k).
		\]
\end{enumerate}
In particular, $J(k)$ is finite if and only if $X(k,2)^n$ is finite.
\end{lem}

\begin{proof}
From the discussion immediately preceding the lemma, we have
	\[
		X(k,2)^n \subseteq \Supp J(k) \subseteq X(k,2)^n \cup X(k).
	\]
First, suppose $X$ has only a single $k$-rational point $P$. Then $P$ is necessarily a Weierstrass point, since $\iota P \in X(k)$ as well. If $\{P,Q\} \in J(k)$, then $Q \in X(k)$; however, this implies that $Q = P = \iota P$, so $\{P,Q\} = \calO$. Therefore $P \not \in \Supp J(k)$, hence $X(k)$ and $\Supp J(k)$ are disjoint, proving (A).

Now suppose $X$ has more than one $k$-rational point, and let $P \in X(k)$. If $P$ is not a Weierstrass point, then $\{P,P\}$ is a nonzero point on $J(k)$. If $P$ is a Weierstrass point, then choosing any $Q \in X(k)$ different from $P$ yields a nonzero point $\{P,Q\} \in J(k)$. In both cases, we have $P \in \Supp J(k)$, and therefore $X(k) \subseteq \Supp J(k)$, proving (B).

Finally, since $X(k)$ is finite by Faltings' theorem \cite[Satz 7]{faltings:1983}, the set $X(k,2)^n$ is finite if and only if $\Supp J(k)$ is finite, which is equivalent to the statement that $J(k)$ is finite.
\end{proof}

We end this section by recording a result that allows us to bound the number of \emph{rational} points on some of curves that we discuss in the next two sections. The following result due to Stoll \cite{stoll:2006} is obtained by a modification of the method of Chabauty and Coleman \cite{chabauty:1941,coleman:1985}. Though a more general version holds for arbitrary number fields, we will only use the theorem for rational points, and we therefore state the result only in this case.

\begin{thm}[{\cite[Cor. 6.7]{stoll:2006}}]\label{thm:stoll}
Let $X$ be a curve of genus $g \ge 1$ defined over $\bbQ$, let $J$ be the Jacobian of $X$, and let $p$ be a prime of good reduction for $X$. Suppose that $r := \rk J(\bbQ) < g$. Then
	\[ \#X(\bbQ) \le \#X(\bbF_p) + 2r + \left\lfloor \frac{2r}{p - 2} \right\rfloor. \]
In particular, if $p > 2(r+1)$, then
	\[ \#X(\bbQ) \le \#X(\bbF_p) + 2r. \]
\end{thm}


\section{Periodic points for quadratic polynomials}\label{sec:periodic}

For each $N \in \{1,2,3,4\}$, it is known that there are infinitely many quadratic pairs $(K,c)$ that admit points of period $N$. On the other hand, there are no known quadratic pairs admitting points of period $5$, and there is only one known pair $(K,c) = \left(\bbQ(\sqrt{33}), -71/48\right)$ that admits points of period $6$. As a result of extensive computations, it has been conjectured \cite{hutz/ingram:2013} that a quadratic pair $(K,c)$ cannot admit points of period $N > 6$. Due to the difficulty of working with the curves $U_1(N)$ for $N \ge 5$, we will henceforth restrict our attention to points of period at most 4.

In this section, we determine all possible cycle structures that a graph $G(f_c,K)$ may admit, where $(K,c)$ is any quadratic pair with no points of period greater than four. For example, we will show that although there are infinitely many quadratic pairs that admit points of period 3, and infinitely many that admit points of period 4, there are no quadratic pairs that simultaneously admit points of period 3 \emph{and} period 4. In other words, there are no quadratic pairs $(K,c)$ that admit a cycle structure of $(3,4)$. 

For each of the different cycle structures $(N_1,\ldots,N_m)$ that we consider, we first provide a model for the corresponding dynamical modular curve $U_1(N_1,\ldots,N_m)$. We then describe the sets of quadratic points on each of these models, and we use this information to describe the set of quadratic pairs $(K,c)$ for which $G(f_c,K)$ has a cycle structure that contains $(N_1,\ldots,N_m)$. In this and the subsequent section, we make extensive use of the Magma computational algebra system \cite{magma}. The necessary computations have been included in an ancillary file with this article's arXiv submission.

\subsection{The curves $U_1(N)$ for $N \le 4$}

Before we study cycle structures involving multiple cycles, we begin by giving models for $U_1(N)$ for $N \le 4$. These models will be used throughout the remainder of the section. The following result for periods 1, 2, and 3 is due to Walde and Russo.

\begin{prop}[{\cite[Thms. 1, 3]{walde/russo:1994}}]\label{prop:1or2or3}
	Let $K$ be a number field, and let $c \in K$.
	\begin{enumerate}
		\item The map $f_c$ admits a point $\alpha \in K$ of period 1 (i.e., a $K$-rational fixed point) if and only if there exists $r \in K$ such that
			\[ \alpha = 1/2 + r , \ c = 1/4 - r^2. \]
		In this case, there are exactly two fixed points,
			\[ \alpha = 1/2 + r , \ \alpha' = 1/2 - r, \]
		unless $r = 0$, in which case there is only one. Therefore $U_1(1)$ (resp., $U_1(1,1)$) is isomorphic to $\bbA^1$ (resp., $\bbA^1$ with one point removed).
		\item The map $f_c$ admits a point $\alpha \in K$ of period 2 if and only if there exists $s \in K$, with $s \ne 0$, such that
			\[ \alpha = -1/2 + s , \ c = -3/4 - s^2. \]
		In this case, there are exactly two points of period 2, and they form a 2-cycle:
			\[ \alpha = -1/2 + s , \ f_c(\alpha) = -1/2 - s. \]
		Therefore $U_1(2)$ is isomorphic to $\bbA^1$ with one point removed.
		\item The map $f_c$ admits a point $\alpha \in K$ of period 3 if and only if there exists $t \in K$, with $t(t+1)(t^2 + t + 1) \ne 0$, such that
			\[ \alpha = \frac{t^3 + 2t^2 + t + 1}{2t(t + 1)} ,\ c = -\frac{t^6 + 2t^5 + 4t^4 + 8t^3 + 9t^2 + 4t + 1}{4t^2(t + 1)^2}. \]
		In this case, the 3-cycle containing $\alpha$ consists of
			\[ \alpha = \frac{t^3 + 2t^2 + t + 1}{2t(t + 1)}, \ f_c(\alpha) = \frac{t^3 - t - 1}{2t(t + 1)} , \ f_c^2(\alpha) = -\frac{t^3 + 2t^2 + 3t + 1}{2t(t + 1)}. \]
		Therefore $U_1(3)$ is isomorphic to $\bbA^1$ with four points removed.
	\end{enumerate}
\end{prop}

\begin{rem}\label{rem:3aut}
A quick calculation shows that if $\alpha$ is a point of period 3 for $f_c$, then the parameter $t$ may be recovered from $\alpha$ and $c$ by the identity $t = \alpha + f_c(\alpha)$. From this, we can see that the order 3 automorphism on $X_1(3)$ given by $(x,c) \mapsto (f_c(x),c)$ corresponds to the order 3 automorphism on $\bbP^1$ that takes $t \mapsto -\frac{t+1}{t} \mapsto -\frac{1}{t+1} \mapsto t$.
\end{rem}

The following model for $U_1(4)$ is due to Morton. Though he does not state his result as explicitly as we do here, the details of the proof of Proposition~\ref{prop:4curve} are found in \cite[pp. 91--93]{morton:1998}.

\begin{prop}[{\cite[Prop. 3]{morton:1998}}]\label{prop:4curve}
Let $Y$ be the affine curve of genus 2 defined by the equation
	\begin{equation}\label{eq:4curve}
	v^2 = -u(u^2 + 1)(u^2 - 2u - 1),
	\end{equation}
and let $U$ be the open subset of $Y$ defined by
	\begin{equation}\label{eq:4cusps}
	v(u-1)(u+1) \ne 0.
	\end{equation}
Consider the morphism $\Phi : U \to \bbA^2$, $(u,v) \mapsto (\alpha,c)$, given by
	\[ \alpha = \frac{u-1}{2(u+1)} + \frac{v}{2u(u-1)}, \ c = \frac{(u^2 - 4u - 1)(u^4 + u^3 + 2u^2 - u + 1)}{4u(u+1)^2(u-1)^2}.\]
Then $\Phi$ maps $U$ isomorphically onto $U_1(4)$, with the inverse map given by
	\begin{equation}\label{eq:4inverse}
	u = -\frac{f_c^2(\alpha) + \alpha + 1}{f_c^2(\alpha) + \alpha - 1} \;,\;\; v = \frac{u(u-1)(2\alpha u + 2\alpha - u + 1)}{u+1}.
	\end{equation}
\end{prop}

\begin{rem}\label{rem:4aut}
As noted in \cite{morton:1998}, the dynamical modular curve $X_1(4)$ is birational to the classical modular curve $X_1^{\Ell}(16)$, since \eqref{eq:4curve} is the model for $X_1^{\Ell}(16)$ appearing in \cite{washington:1991}. The order-4 automorphism of $X_1(4)$ obtained by mapping $\alpha \mapsto f_c(\alpha)$ corresponds to the map $(u,v) \mapsto (-1/u, v/u^3)$ on this model.
\end{rem}

\subsection{Cycle structures with cycles of equal length}

From Proposition~\ref{prop:1or2or3}, a quadratic polynomial $f_c$ defined over $\QQbar$ has precisely two fixed points (unless $c = 1/4$, in which case there is only one) and precisely one 2-cycle (unless $c = -3/4$, in which case there are none). On the other hand, since $r(3) = 2$ and $r(4) = 3$, $f_c$ will generically admit \emph{two} 3-cycles and \emph{three} 4-cycles over $\QQbar$. We now show, however, that a quadratic pair $(K,c)$ may admit at most one cycle of length 3 and at most one cycle of length 4.

\subsubsection{Cycles of length three}

The following result strengthens Proposition~\ref{prop:1or2or3}(C) in the case that $K$ is a quadratic field, stating that if $(K,c)$ is a quadratic pair, then $G(f_c,K)$ cannot contain a cycle structure of (3,3).

\begin{thm}\label{thm:unique3cycle}
Let $K$ be a quadratic field, and let $c \in K$. The map $f_c$ admits a point $\alpha \in K$ of period 3 if and only if there exists $t \in K$, with $t(t+1)(t^2 + t + 1) \ne 0$, such that
	\[ \alpha = \frac{t^3 + 2t^2 + t + 1}{2t(t + 1)} , \ c = -\frac{t^6 + 2t^5 + 4t^4 + 8t^3 + 9t^2 + 4t + 1}{4t^2(t + 1)^2}.\]
In this case, there are \emph{precisely} three points of period 3, and they form a 3-cycle:
	\[ \alpha = \frac{t^3 + 2t^2 + t + 1}{2t(t + 1)} , \ f_c(\alpha) = \frac{t^3 - t - 1}{2t(t + 1)} , \ f_c^2(\alpha) = -\frac{t^3 + 2t^2 + 3t + 1}{2t(t + 1)}. \]
\end{thm}

To prove Theorem~\ref{thm:unique3cycle} we will require the following result of Morton which, in particular, implies Theorem~\ref{thm:unique3cycle} with the quadratic field $K$ replaced with $\bbQ$.

\begin{lem}[{\cite[Thm. 3]{morton:1992}}]\label{lem:3cyclesQ}
Fix $c \in \bbQ$. The degree six polynomial $\Phi_3(x,c) \in \bbQ[x]$ cannot have irreducible quadratic factors, nor can it split completely into linear factors, over $\bbQ$.
\end{lem}

\begin{cor}\label{cor:3cyclesQ}
If $c \in \bbQ$, then $f_c$ admits at most three rational points of period 3 and admits no quadratic points of period 3.
\end{cor}

In order to prove Theorem~\ref{thm:unique3cycle}, we must show that the curve 
	\[ U_1(3,3) = \{ (\alpha,\beta,c) \in \bbA^3 : \mbox{$\alpha$ and $\beta$ have period 3 and lie in distinct orbits under $f_c$} \} \]
has no quadratic points. We first provide a model for $U_1(3,3)$.
\begin{prop}\label{prop:3and3curve}
Let $Y$ be the affine curve of genus 4 defined by the equation
	\begin{equation}\label{eq:3and3curve}
	h(t,u) := t(t+1)u^3 + (t^3 + 2t^2 - t - 1)u^2 + (t^3 - t^2 - 4t - 1)u - t(t+1) = 0,
	\end{equation}
and let $U$ be the open subset of $Y$ defined by
	\begin{equation}\label{eq:3and3cusps}
	t(t+1)(t^2 + t + 1)(u^2 + u + 1)(u^3 + u^2 - 2u - 1) \ne 0.
	\end{equation}
Consider the morphism $\Phi : U \to \bbA^3$, $(t,u) \mapsto (\alpha,\beta,c)$, given by
	\begin{equation}\label{eq:unique3cycle_param}
		\alpha = \frac{t^3 + 2t^2 + t + 1}{2t(t+1)} , \ 
		\beta = \frac{u^3 + 2u^2 + u + 1}{2u(u+1)} , \
		c = -\frac{t^6 + 2t^5 + 4t^4 + 8t^3 + 9t^2 + 4t + 1}{4t^2(t+1)^2}.
	\end{equation}
Then $\Phi$ maps $U$ isomorphically onto $U_1(3,3)$, with the inverse map given by
	\begin{equation}\label{eq:3and3inverse}
	t = \alpha + f_c(\alpha) \ , \ u = \beta + f_c(\beta).
	\end{equation}
\end{prop}

\begin{rem}
Though the expression for $h(t,u)$ is different, the function field of the curve $Y$ is isomorphic to the function field $N$ studied in \cite{morton:1992}. We have introduced a change of variables in order to keep our parameter $t$ consistent with the parameter $\tau$ appearing in \cite{poonen:1998,walde/russo:1994}. The fact that $\Phi$ maps onto $U_1(3,3)$ is due to Morton \cite[pp. 354--5]{morton:1992}.
\end{rem}

\begin{proof}
In order for $\Phi$ to be well-defined, we must have $t(t+1)u(u+1) \ne 0$ for $(t,u) \in U$. That $t(t+1) \ne 0$ follows from \eqref{eq:3and3cusps}, and if $u(u+1) = 0$, then \eqref{eq:3and3curve} implies that $t(t+1) = 0$, contradicting \eqref{eq:3and3cusps}. Moreover, a simple computation shows that \eqref{eq:3and3inverse} provides a left inverse for $\Phi$, which in particular shows that $\Phi$ is injective.

Now we must show that $\Phi$ maps $U$ into $U_1(3,3)$. Let $(\alpha,\beta,c) = \Phi(t,u)$ for some $(t,u) \in U$. One can easily verify that $f_c^3(\alpha) = \alpha$, $f_c^3(\beta) = \beta$, and
	\begin{equation}\label{eq:3not1}
	\alpha - f_c(\alpha) = \frac{t^2 + t + 1}{t(t+1)} , \ \beta - f_c(\beta) = \frac{u^2 + u + 1}{u(u+1)} .
	\end{equation}
Since $t^2 + t + 1$ and $u^2 + u + 1$ are nonzero on $U$, $\alpha$ and $\beta$ cannot be fixed points. Hence $\alpha$ and $\beta$ have period exactly 3.

We also need for $\alpha$ and $\beta$ to lie in distinct orbits under $f_c$. It suffices to show that $\beta \not \in \{\alpha,f_c(\alpha),f_c^2(\alpha)\}$. A simple calculation shows that
	\[
		\begin{gathered}
		\alpha - \beta = \frac{(t+u+1) \left(t u^2+(t-2) u-1\right)}{u (u+1)} , \ f_c(\alpha) - \beta = \frac{\left(t^2-t-1\right) u-t}{t u} , \\
		f_c^2(\alpha) - \beta = -\frac{u^2 + (t+2) u+t}{u+1}.
		\end{gathered}
	\]

\begin{itemize}
	\item If $t + u + 1 = 0$, then $u = -(t+1)$, and thus $0 = h(t,u) = -t^2(t+1)^2$.
	\item If $tu^2 + (t-2)u - 1 = 0$, then $t = \frac{2u + 1}{u(u+1)}$, and thus $0 = h(t,u) = \frac{u^3 + u^2 - 2u - 1}{u + 1}$.
	\item If $(t^2 - t - 1)u - t = 0$, then $u = \frac{t}{t^2 - t - 1}$, and thus $0 = h(t,u) = \frac{t^2(t+1)(t^3 + t^2 - 2t - 1)}{(t^2 - t - 1)^3}$. Note that if $t^3 + t^2 - 2t - 1 = 0$ and $u = \frac{t}{t^2 - t - 1}$, then $u$ satisfies $u^3 + u^2 - 2u - 1 = 0$.
	\item If $u^2 + (t+2)u + t = 0$, then $t = -\frac{u(u+2)}{u+1}$, and thus $0 = h(t,u) = -u(u^3 + u^2 - 2u - 1)$.
\end{itemize}

In each of these four cases, the resulting equality contradicts \eqref{eq:3and3cusps}. Therefore $\beta$ cannot lie in the orbit of $\alpha$, from which it now follows that the image of $\Phi$ is contained in $U_1(3,3)$. It remains only to show that $\Phi$ maps onto $U_1(3,3)$.

Suppose $\alpha$ and $\beta$ are points of period 3 for $f_c$ such that $\alpha$ and $\beta$ lie in distinct orbits. By Proposition~\ref{prop:1or2or3}(C), there exist $t$ and $u$ satisfying $t(t+1)(t^2 + t + 1)u(u+1)(u^2 + u + 1) \ne 0$ such that
	\[
	\begin{gathered}
	\alpha = \frac{t^3 + 2t^2 + t + 1}{2t(t+1)} , \ \beta = \frac{u^3 + 2u^2 + u + 1}{2u(u+1)}, \\
	c = -\frac{t^6 + 2t^5 + 4t^4 + 8t^3 + 9t^2 + 4t + 1}{4t^2(t+1)^2} = -\frac{u^6 + 2u^5 + 4u^4 + 8u^3 + 9u^2 + 4u + 1}{4u^2(u+1)^2}.
	\end{gathered}
	\]
Subtracting these two expressions for $c$ and clearing denominators, we find that
	\[ (t-u) (t u+t+1) (t u+u+1) \cdot h(t,u) = 0. \]
We must therefore show that $(t-u)(tu + t + 1)(tu + u + 1) \ne 0$. It is clear that $t - u \ne 0$, since $t = u$ implies that $\alpha = \beta$. Now suppose $tu + t + 1 = 0$. Then we can write $u = -(t+1)/t$. However, substituting $u = -(t+1)/t$ into the above expression for $\beta$ yields
	\[ \beta = \frac{t^3 - t - 1}{2t(t+1)} = f_c(\alpha), \]
where the second equality follows from Proposition~\ref{prop:1or2or3}(C). If $tu + u + 1 = 0$, then a similar argument yields $\beta = f_c^2(\alpha)$. We have therefore shown that if $(t-u)(tu + t + 1)(tu + u + 1) = 0$, then $\beta$ lies in the orbit of $\alpha$ under $f_c$, a contradiction. Thus $h(t,u) = 0$, so $(t,u)$ lies on the curve $Y$.

Finally, we must show that \eqref{eq:3and3cusps} is satisfied by $(t,u)$. We have already stated that $t(t+1)(t^2+t+1)(u^2+u+1) \ne 0$, so it suffices to show that $u^3 + u^2 - 2u - 1 \ne 0$. Suppose to the contrary that $u^3 + u^2 - 2u - 1 = 0$. Then we have
	\[
		c = -\frac{7u^2(u+1)^2 + (u^3 + u^2 - 2u - 1)^2}{4u^2(u+1)^2} = -\frac{7}{4}.
	\]
However, the map $f_{-7/4}$ admits only a single 3-cycle over $\QQbar$ (see \cite[p. 358]{morton:1992}), so $\alpha$ and $\beta$ could not lie in distinct orbits.
\end{proof}

It now remains to find all quadratic points on $U_1(3,3)$.

\begin{thm}\label{thm:3and3_quad_pts}
Let $Y$ be the genus 4 affine curve defined by \eqref{eq:3and3curve}. Then
	\[ Y(\bbQ,2) = Y(\bbQ) = \{(0,-1), (0,0), (-1,0), (-1,-1)\}. \]
\end{thm}

Theorem~\ref{thm:3and3_quad_pts} immediately implies Theorem~\ref{thm:unique3cycle}, since $Y(\bbQ,2) \subset Y(\QQbar) \setminus U(\QQbar)$; in other words, $U_1(3,3)$ has no rational or quadratic points.

The computation of rational points on the curve $Y$ was done by Morton in \cite[pp. 362--364]{morton:1992}, though with a different model for $Y$ as explained above. We therefore need only show that $Y$ has no quadratic points.

Consider the automorphism $\sigma$ on $X$ defined by
	\[ \sigma(t) = -\frac{t+1}{t} , \ \sigma(u) = -\frac{1}{u+1}. \]
By Remark~\ref{rem:3aut} and the relations in \eqref{eq:unique3cycle_param}, the automorphism $\sigma$ corresponds to the automorphism on $Y_1(3,3)$ that takes $(\alpha,\beta,c) \mapsto (f_c(\alpha), f_c^2(\beta),c)$. Set
	\begin{align}
		\begin{split}
		\label{eq:3and3_wxy}
		w &:= t + \sigma(t) + \sigma^2(t);\\
		x &:= tu + \sigma(tu) + \sigma^2(tu);
		\end{split}
	\end{align}
Using the fact that $c = -\dfrac{t^6 + 2t^5 + 4t^4 + 8t^3 + 9t^2 + 4t + 1}{4t^2(t+1)^2}$, a Magma computation verifies that
	\begin{equation}\label{eq:3and3_cbyx}
		c = \frac{x^3 - 8x^2 + 19x - 13}{4(x-2)(x-3)}.
	\end{equation}
Furthermore, one can show that $w$ and $x$ satisfy the equation
	\[ (w+1)^2 = -\frac{x^3 - x^2 - 16x + 29}{(x-2)(x-3)}, \]
so that $(x,(x-2)(x-3)(w+1))$ is a point on the genus 2 curve\footnote{The curve $C$ has function field isomorphic to the field $N\_$ described in \cite[p. 368]{morton:1992}.} $C$ defined by
	\[ y^2 = -(x-2)(x-3)(x^3 - x^2 - 16x + 29). \]
By construction, we have a map $\psi: Y \to C$ given by $(t,u) \mapsto (x,(x-2)(x-3)(w+1))$, and any quadratic point on $Y$ must map to a quadratic or rational point on $C$.

\begin{lem}\label{lem:3and3_Ypts}
Let $C$ be the genus two curve defined by
	\[ y^2 = -(x-2)(x-3)(x^3 - x^2 - 16x + 29). \]
The only non-obvious quadratic points on $C$ are the points of the form
	\[  (x_1, \pm(17x_1 - 36)) , \ (x_2, \pm(11x_2 - 34)) , \ (x_3, \pm(3x_3 - 8)) , \ \left(x_4, \pm\frac{1}{27}(x_4 + 2)\right) , \ \]
	where
\[ x_1^2 + 5x_1 - 15 = x_2^2 + 3x_2 - 19 = x_3^2 - 5x_3 + 7 = 9x_4^2 - 47 x_4 + 59 = 0. \]
\end{lem}

\begin{proof}
Let $J$ denote the Jacobian of $C$. A Magma computation shows that $\rk J(\bbQ) = 0$, so we may apply Lemma~\ref{lem:rank0jac} to find all non-obvious quadratic points on $C$. We find that $J(\bbQ)$ consists of the following twelve points:

	\begin{center}
	\begin{tabular}{c c}
	$\calO$ & $\left\{\left(x_1, \pm(17x_1 - 36)\right),\left( x_1', \pm(17x_1' - 36) \right)\right\}$\\
	$\{\infty, (2,0)\}$ & $\left\{(x_2, \pm(11x_2 - 34)),(x_2', \pm(11x_2' - 34))\right\}$\\
	$\{\infty, (3,0)\}$ & $\left\{ (x_3, \pm(3x_3 - 8)) , (x_3', \pm(3x_3' - 8))\right\}$\\
	$\{(2,0),(3,0)\}$ & $\left\{ \left(x_4, \pm\frac{1}{27}(x_4 + 2)\right) , \left(x_4', \pm\frac{1}{27}(x_4' + 2)\right) \right\}$
	\end{tabular}
	\end{center}
with each $x_i$ as in the statement of the lemma and $x_i'$ the Galois conjugate of $x_i$ over $\bbQ$. Applying Lemma~\ref{lem:rank0jac} yields the result.
\end{proof}

We may now complete the proof of Theorem~\ref{thm:3and3_quad_pts}. First, a computation in Magma verifies that the preimages under $\psi$ of the non-obvious quadratic points listed in Lemma~\ref{lem:3and3_Ypts} have degree strictly greater than two. Therefore a quadratic point on $Y$ must map to an obvious quadratic point or a rational point on $C$. In either case, we have $x \in \bbQ$, which implies $c \in \bbQ$ by \eqref{eq:3and3_cbyx}.

Suppose $(t,u)$ is a quadratic point on $Y$, and let $K$ be the quadratic extension of $\bbQ$ generated by $t$ and $u$. First, we note that $t(t+1)u(u+1) \ne 0$, since the only points on $Y$ that satisfy $t(t+1)u(u+1) = 0$ are the four rational points listed above. We may therefore define $\alpha,\beta,c \in K$ as in \eqref{eq:unique3cycle_param}, and $\alpha$ and $\beta$ must either be fixed points or points of period 3 for $f_c$. If $\alpha$ or $\beta$ is a fixed point, then by \eqref{eq:3not1} we must have $(t^2 + t + 1)(u^2 + u + 1) = 0$. However, it is not difficult to verify that such a point $(t,u)$ on $Y$ must have degree strictly greater than two over $\bbQ$. Therefore $\alpha$ and $\beta$ must have period precisely three for $f_c$. By Corollary~\ref{cor:3cyclesQ}, since $c \in \bbQ$, $f_c$ cannot admit quadratic points of period 3, so $\alpha,\beta \in \bbQ$. Finally, since $\alpha,\beta,c \in \bbQ$, we have $t = \alpha + f_c(\alpha) \in \bbQ$ and $u = \beta + f_c(\beta) \in \bbQ$, a contradiction. Therefore $Y$ has no quadratic points, completing the proof of Theorem~\ref{thm:3and3_quad_pts} and, consequently, Theorem~\ref{thm:unique3cycle}. \qed

\subsubsection{Cycles of length four}

In this section, we prove the following analogue of Theorem~\ref{thm:unique3cycle} for points of period four, showing that if $(K,c)$ is a quadratic pair, then $G(f_c,K)$ cannot contain a cycle structure of (4,4).

\begin{thm}\label{thm:unique4cycle}
Let $K$ be a quadratic field, and let $c \in K$. The map $f_c$ admits a point $\alpha \in K$ of period 4 if and only if there exist $u,v \in K$, with $v^2 = -u(u^2 + 1)(u^2 - 2u - 1)$ and $v(u+1)(u-1) \ne 0$, such that
	\[ \alpha = \frac{u-1}{2(u+1)} + \frac{v}{2u(u-1)} , \ c = \frac{(u^2 - 4u - 1)(u^4 + u^3 + 2u^2 - u + 1)}{4u(u+1)^2(u-1)^2}.\]
In this case, there are \emph{precisely} four points of period 4, and they form a 4-cycle:
	\begin{align*}
	\alpha &= \frac{u-1}{2(u+1)} + \frac{v}{2u(u-1)}\\
	f_c(\alpha) &= -\frac{u+1}{2(u-1)} + \frac{v}{2u(u+1)}\\
	f_c^2(\alpha) &= \frac{u-1}{2(u+1)} - \frac{v}{2u(u-1)}\\
	f_c^3(\alpha) &= -\frac{u+1}{2(u-1)} - \frac{v}{2u(u+1)}.
	\end{align*}
Moreover, we must have $c \in \bbQ$, and $\alpha$ and $f_c^2(\alpha)$ are Galois conjugates over $\bbQ$.
\end{thm}

By Proposition~\ref{prop:4curve}, the only part of Theorem~\ref{thm:unique4cycle} that remains to be proven is the statement that a quadratic pair $(K,c)$ may admit \emph{at most} four points of period four, and that for every such pair we must have $c \in \bbQ$. We begin by proving the latter:

\begin{prop}\label{prop:4cycle_rat_c}
Let $K$ be a quadratic field, and let $c \in K$ be such that $f_c$ admits a point $\alpha \in K$ of period four. Then $c \in \bbQ$, and $\alpha$ and $f_c^2(\alpha)$ are Galois conjugates over $\bbQ$.
\end{prop}

\begin{proof}
Consider the Jacobian $J$ of the curve $Y$ defined in Proposition~\ref{prop:4curve}. Since $Y$ is birational to $\Xell_1(16)$, it is known that $\rk J(\bbQ) = 0$ and $\#J(\bbQ) = \#J(\bbQ)_{\tors} = 20$. (A proof is given in \cite[Lem. 14]{bosman/etc:2013}, for example.) Therefore the following twenty points are all of the points on $J(\bbQ)$:

	\begin{center}
	\begin{tabular}{c c c c}
	$\calO$ & $\{ \infty, ( -1 , -2 )\}$ & $\{(1,2),( 1,2) \}$ & $\{( 1 , -2 ),( -1 , -2 ) \}$ \\
	
	$\{ \infty, ( 0 , 0 )\}$ & $\{( 0 , 0 ),( 1 , 2 ) \}$ & $\{( 1 , 2 ),( -1 , 2 ) \}$ & $\{( -1 , 2 ),( -1 , 2 ) \}$ \\
	
	$\{ \infty, ( 1 , 2 )\}$ & $\{( 0 , 0 ),( 1 , -2 ) \}$ & $\{( 1 , 2 ),( -1 , -2 ) \}$ & $\{( -1 , -2 ),( -1 , -2 ) \}$ \\
	
	$\{ \infty, ( 1 , -2 )\}$ & $\{( 0, 0 ),( -1 , 2 ) \}$ & $\{( 1 , -2 ),( 1 , -2 ) \}$ & $\{(\sqrt{-1} , 0) , (-\sqrt{-1} , 0) \}$  \\
	
	$\{ \infty, ( -1 , 2 )\}$ & $\{( 0, 0 ),( -1 , -2 ) \}$ & $\{( 1 , -2 ),( -1 , 2 ) \}$ & $\{(1 + \sqrt{2} , 0) , (1 - \sqrt{2} , 0) \}$
	\end{tabular}
	\end{center}
	
Applying Lemma~\ref{lem:rank0jac}, we find that if $(u,v)$ is a non-obvious quadratic point on $Y$, then $v = 0$, and such points are not allowed by Proposition~\ref{prop:4curve}. For every other quadratic point on $Y$, we have $u \in \bbQ$. Writing
	\[ \alpha = \frac{u-1}{2(u+1)} + \frac{v}{2u(u-1)} , \ c = \frac{(u^2 - 4u - 1)(u^4 + u^3 + 2u^2 - u + 1)}{4u(u+1)^2(u-1)^2} \]
shows that $c \in \bbQ$, proving the first claim in the proposition.

Finally, recall from Proposition~\ref{prop:4curve} that
	\[
		f_c^2(\alpha) = \frac{u-1}{2(u+1)} - \frac{v}{2u(u-1)}.
	\]
Since $u \in \bbQ$, $v \not \in \bbQ$, and $v^2 \in \bbQ$, the point $f_c^2(\alpha)$ is visibly the Galois conjugate of $\alpha$.
\end{proof}

The following result, together with Propositions~\ref{prop:4curve} and \ref{prop:4cycle_rat_c}, completes the proof of Theorem~\ref{thm:unique4cycle}.

\begin{prop}\label{prop:unique4cycle}
Let $K$ be a quadratic field, and let $c \in K$. The map $f_c$ admits \emph{at most} four $K$-rational points of period 4.
\end{prop}

\begin{proof}
By Proposition~\ref{prop:4cycle_rat_c}, if $(K,c)$ is a quadratic pair that admits a point of period 4, then $c \in \bbQ$. Panraksa \cite[Thm. 2.3.5]{panraksa:2011} has shown that if one fixes $c \in \bbQ$, then the degree 12 polynomial $\Phi_4(x,c) \in \bbQ[x]$ cannot have four distinct quadratic factors over $\bbQ$. Since each point of period 4 for $f_c$ is a root of $\Phi_4(x,c)$, and since having more than four points of period 4 implies at least eight such points, it follows that $f_c$ cannot admit more than four quadratic points of period 4.
\end{proof}

\begin{rem}
A somewhat stronger statement than Proposition~\ref{prop:unique4cycle} is true. By Proposition~\ref{prop:4cycle_rat_c}, if $(K,c)$ is a quadratic pair with a point $\alpha \in K$ of period four, then actually $c \in \bbQ$, and $\alpha$ and $f_c^2(\alpha)$ (hence also $f_c(\alpha)$ and $f_c^3(\alpha)$) are Galois conjugates over $\bbQ$. Therefore the 4-cycle $\{\alpha,f_c(\alpha),f_c^2(\alpha),f_c^3(\alpha)\}$ is \emph{rational}. Proposition~\ref{prop:unique4cycle} then follows from \cite[Cor. 4.19]{doyle:thesis}, which states that if $c \in \bbQ$, then $f_c$ admits at most one rational 4-cycle.
\end{rem}

We conclude this section by combining the statements of Theorems~\ref{thm:unique3cycle} and \ref{thm:unique4cycle}:

\begin{cor}\label{cor:number_of_cycles}
If $K$ is a quadratic field and $c \in K$, then the cycle structure of $G(f_c,K)$ may contain the integer 1 at most twice and the integers 2, 3, and 4 at most once.
\end{cor}

\subsection{Cycles of different lengths}

The following result for the cycle structures $(1,2)$, $(1,3)$, and $(2,3)$ follows from the proof of \cite[Thm. 2]{poonen:1998}; see also \cite[Thm. 2]{walde/russo:1994} for the $(1,2)$ case. A detailed proof of parts (B) and (C) may also be found in \cite[Lems. 3.24, 3.27]{doyle/faber/krumm:2014}.

\begin{prop}\label{prop:1and2etc}
	\mbox{}
	\begin{enumerate}
		\item Let $K$ be a number field, and let $c \in K$. The map $f_c$ admits $K$-rational points of period 1 and period 2 if and only if
			\[ c = -\frac{\left(q^2+3\right) \left(3 q^2+1\right)}{4(q-1)^2(q+1)^2} \]
		for some $q \in K$, with $q(q-1)(q+1) \ne 0$. In this case, the parameters $r$ and $s$ from Proposition~\ref{prop:1or2or3} are given by
			\[ r = -\frac{q^2 + 1}{(q-1)(q+1)} , \ s = \frac{2q}{(q-1)(q+1)}. \]
		\item Let $Y$ be the affine curve of genus 2 defined by the equation
			\begin{equation}\label{eq:1and3curve}
				y^2 = t^6 + 2t^5 + 5t^4 + 10t^3 + 10t^2 + 4t + 1,
			\end{equation}
		and let $U$ be the open subset of $Y$ defined by
			\begin{equation}\label{eq:1and3cusps}
				t(t+1)(t^2 + t + 1) \ne 0.
			\end{equation}
		Consider the morphism $\Phi: U \to \bbA^3$,  $(t,y) \mapsto (\alpha,\beta,c)$, given by
			\[ \alpha = \frac{t^2+t+y}{2 t (t+1)} , \ \beta = \frac{t^3+2 t^2+t+1}{2 t (t+1)}, \ c = -\frac{t^6+2 t^5+4 t^4+8 t^3+9 t^2+4 t+1}{4 t^2 (t+1)^2}.\]
		Then $\Phi$ maps $U$ isomorphically onto $U_1(1,3)$, with the inverse map given by
			\[ t = \beta + f_c(\beta) , \ y = (2\alpha - 1)t(t+1). \]
		\item Let $Y'$ be the affine curve of genus 2 defined by the equation
			\begin{equation}\label{eq:2and3curve}
				z^2 = t^6 + 2t^5 + t^4 + 2t^3 + 6t^2 + 4t + 1,
			\end{equation}
		and let $U'$ be the open subset of $Y'$ defined by
			\begin{equation}\label{eq:2and3cusps}
				t(t+1)(t^2 + t + 1)z \ne 0.
			\end{equation}
		Consider the morphism $\Phi': U' \to \bbA^3$, $(t,z) \mapsto (\alpha,\beta,c)$, given by
			\[ \alpha = -\frac{t^2+t-z}{2 t (t+1)} , \ \beta = \frac{t^3+2 t^2+t+1}{2 t (t+1)} , \ c = -\frac{t^6+2 t^5+4 t^4+8 t^3+9 t^2+4 t+1}{4 t^2 (t+1)^2}.\]
		Then $\Phi'$ maps $U'$ isomorphically onto $U_1(2,3)$, with the inverse map given by
			\[ t = \beta + f_c(\beta) , \ z = (2\alpha + 1)t(t+1) .\]
	\end{enumerate}
In particular, for each pair $(N_1,N_2) \in \{(1,2),(1,3),(2,3) \}$, there exist infinitely many quadratic pairs $(K,c)$ that admit points of periods $N_1$ and $N_2$.
\end{prop}

\begin{rem}
As explained in \cite{poonen:1998}, the curves $X_1(1,3) \cong X_1(1,1,3)$ and $X_1(2,3)$ are birational to $\Xell_1(18)$ and $\Xell_1(13)$, respectively.
\end{rem}

The following description of the rational and quadratic points on the curves appearing in parts (B) and (C) of Theorem~\ref{prop:1and2etc} appears in \cite[\textsection 2.6]{krumm:2013} and \cite[Thm. 2.4]{doyle/faber/krumm:2014}.

\begin{lem}\label{lem:1and3/2and3curves}
Let $Y$ and $Y'$ be the curves defined by \eqref{eq:1and3curve} and \eqref{eq:2and3curve}, respectively.
\begin{enumerate}
	\item The set of rational points on $Y$ is
		\[ Y(\bbQ) = \{(-1, \pm 1) , (0, \pm 1)\}, \]
	and the set of quadratic points on $Y$ consists of the obvious quadratic points
		\[	Y(\bbQ,2)^o = \left\{\left(t,\pm \sqrt{t^6 + 2t^5 + 5t^4 + 10t^3 + 10t^2 + 4t + 1}\right) : t \in \bbQ, t(t+1) \ne 0\right\} \]
	and the non-obvious quadratic points
		\[	Y(\bbQ,2)^n = \left\{(t, \pm (t-1)) : t^2 + t + 1 = 0 \right\}. \]
	\item The set of rational points on $Y'$ is
		\[ Y'(\bbQ) = \{ (-1, \pm 1) , (0, \pm 1) \} ,\]
	and the set of quadratic points on $Y'$ consists only of the obvious quadratic points
		\[ Y'(\bbQ,2)^o = \left\{\left(t,\pm \sqrt{t^6 + 2t^5 + t^4 + 2t^3 + 6t^2 + 4t + 1}\right) : t \in \bbQ, t(t+1) \ne 0\right\}. \]
\end{enumerate}
\end{lem}

\begin{prop}\label{prop:1and3/2and3}
Let $K$ be a quadratic field, and let $c \in K$. If $f_c$ has $K$-rational points of periods 1 and 3 (resp. 2 and 3), then $c \in \bbQ$. In this case, the points of period 3 are rational, and the points of period 1 (resp. 2) are quadratic.
\end{prop}

\begin{proof}
We first suppose that $f_c$ admits $K$-rational points of period 1 and 3. Then there exists a $K$-rational point $(t,y)$ on the curve $Y$ defined by \eqref{eq:1and3curve}, with $t(t+1)(t^2 + t + 1) \ne 0$, for which
	\[ c = -\frac{t^6 + 2t^5 + 4t^4 + 8t^3 + 9t^2 + 4t + 1}{4t^2(t + 1)^2}. \]
By Lemma~\ref{lem:1and3/2and3curves}, the only quadratic points on $Y$ satisfying $t^2 + t + 1 \ne 0$ are those with $t \in \bbQ$. It follows from Proposition~\ref{prop:1and2etc} that $c$ and the points of period 3 are also in $\bbQ$ but, since $y \not \in \bbQ$, the fixed points for $f_c$ are not in $\bbQ$.

A similar argument applies to the case that $f_c$ admits $K$-rational points of period 2 and 3, using the fact that the only quadratic points on the curve $Y'$ defined by \eqref{eq:2and3curve} have $t \in \bbQ$ by Lemma~\ref{lem:1and3/2and3curves}.
\end{proof}

For each of the three different cycle structures described in Theorem~\ref{prop:1and2etc}, we can find infinitely many quadratic pairs $(K,c)$ for which $G(f_c,K)$ contains the given structure. For the remainder of this section we show that for every other possible cycle structure with cycles of length at most 4, there are at most finitely many quadratic pairs $(K,c)$ for which $G(f_c,K)$ contains the given structure.

\subsubsection[Periods 1 and 4]{Periods 1 and 4}\label{subsub:1and4}

\begin{figure}[h]
\centering
	\begin{overpic}[scale=.6]{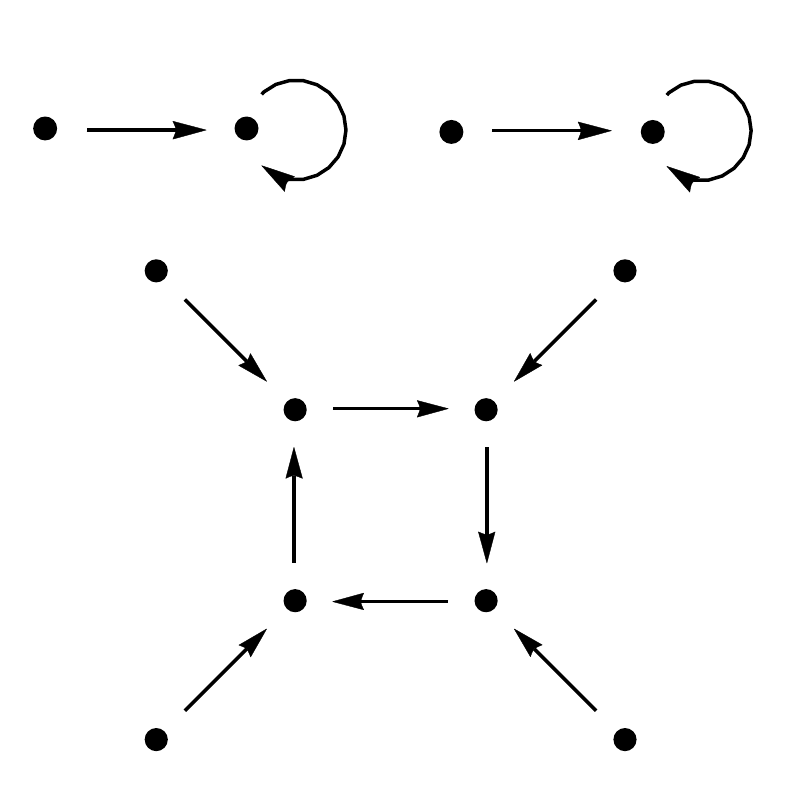}
		\put(35,119){$\alpha$}
		\put(105,119){$\alpha'$}
		\put(78,72){$\beta$}
	\end{overpic}
	\caption{The graph $G_0$ generated by fixed points $\alpha$ and $\alpha'$ and a point $\beta$ of period 4}
	\label{fig:1and4}
\end{figure}

We now discuss the situation in which a quadratic pair $(K,c)$ simultaneously admits $K$-rational points of periods 1 and 4. Unfortunately, we have thus far been unable to fully determine all such occurrences. The problem is that the general bound we obtain from Theorem~\ref{thm:stoll} for the number of rational points on a certain hyperelliptic curve does not appear to be sharp, so different techniques will be required in order to complete the analysis in this case. We are able, however, to make the following statement.

\begin{thm}\label{thm:1and4}
There is at most one quadratic pair $(K,c)$ that simultaneously admits points of period 1 and period 4. (In other words, there is at most one quadratic pair $(K,c)$ for which $G(f_c,K)$ contains the graph $G_0$ shown in Figure~\ref{fig:1and4}.) Moreover, such a pair must have $c \in \bbQ$.
\end{thm}

We begin by giving a more computationally useful model for the dynamical modular curve
	\[ U_1(1,4) = \{(\alpha,\beta,c) \in \bbA^3 : \mbox{ $\alpha$ is a fixed point and $\beta$ is a point of period 4 for $f_c$} \}. \]

\begin{prop}\label{prop:1and4curve}
Let $Y$ be the affine curve of genus 9 defined by the equation
	\begin{equation}\label{eq:1and4curve}
		\begin{cases}
			v^2 &= -u(u^2 + 1)(u^2 - 2u - 1)\\
			w^2 &= -u(u^6 - 4u^5 - 3u^4 - 8u^3 + 3u^2 - 4u - 1),
		\end{cases}
	\end{equation}
and let $U$ be the open subset of $Y$ defined by
	\begin{equation}\label{eq:1and4cusps}
		(u-1)(u+1)v \ne 0.
	\end{equation}
Consider the morphism $\Phi: U \to \bbA^3$, $(u,v,w) \mapsto (\alpha,\beta,c)$, given by
	\[ \alpha = \frac{1}{2} + \frac{w}{2u(u-1)(u+1)} , \ \beta = \frac{u-1}{2(u+1)} + \frac{v}{2u(u-1)} , \ c = \frac{(u^2 - 4u - 1)(u^4 + u^3 + 2u^2 - u + 1)}{4u(u+1)^2(u-1)^2} .\]
Then $\Phi$ maps $U$ isomorphically onto $U_1(1,4)$, with the inverse map given by
	\begin{equation}\label{eq:1and4inverse}
		u = -\frac{\beta + f_c^2(\beta) + 1}{\beta + f_c^2(\beta) - 1} , \ v = \frac{u(u-1)(2\beta u + 2\beta - u + 1)}{u+1} , \ w = u(u-1)(u+1)(2\alpha - 1).
	\end{equation}
\end{prop}

\begin{proof}
Since $v \ne 0$ implies $u \ne 0$, the condition $(u-1)(u+1)v \ne 0$ implies that $\Phi$ is well-defined. One can verify that $f_c(\alpha) = \alpha$, so $\alpha$ is a fixed point for $f_c$, and that $f_c^4(\beta) = \beta$. Moreover,
	\[ \beta - f_c^2(\beta) = \frac{v}{u(u-1)}, \]
which is nonzero by hypothesis, so $\beta$ is a point of exact period 4 for $f_c$. Therefore $\Phi$ maps $Y$ into $U_1(1,4)$. A quick computation shows that \eqref{eq:1and4inverse} serves as a left inverse for $\Phi$, so $\Phi$ is injective.

It now remains to show that $\Phi$ is surjective. Suppose $(\alpha,\beta,c) \in U_1(1,4)$. By Proposition~\ref{prop:4curve}, there exist $u,v \in K$ such that
	\[
		\begin{gathered}
			v^2 = -u(u^2 + 1)(u^2 - 2u - 1) , \\
			\beta = \frac{u-1}{2(u+1)} + \frac{v}{2u(u-1)} , \
			c = \frac{(u^2 - 4u - 1)(u^4 + u^3 + 2u^2 - u + 1)}{4u(u+1)^2(u-1)^2}.
		\end{gathered}
	\]
Also, by Proposition~\ref{prop:1or2or3}, there exists $r \in K$ such that
	\[  \alpha = 1/2 + r , \ c = 1/4 - r^2. \]
Equating these two expressions for $c$ yields
	\[ r^2 = -\frac{u^6-4 u^5-3 u^4-8 u^3+3 u^2-4 u-1}{4 u(u-1)^2 (u+1)^2}. \]
Setting $w := 2u(u-1)(u+1)r$ yields the second equation in \eqref{eq:1and4curve} and allows us to rewrite $\alpha$ as
	\[ \alpha = \frac{1}{2} + \frac{w}{2u(u-1)(u+1)}, \]
completing the proof.
\end{proof}

We now offer a bound on the number of quadratic points on the affine curve $Y$ defined by \eqref{eq:1and4curve}.

\begin{thm}\label{thm:1and4quad_pts}
Let $Y$ be the genus 9 affine curve defined by \eqref{eq:1and4curve}. Then
	\begin{align*}
	Y(\bbQ,2) &= Y(\bbQ) \cup \{(t,0,\pm(2t + 2)) : t^2 + 1 = 0 \} \cup \calS \\
		&= \{ (0,0,0), (\pm 1,\pm 2,\pm 4)\} \cup \{(t,0,\pm(2t + 2)) : t^2 + 1 = 0 \} \cup \calS,
	\end{align*}
where $\#\calS \in \{0,8\}$. Moreover, if $(u,v,w) \in \calS$, then $u \in \bbQ$, and
	\begin{equation}\label{eq:1and4auto}
	\calS = \left\{(u,\pm v, \pm w) , \left( -\frac{1}{u}, \pm \frac{v}{u^3}, \pm \frac{w}{u^4} \right) \right\}.
	\end{equation}
\end{thm}

Assuming Theorem~\ref{thm:1and4quad_pts} for the moment, we now prove Theorem~\ref{thm:1and4}.

\begin{proof}[Proof of Theorem~\ref{thm:1and4}]
Suppose $(K,c)$ is a quadratic pair that simultaneously admits points $\alpha$ and $\beta$ of periods 1 and 4, respectively. By Proposition~\ref{prop:1and4curve}, there exists a $K$-rational point $(u,v,w)$ on the curve $Y$ defined by \eqref{eq:1and4curve} such that
	\[ \alpha = \frac{1}{2} + \frac{w}{2u(u-1)(u+1)} , \ \beta = \frac{u-1}{2(u+1)} + \frac{v}{2u(u-1)} , \ c = \frac{(u^2 - 4u - 1)(u^4 + u^3 + 2u^2 - u + 1)}{4u(u+1)^2(u-1)^2} .\]
The point $(u,v,w)$ cannot be one of the known points listed in Theorem~\ref{thm:1and4quad_pts}, since all such points satisfy $(u-1)(u+1)v = 0$ and therefore lie outside the open subset $U \subset Y$ defined by \eqref{eq:1and4cusps}. Hence $(u,v,w)$ is a $K$-rational point in $\calS$. This implies that $\#\calS = 8$, and from \eqref{eq:1and4auto} all eight elements of $\calS$ lie in $Y(K)$. Since the expression for $c$ in terms of $u$ is invariant under $u \mapsto -1/u$, all eight points yield the same value of $c$, and therefore the same quadratic pair $(K,c)$. Finally, since $(u,v,w) \in \calS$, we must have $u \in \bbQ$, and therefore $c \in \bbQ$ as well.
\end{proof}

The bound in Theorem~\ref{thm:1and4quad_pts} for the number of quadratic points on $Y$ relies on the following lemma, which describes the set of \emph{rational} points on two different quotients of $Y$.

\begin{lem}\label{lem:1and4quotients} \mbox{}
\begin{enumerate}
\item Let $C_3$ be the genus 3 hyperelliptic curve defined by
	\begin{equation}\label{eq:1and4C3}
		w^2 = -u(u^6 - 4u^5 - 3u^4 - 8u^3 + 3u^2 - 4u - 1).
	\end{equation}
Then
	\[ C_3(\bbQ) = \{ (0,0) , (\pm 1, \pm 4), \infty \}. \]
\item Let $C_4$ be the genus 4 hyperelliptic curve defined by
	\begin{equation}\label{eq:1and4C4}
	y^2 = (u^2 + 1)(u^2 - 2u - 1)(u^6 - 4u^5 - 3u^4 - 8u^3 + 3u^2 - 4u - 1).
	\end{equation}
Then
	\[ C_4(\bbQ) = \{(0,\pm 1), (\pm 1, \pm 8), \infty^{\pm} \} \cup \calS_4, \]
where $\#\calS_4 \in \{0,4\}$. Moreover, if $\calS_4$ is nonempty and $(u,y) \in \calS_4$, then
	\[ \calS_4 = \left\{(u,\pm y),(-1/u,\pm y/u^5)\right\}. \]
\end{enumerate}
\end{lem}

\begin{proof}
We first consider the curve $C_3$. Let $J_3$ denote the Jacobian of $C_3$. A two-descent using Magma's \texttt{RankBound} function gives an upper bound of 1 for $\rk J_3(\bbQ)$. (The point $[(1,4) - \infty] \in J_3(\bbQ)$ has infinite order, so $J_3(\bbQ)$ has rank equal to one.) Theorem~\ref{thm:stoll} may therefore be used to bound the number of rational points on $C_3$.

The prime $p = 7$ is a prime of good reduction for $C_3$, and a quick computation shows that $C_3(\bbF_7)$ consists of six points. Since $\rk J_3(\bbQ) = 1$, applying Theorem~\ref{thm:stoll} yields the bound
	\[ \#C_3(\bbQ) \le \#C_3(\bbF_7) + 2 \cdot \rk J_3(\bbQ) = 8. \]
We have already listed six rational points, so there can be at most two additional rational points on $C_3$.

Let $\iota$ be the hyperelliptic involution $(u,w) \mapsto (u,-w)$, and let $\sigma$ be the automorphism given by
	\[
		\sigma(u,w) = \left(-\frac{1}{u},\frac{w}{u^4}\right),
	\]
which also has order 2. We claim that if $(u,w) \not \in \{(0,0),\infty\}$ is a rational point on $C_3$, then the set
	\[ \{(u,w) , \iota(u,w), \sigma(u,w), \sigma \iota (u,w) \} \]
contains exactly four elements. This will complete the proof of (A), since we have already said that there can be at most \emph{two} rational points in addition to the six we have already found.

First, we note that the only points fixed by $\iota$ are the Weierstrass points, and the only rational Weierstrass points on $C_3$ are $(0,0)$ and $\infty$, since the sextic polynomial appearing in \eqref{eq:1and4C3} is irreducible over $\bbQ$. Now suppose $(u,w)$ is fixed by one of $\sigma$ and $\sigma \iota$. In this case, we necessarily have $u = -1/u$, which implies that $(u,w)$ is not rational. Since just one additional rational point would actually yield four such points, we conclude that $C_3(\bbQ)$ consists only of the six points listed in the lemma.

We prove (B) similarly. Let $J_4$ denote the Jacobian of $C_4$. Magma's \texttt{RankBound} function gives a bound of 2 for $\rk J_4(\bbQ)$, so we may again apply Theorem~\ref{thm:stoll} to bound the number of rational points on $C_4$. The prime $p = 5$ is a prime of good reduction for $C_4$, and we find that $\#C_4(\bbF_5) = 10$, so
	\[ \#C_4(\bbQ) \le \#C_4(\bbF_5) + 2 \cdot \rk J_4(\bbQ) + \left \lfloor \frac{2 \cdot \rk J_4(\bbQ)}{3} \right\rfloor \le 15. \]
We have already listed eight rational points, so there can be at most seven additional rational points.

Suppose now that $C_4(\bbQ)$ contains an additional rational point $(u,y)$. Consider the automorphism $\sigma$ given by
	\[ \sigma(u,y) = \left(-\frac{1}{u},\frac{y}{u^5}\right), \]
which has order four:
	\[ (u,y) \mapsto \left(-\frac{1}{u},\frac{y}{u^5}\right) \mapsto (u,-y) \mapsto \left(-\frac{1}{u},-\frac{y}{u^5}\right) \mapsto (u,y). \]
The points $\sigma^k(u,y)$ are distinct for $k \in \{0,1,2,3\}$ unless $(u,y)$ is fixed by $\sigma^k$ for some $k \in \{1,2,3\}$, in which case $(u,y)$ must be fixed by the hyperelliptic involution $\iota = \sigma^2$. However, since the quadratic and sextic polynomials appearing in \eqref{eq:1and4C4} are irreducible over $\bbQ$, $C_4$ has no finite rational Weierstrass points. Therefore $\{\sigma^k(u,y) : k \in \{0,1,2,3\} \}$ consists of four distinct points.

In fact, this argument shows that $\#C_4(\bbQ)$ is divisible by four. Since we have already found eight points and have obtained an upper bound of 15 for $\#C_4(\bbQ)$, we have that $\#C_4(\bbQ) \in \{8,12\}$ and therefore $\#\calS_4 \in \{0,4\}$, as claimed. If $\calS_4$ is nonempty and $(u,y) \in \calS_4$, then the four rational points in $\calS_4$ are the orbit of $(u,y)$ under $\sigma$.
\end{proof}

We are now ready to prove Theorem~\ref{thm:1and4quad_pts}.

\begin{proof}[Proof of Theorem~\ref{thm:1and4quad_pts}]
Since the only rational points on the affine curve defined by
	\[
		v^2 = -u(u^2 + 1)(u^2 - 2u - 1)
	\]
satisfy $u \in \{0,\pm 1\}$, $Y(\bbQ)$ consists only of the nine rational points listed in the statement of the theorem; that is,
	\[ Y(\bbQ) = \{ (0,0,0), (\pm 1,\pm 2,\pm 4)\}. \]
Now suppose $(u,v,w)$ is a quadratic point on $Y$ different from those listed in the theorem (i.e., an element of $\calS$). We first show that $vw \ne 0$. Indeed, suppose $v = 0$. Then either $u = 0$, in which case $w = 0$ and $(u,v,w)$ is a rational point; $u^2 + 1 = 0$, in which case we have one of the known quadratic points on $Y$; or $u^2 - 2u - 1 = 0$, in which case one can check that $[\bbQ(w) : \bbQ] > 2$. A similar argument shows that $w \ne 0$.

Since $vw \ne 0$, and since $u^2 + 1 \ne 0$, the following eight points must be distinct quadratic points on $Y$, all defined over the same quadratic extension of $\bbQ$:
\[	\left\{(u,\pm v, \pm w) , \left( -\frac{1}{u}, \pm \frac{v}{u^3}, \pm \frac{w}{u^4} \right) \right\}. \]
Therefore, the addition of a single quadratic point implies the existence of at least eight additional quadratic points. We now show that there can be \emph{at most} eight additional quadratic points.

If $(u,v,w) \in \calS$, it follows from the proof of Proposition~\ref{prop:4cycle_rat_c} that either $u \in \bbQ$ or $(u^2 + 1)(u^2 - 2u - 1) = 0$. However, we have already ruled out the possibility that $(u^2+1)(u^2-2u-1)= 0$, so we must have $u \in \bbQ$. Given that $u \in \bbQ$, we have already seen that $v \in \bbQ$ if and only if $u \in \{0,\pm1\}$, and the same is true for $w$ by Lemma~\ref{lem:1and4quotients}(A). Since $(u,v,w)$ is not a rational point, we must have $u \in \bbQ \setminus \{0,\pm 1\}$, $v,w \not \in \bbQ$, and $v^2,w^2 \in \bbQ$. We may therefore write $v = v'\sqrt{d}$, $w = w'\sqrt{d}$ for some $v',w' \in \bbQ$ and $d \ne 1$ a squarefree integer, which means $vw \in \bbQ$. Since
	\[ (vw)^2 = u^2(u^2 + 1)(u^2 - 2u - 1)(u^6 - 4u^5 - 3u^4 - 8u^3 + 3u^2 - 4u - 1), \]
setting $y = vw/u$ gives a (finite) rational point on the curve $C_4$ from Lemma~\ref{lem:1and4quotients}(B). Since we have $u \not \in \{0,\pm 1\}$, the point $(u,y)$ must be in the set $\calS_4$ from Lemma~\ref{lem:1and4quotients}(B). Because the map
	\begin{align*}
		Y &\to C_4\\
		(u,v,w) &\mapsto \left(u,\frac{vw}{u}\right)
	\end{align*}
has degree two, and since $\#\calS_4 \le 4$, this allows at most eight additional points $(u,v,w)$ on $Y$.
\end{proof}

\subsubsection[Periods 2 and 4]{Periods 2 and 4}
For the question of whether a quadratic pair $(K,c)$ may admit $K$-rational points of periods 2 and 4, we refer to the following result from \cite[\textsection 3.17]{doyle/faber/krumm:2014}:
\begin{thm}\label{thm:2and4}
In addition to the known pair $(\bbQ(\sqrt{-15}), -31/48)$, there is at most one quadratic pair $(K,c)$ that simultaneously admits points of period 2 and period 4. Moreover, such a pair must have $c \in \bbQ$.
\end{thm}

\subsubsection[Periods 3 and 4]{Periods 3 and 4}
In this section, we prove the following result:

\begin{thm}\label{thm:3and4}
Let $K$ be a quadratic field, and let $c \in K$. The map $f_c$ cannot simultaneously admit $K$-rational points of periods 3 and 4.
\end{thm}

The appropriate dynamical modular curve to study in this case is
	\[ U_1(3,4) = \{(\alpha,\beta,c) \in \bbA^3 : \mbox{ $\alpha$ has period 3 and $\beta$ has period 4 for $f_c$} \}. \]

\begin{prop}\label{prop:3and4}
Let $Y$ be the affine curve of genus 49 defined by
	\begin{equation}\label{eq:3and4curve}
		\left\{
			\begin{aligned}
			-(t^6 + 2t^5 + 4t^4 + 8t^3 + 9t^2 &+ 4t + 1)u(u+1)^2(u-1)^2\\
				&= (u^2 - 4u - 1)(u^4 + u^3 + 2u^2 - u + 1)t^2(t + 1)^2\\
			v^2 &= -u(u^2 + 1)(u^2 - 2u - 1),
			\end{aligned}
			\right.
	\end{equation}
and let $U$ be the open subset of $Y$ defined by
	\begin{equation}\label{eq:3and4cusps}
		t(t+1)(t^2 + t + 1)(u-1)(u+1)v \ne 0.
	\end{equation}	
Consider the morphism $\Phi: U \to \bbA^3$, $(t,u,v) \mapsto (\alpha,\beta,c)$, given by
	\begin{equation}\label{eq:3and4param}
	\begin{gathered}
	\alpha = \frac{t^3+2 t^2+t+1}{2 t (t+1)} , \ \beta = 
	\frac{u-1}{2(u+1)} + \frac{v}{2u(u-1)} , \
	c = -\frac{t^6+2 t^5+4 t^4+8 t^3+9 t^2+4 t+1}{4 t^2 (t+1)^2}.
	\end{gathered}
	\end{equation}
Then $\Phi$ maps $U$ isomorphically onto $U_1(3,4)$, with the inverse map given by
	\begin{equation}\label{eq:3and4inverse}
	t = \alpha + f_c(\alpha) , \ u = -\frac{f_c^2(\beta) + \beta + 1}{f_c^2(\beta) + \beta - 1} , \ v = \frac{u(u-1)(2\beta u + 2\beta - u + 1)}{u+1}.
	\end{equation}
\end{prop}

\begin{proof}
This follows from Propositions \ref{prop:1or2or3}(C) and \ref{prop:4curve}, noting that the first equation in \eqref{eq:3and4curve} comes from equating the two expressions for $c$ given in those two propositions (and clearing denominators), while the second equation comes directly from \eqref{eq:4curve}.
\end{proof}

In order to find all quadratic points on $U_1(3,4)$, it will be beneficial for us to find the \emph{rational} points on the quotient $U_0(3,4)$. Recall from \textsection\ref{sub:per_dyn} that there are automorphisms $\sigma_3$ and $\sigma_4$ on $U_1(3,4)$ given by
\[ \sigma_3(\alpha,\beta,c) = (f_c(\alpha),\beta,c) \ \mbox{ and } \ \sigma_4(\alpha,\beta,c) = (\alpha,f_c(\beta),c). \]
The curve $U_0(3,4)$ is the quotient of $U_1(3,4)$ by the group of automorphisms generated by $\sigma_3$ and $\sigma_4$. Abusing notation, we will also denote by $\sigma_3$ and $\sigma_4$ the corresponding automorphisms on the curve $Y$ defined in \eqref{eq:3and4curve}, which is birational to $U_1(3,4)$, and we let $Y'$ denote the quotient of $Y$ by $\langle \sigma_3, \sigma_4 \rangle$ (hence $Y'$ is birational to $U_0(3,4)$). By using the relations given in \eqref{eq:3and4param} and \eqref{eq:3and4inverse} (cf. Remarks~\ref{rem:3aut} and \ref{rem:4aut}), we see that $\sigma_3$ and $\sigma_4$ act on $Y$ as follows:
	\begin{align*}
		\sigma_3(t,u,v) &= \left(-\frac{t+1}{t},u,v\right)\\
		\sigma_4(t,u,v) &= \left(t,-\frac{1}{u},\frac{v}{u^3} \right).
	\end{align*}
	
Since $\sigma_3$ and $\sigma_4$ have orders 3 and 4, respectively, the group $\langle \sigma_3,\sigma_4 \rangle$ is cyclic of order 12, generated by the automorphism $\sigma := \sigma_3 \circ \sigma_4$ defined by
	\[ \sigma(t,u,v) = \left(-\frac{t+1}{t},-\frac{1}{u},\frac{v}{u^3}\right). \]
Hence $Y'$ is the quotient of $Y$ by $\sigma$. It will be simpler, however, to first consider the quotient $Y''$ of $Y$ by the automorphism $\sigma^6$, and then take the quotient of $Y''$ by the automorphism $\tau \in \Aut(Y'')$ that descends from $\sigma \in \Aut(Y)$. Since
	\[ \sigma^6(t,u,v) = (t,u,-v), \]
the curve $Y''$ is simply given by the first equation from \eqref{eq:3and4curve}, which we rewrite for convenience as
	\begin{equation}\label{eq:3and4quotient1}
	-\frac{t^6 + 2t^5 + 4t^4 + 8t^3 + 9t^2 + 4t + 1}{4t^2(t + 1)^2} = \frac{(u^2 - 4u - 1)(u^4 + u^3 + 2u^2 - u + 1)}{4u(u+1)^2(u-1)^2}.
	\end{equation}
Note that this is precisely the equation obtained by equating the expressions for $c$ found in Theorems~\ref{thm:unique3cycle} and \ref{thm:unique4cycle}. We now seek the quotient of $Y''$ by the automorphism $\tau$, which is given by
	\[ \tau(t,u) := \left(-\frac{t+1}{t},-\frac{1}{u}\right). \]
Consider the following invariants of $\tau$:
	\begin{equation}\label{eq:X0(3,4)TU}
	T := t - \frac{t+1}{t} - \frac{1}{t+1} = \frac{t^3 - 3t - 1}{t(t+1)} , \ U := u - \frac{1}{u} = \frac{(u-1)(u+1)}{u}.
	\end{equation}
To find the quotient of $Y''$ by $\tau$, we rewrite \eqref{eq:3and4quotient1} in terms of $T$ and $U$, which yields
	\[ -(T^2 + 2T + 8) = \frac{(U^2 + U + 4)(U - 4)}{U^2}. \]
Clearing denominators, we see that $Y'$ has the model
	\begin{equation}\label{eq:X0(3,4)}
	-(T^2 + 2T + 8)U^2 = (U^2 + U + 4)(U - 4).
	\end{equation}
This curve has genus one and is birational to $\Xell_1(11)$ (see \cite{reichert:1984}), which is labeled 11A3 in \cite{cremona:1997} and given by
	\[ y^2 + y = x^3 - x^2, \]
via the inverse birational maps
	\begin{align}
	\label{eq:X0(3,4)birational} (x,y) &\mapsto \left(-\frac{x - 2y - 1}{x},-4x\right)\\
	\notag (T,U) &\mapsto \left(-\frac{U}{4},-\frac{TU + U + 4}{8}\right).
	\end{align}
We have just shown the following:
\begin{prop}\label{prop:X0(3,4)}
The curve $X_0(3,4)$ is isomorphic to the elliptic curve $\Xell_1(11)$.
\end{prop}

In order to determine all quadratic points on the curve $Y$, we require knowing all rational points on its quotient $Y'$:

\begin{lem}\label{lem:X0(3,4)}
Let $Y'$ be the affine curve of genus one defined by \eqref{eq:X0(3,4)}. Then
	\[
		Y'(\bbQ) = \{ (0,-4), (-2,-4) \}.
	\]
\end{lem}

\begin{proof}
Let $\tilde{Y'}$ denote the projective closure of $Y'$; that is, $\tilde{Y'}$ is the curve in $\bbP^2$ given by the homogeneous equation
	\begin{equation}\label{eq:X0(3,4)proj_model}
	-(T^2 + 2TZ + 8V^2)U^2 = (U^2 + UV + 4V^2)(U - 4V)V.
	\end{equation}
A quick search reveals the four rational points
	\[ [T:U:V] \in \{ [0:-4:1],[-2:-4:1],[1:0:0],[0:1:0] \}. \]
We claim that this is the complete list of rational points on $\tilde{Y'}$, from which the lemma follows.

Let $\phi : \Xell_1(11) \to Y'$ be the rational map given by \eqref{eq:X0(3,4)birational}. In projective coordinates, we have
	\[ \phi\left([x:y:z]\right) = [(x - 2y - z)z : 4x^2 : -xz]. \]
It is well known that the curve $\Xell_1(11)$ has only five rational points, so
	\[ \{ [0:0:1] , [0:-1:1], [1:-1:1], [1:0:1], [0:1:0] \} \]
is the full set of rational points. The rational points map to $\tilde{Y'}$ as follows:
	\begin{align*}
	[0:0:1],[0:-1:1] &\mapsto [1:0:0]\\
	[1:-1:1] &\mapsto [-2:-4:1]\\
	[1:0:1] &\mapsto [0:-4:1]\\
	[0:1:0] &\mapsto [0:1:0].
	\end{align*}
The point $[1:0:0]$ is the only singular point on $\tilde{Y'}$, so if $\tilde{Y'}$ had another (necessarily nonsingular) rational point $P$, then $P$ would be the image under $\phi$ of a rational point on $\Xell_1(11)$. Since we have already accounted for the images of all five rational points on $\Xell_1(11)$, $\tilde{Y'}$ cannot have any additional rational points.
\end{proof}

We now show that the curve $Y$ defined by \eqref{eq:3and4curve} has no quadratic points.

\begin{thm}\label{thm:3and4curve_pts}
Let $Y$ be the genus 49 affine curve defined by \eqref{eq:3and4curve}. Then 
\[ Y(\bbQ,2) = Y(\bbQ) = \{(0,0,0),(0,\pm 1,\pm 2),(-1,0,0),(-1,\pm 1, \pm 2)\}. \]
\end{thm}

Assuming this theorem for the moment, we complete the proof of Theorem~\ref{thm:3and4}. Theorem~\ref{thm:3and4curve_pts} implies that all points of degree at most 2 on $Y$ fail to satisfy \eqref{eq:3and4cusps}. It follows that $U_1(3,4)(\bbQ,2)$ is empty. Therefore, there are no quadratic pairs $(K,c)$ that simultaneously admit points of periods 3 and 4. \qed

\begin{proof}[Proof of Theorem~\ref{thm:3and4curve_pts}]
The only rational solutions to the equation
	\[ v^2 = -u(u^2 + 1)(u^2 - 2u - 1) \]
satisfy $u \in \{0,\pm 1\}$, so the rational points listed above must be all rational points on $Y$. The rational points all fail to satisfy condition \eqref{eq:3and4cusps}, and one can verify that all \emph{other} points failing condition \eqref{eq:3and4cusps} have degree strictly greater than two over $\bbQ$. Thus any quadratic point $(t,u,v)$ on $Y$ actually lies on the open subset $U \subset Y$ defined in Proposition~\ref{prop:3and4}, hence $u \in \bbQ$. It follows that if we let $K$ be the field of definition of $(t,u,v)$ and set
	\[ c = -\frac{t^6 + 2t^5 + 4t^4 + 8t^3 + 9t^2 + 4t + 1}{4t^2(t + 1)^2} = \frac{(u^2 - 4u - 1)(u^4 + u^3 + 2u^2 - u + 1)}{4u(u+1)^2(u-1)^2} \in K, \]
then the map $f_c$ admits $K$-rational points $\alpha$ and $\beta$ of periods 3 and 4, respectively.

Since $\beta$ has period 4 for $f_c$, Proposition~\ref{prop:4cycle_rat_c} says that $c \in \bbQ$ and that the 4-cycle containing $\beta$ is rational. Since $c \in \bbQ$, it follows from Corollary~\ref{cor:3cyclesQ} that $\alpha$ --- and therefore the 3-cycle containing $\alpha$ --- is rational as well. Therefore the image $(T,U)$ of $(t,u,v)$ under the quotient map $Y \to Y'$ is a rational point. Since we are assuming that the point $(t,u,v)$ satisfies the open condition \eqref{eq:3and4cusps}, then in particular $t(t+1)u \ne 0$, so $T$ and $U$ are finite, hence $U = -4$ by Lemma~\ref{lem:X0(3,4)}. Therefore $u - 1/u = -4$, which implies that $u \not \in \bbQ$, a contradiction.
\end{proof}

\subsubsection[Periods 1, 2, and 3]{Periods 1, 2, and 3} For points of period 1, 2, and 3, we have the following:

\begin{thm}\label{thm:1and2and3}
Let $K$ be a quadratic field, and let $c \in K$. The map $f_c$ cannot simultaneously admit $K$-rational points of periods 1, 2, and 3.
\end{thm}

To prove Theorem~\ref{thm:1and2and3}, we must find all rational and quadratic points on the curve
	\[ U_1(1,2,3) = \{(\alpha_1,\alpha_2,\alpha_3,c) \in \bbA^4 : \mbox{ $\alpha_i$ has period $i$ for $f_c$ for each $i \in \{1,2,3\}$} \}. \]

The following is an immediate consequence of Proposition~\ref{prop:1and2etc}.

\begin{prop}
Let $Y$ be the affine curve of genus 9 defined by the equation
	\begin{equation}\label{eq:1and2and3curve}
		\begin{cases}
			y^2 &= t^6 + 2t^5 + 5t^4 + 10t^3 + 10t^2 + 4t + 1\\
			z^2 &= t^6 + 2t^5 + t^4 + 2t^3 + 6t^2 + 4t + 1,
		\end{cases}
	\end{equation}
and let $U$ be the open subset of $Y$ defined by
	\begin{equation}\label{eq:1and2and3cusps}
		t(t+1)(t^2 + t + 1)z \ne 0.
	\end{equation}
Consider the morphism $\Phi: U \to \bbA^4$, $(t,y,z) \mapsto (\alpha_1,\alpha_2,\alpha_3,c)$, given by
	\begin{gather*}
	\alpha_1 = \frac{t^2 + t + y}{2t(t+1)} , \ \alpha_2 = -\frac{t^2 + t - z}{2t(t+1)} , \ \alpha_3 = \frac{t^3 + 2t^2 + t + 1}{2t(t+1)} , \\
	c = -\frac{t^6+2 t^5+4 t^4+8 t^3+9 t^2+4 t+1}{4 t^2 (t+1)^2}.
	\end{gather*}
Then $\Phi$ maps $U$ isomorphically onto $U_1(1,2,3)$, with the inverse map given by
	\begin{equation}\label{eq:1and2and3inverse}
		t = \alpha_3 + f_c(\alpha_3) , \ y = (2\alpha_1 - 1)t(t+1) , \ z = (2\alpha_2 + 1)t(t+1).
	\end{equation}
\end{prop}

Theorem~\ref{thm:1and2and3} follows immediately from the following theorem, which says that $Y(\bbQ,2)$ does not intersect the open set $U$ defined by \eqref{eq:1and2and3cusps}.

\begin{thm}\label{thm:1and2and3curve_pts}
Let $Y$ be the genus 9 affine curve defined by \eqref{eq:1and2and3curve}.
Then 
\[ Y(\bbQ,2) = Y(\bbQ) = \{(0, \pm 1, \pm 1) , (-1, \pm 1, \pm 1)\}. \]
\end{thm}

\begin{proof} 
Let
	\begin{align*}
	f(t) &:= t^6 + 2t^5 + 5t^4 + 10t^3 + 10t^2 + 4t + 1;\\
	g(t) &:= t^6 + 2t^5 + t^4 + 2t^3 + 6t^2 + 4t + 1.
	\end{align*}
By Lemma~\ref{lem:1and3/2and3curves}, the only rational points on $y^2 = f(t)$ and $z^2 = g(t)$ are those points with $t \in \{-1,0\}$, so the eight points listed in the theorem are the only rational points. It now remains to show that the curve $Y$ has no quadratic points.

Suppose to the contrary that $(t,y,z)$ is a quadratic point on $X$. From Lemma~\ref{lem:1and3/2and3curves}, we know that $t \in \bbQ$; furthermore, if $t \not \in \{0,-1\}$, then neither $y$ nor $z$ is in $\bbQ$. Thus $y$ and $z$ generate the same quadratic extension $K = \bbQ(\sqrt{d})$, with $d \ne 1$ a squarefree integer. Since $y^2 , z^2 \in \bbQ$, we may write $y = u\sqrt{d}$, $z = v\sqrt{d}$ for some $u,v \in \bbQ$. This allows us to rewrite \eqref{eq:1and2and3curve} as
	\begin{equation}\label{eq:1and2and3twists}
		\begin{cases}
			du^2 = f(t)\\
			dv^2 = g(t),
		\end{cases}
	\end{equation}
where $t,u,v \in \bbQ$. We first show that if $p \in \bbZ$ is a prime dividing $d$, then $t$, $u$, and $v$ are integral at $p$. Indeed, if $\ord_p(t) < 0$, then $\ord_p(f(t)) = \ord_p(t^6)$ is even, but $\ord_p(du^2) = 2\ord_p(u) + 1$ is odd. Therefore $\ord_p(t) \ge 0$, and it can easily be seen now that $\ord_p(u),\ord_p(v) \ge 0$ as well. We may therefore reduce \eqref{eq:1and2and3twists} modulo $p$, yielding
	\[ f(t) \equiv g(t) \equiv 0 \mod p. \]
Hence $p$ divides the resultant $\Res(f,g) = 2^{12}$, so $p = 2$, and therefore $d \in \{-1, \pm 2\}$. However, $f(t)$ is strictly positive for all $t \in \bbR$, so we must have $d = 2$.

Let $C_f^{(2)}$ and $C_g^{(2)}$ be the curves defined by $2u^2 = f(t)$ and $2v^2 = g(t)$, respectively, and let $J_f^{(2)}$ and $J_g^{(2)}$ be their Jacobians. Each of $J_f^{(2)}$ and $J_g^{(2)}$ has a trivial Mordell-Weil group over $\bbQ$, so in particular neither curve has rational points. (We are also using the fact that, since $f$ and $g$ are irreducible over $\bbQ$ and have even degree, the curves $C_f^{(2)}$ and $C_g^{(2)}$ have no rational Weierstrass points.) We conclude that $Y$ has no quadratic points.
\end{proof}

\section{Larger strongly admissible graphs}\label{sec:preper_pts}

We now consider {\it strongly} admissible graphs that are larger than those generated solely by periodic points. We begin with the observation that a point $\alpha$ has type $m_1$ for $f_c$ if and only if $-\alpha$ is a nonzero point of period $m$. Indeed, if $\alpha$ is a point of type $m_1$, then $\beta := f_c(\alpha)$ is a point of period $m$. Every periodic point has precisely one periodic preimage, so there exists a periodic point $\alpha'$ for which $f_c(\alpha') = \beta = f_c(\alpha)$, hence $\alpha' = -\alpha$. Conversely, if $\alpha'$ is a nonzero point of period $m$, then $\alpha := -\alpha' \ne \alpha'$ satisfies $f_c(\alpha) = f_c(\alpha')$, which is also a point of period $m$. Since $f_c(\alpha')$ has precisely one periodic preimage, $\alpha$ cannot itself be periodic, so $\alpha$ is a point of type $m_1$.

From this discussion, it follows that if $f_c$ admits a nonzero $K$-rational point of period $m$, then $f_c$ automatically admits a $K$-rational point of type $m_1$. We therefore focus on preperiodic points with preperiod at least two.

The strongly admissible graphs studied in this section are grouped according to the their cycle structures. For example, the first subsection discusses those strongly admissible graphs whose only periodic points are fixed points; i.e., those graphs with cycle structure (1,1). The cycle structures appearing in this section are those with cycles of length at most 4 that were not ruled out by the results of \textsection \ref{sec:periodic}. In other words, we will only consider the following cycle structures in this section:
	\[ (1,1), (2), (3), (4), (1,1,2), (1,1,3), (2,3). \]
(We are using the fact, following from Definition~\ref{defn:admissible} and mentioned in Remark~\ref{rem:redundant}, that a strongly admissible graph cannot have just one fixed point.) We have not listed the structures $(1,1,4)$ and $(2,4)$, since we have already shown that we have found all quadratic pairs $(K,c)$ simultaneously admitting points of period 1 and 4 (resp., period 2 and 4), with at most one exception in each case.

In \textsection \ref{sec:main_thm}, we summarize the results of this and the previous section to complete the proof of Theorem~\ref{thm:main_thm}. For the remainder of this section, we will refer to the graphs appearing in Appendix~\ref{app:graphs} by the labels given there.

\subsection{Period 1}

\begin{figure}[h]
\centering
	\begin{overpic}[scale=.6]{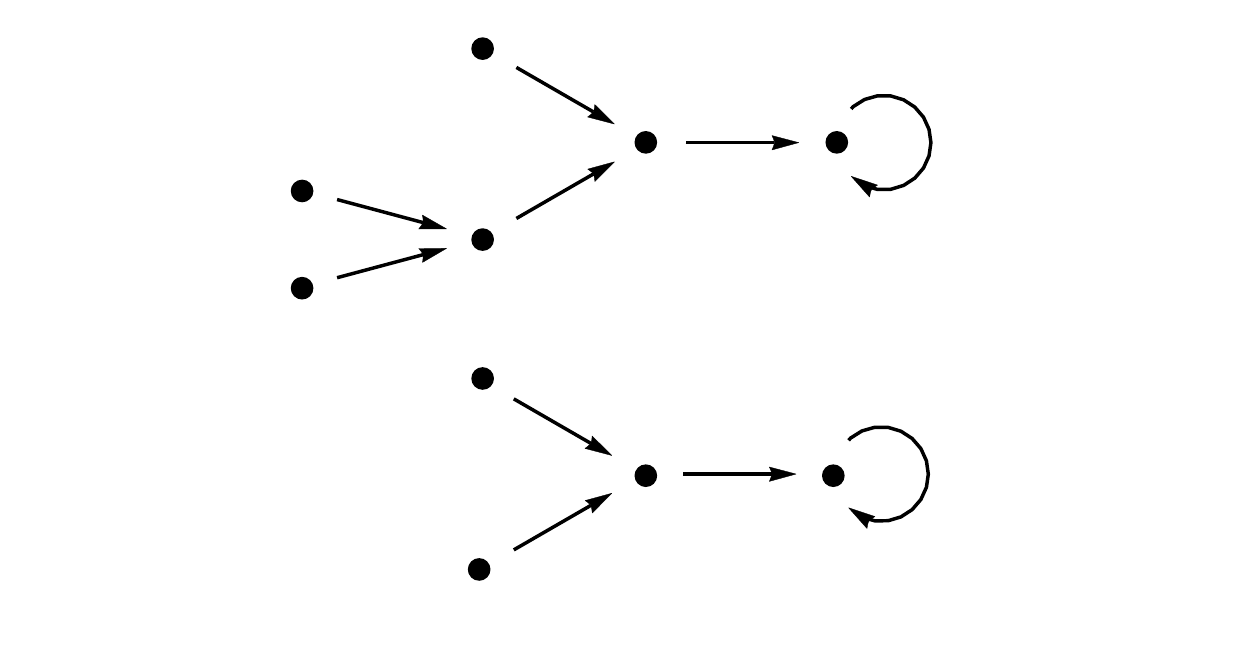}
		\put(42,82){$\alpha$}
		\put(74,16){$\beta$}
	\end{overpic}
	\caption{The graph $G_1$ generated by a point $\alpha$ of type $1_3$ and a point $\beta$ of type $1_2$ with disjoint orbits}
	\label{fig:1_3and_all1_2}
\end{figure}

Among admissible graphs with cycle structure (1,1), only 4(1,1) has four\footnote{Here and throughout the remainder of the article, we will use the fact that an admissible graph necessarily has an even number of vertices.} vertices, only 6(1,1) has six vertices, only 8(1,1)a/b have eight vertices, and only 10(1,1)a/b and $G_1$ have ten vertices, where $G_1$ is the graph appearing in Figure~\ref{fig:1_3and_all1_2}. Therefore, any unknown admissible graph $G(f_c,K)$ of cycle structure (1,1) must contain $G_1$ or \emph{properly} contain 10(1,1)a/b. We now show that the only possibility is that $G(f_c,K)$ properly contains 10(1,1)b.

\begin{prop}\label{prop:(1)}
Let $(K,c)$ be a quadratic pair such that $G(f_c,K)$ is admissible with cycle structure (1,1). If $G(f_c,K)$ does not appear in Appendix~\ref{app:graphs}, then $G(f_c,K)$ properly contains 10(1,1)b.
\end{prop}

In \cite[\textsection 3.8]{doyle/faber/krumm:2014}, we show that there is precisely one quadratic pair, namely $(K,c) = (\bbQ(\sqrt{-7}), 3/16)$, for which $G(f_c,K)$ contains (in fact, is equal to) 10(1,1)a. Therefore, to prove Proposition~\ref{prop:(1)}, it remains to prove the following:

\begin{thm}\label{thm:1_3and_all1_2}
Let $K$ be a quadratic field and let $c \in K$. Then $G(f_c,K)$ does not contain a subgraph isomorphic to $G_1$.
\end{thm}

To prove Theorem~\ref{thm:1_3and_all1_2}, we show that the curve
	\[ U_1(G_1) = \{(\alpha,\beta,c) : \mbox{ $\alpha$ is of type $1_3$ and $\beta$ is of type $1_2$ for $f_c$, with $\beta \ne \pm f_c(\alpha)$} \} \]
has no quadratic points.

\begin{prop}\label{prop:1_3and_all1_2curve}
Let $Y$ be the affine curve of genus 5 defined by the equation
	\begin{equation}\label{eq:1_3and_all1_2curve}
		\begin{cases}
		y^2 &= -(x^2 - 3)(x^2 + 1)\\
		z^2 &= -2(x^3 - x^2 - x - 1),
		\end{cases}
	\end{equation}
and let $U$ be the open subset of $Y$ defined by
	\begin{equation}\label{eq:1_3and_all1_2cusps}
	(x-1)(x+1)(x^2 + 1)(x^2 + 3) \ne 0.
	\end{equation}
Consider the morphism $\Phi : U \to \bbA^3$, $(x,y,z) \mapsto (\alpha,\beta,c)$, given by
	\[ \alpha = \frac{z}{(x-1)(x+1)}, \ \beta = \frac{y}{(x-1)(x+1)} , \ c = \frac{-2(x^2 + 1)}{(x - 1)^2(x+1)^2}.\]
Then $\Phi$ maps $U$ isomorphically onto $U_1(G_1)$, with the inverse map given by
	\begin{equation}\label{eq:1_3and_all1_2inverse}
	x = -\frac{f_c(\alpha)}{f_c^2(\alpha)} , \ y = \beta(x-1)(x+1) , \ z = \alpha(x-1)(x+1).
	\end{equation}
\end{prop}

\begin{proof}
The condition $(x-1)(x+1) \ne 0$ implies that $\Phi$ is well-defined, and one can check that \eqref{eq:1_3and_all1_2inverse} provides a left inverse for $\Phi$, hence $\Phi$ is injective.

Let $(x,y,z)$ be a point on $U$, and let $(\alpha,\beta,c) = \Phi(x,y,z)$. A Magma computation shows that $f_c^4(\alpha) = f_c^3(\alpha)$ and
	\[
		f_c^3(\alpha) - f_c^2(\alpha) = -\frac{4}{(x-1)(x+1)} \ne 0,
	\]
so $\alpha$ is a point of type $1_3$ for $f_c$. Similarly, we find that $f_c^3(\beta) = f_c^2(\beta)$ and
	\[
		f_c^2(\beta) - f_c(\beta) = \frac{2(x^2 + 1)}{(x-1)(x+1)},
	\]
which is nonzero by hypothesis, so $\beta$ is a point of type $1_2$ for $f_c$. Finally, we check that
	\[
		(f_c(\alpha) - \beta)(f_c(\alpha) + \beta) = \frac{x^2 + 3}{(x-1)(x+1)},
	\] 
which is also nonzero by hypothesis. Hence $\beta \ne \pm f_c(\alpha)$, so $\Phi$ maps $U$ into $U_1(G_1)$. It remains now to show that $\Phi$ maps $U$ onto $U_1(G_1)$.

Let $(\alpha,\beta,c)$ be a point on $U_1(G_1)$. By \cite[p. 19]{poonen:1998}, since $f_c(\alpha) \ne \pm \beta$ is a point of type $1_2$, there exists $x \not \in \{\pm 1\}$ such that
	\begin{equation}\label{eq:1_2and1_3params}
	c = \frac{-2(x^2 + 1)}{(x - 1)^2(x+1)^2} , \ f_c(\alpha) = -\frac{2x}{(x-1)(x+1)} , \ \beta^2 = -\frac{(x^2 - 3)(x^2 + 1)}{(x - 1)^2(x+1)^2}.
	\end{equation}
Setting $y = \beta(x-1)(x+1)$ yields the first equation in \eqref{eq:1_3and_all1_2curve}; writing $f_c(\alpha) = \alpha^2 + c$ and using the relations in \eqref{eq:1_2and1_3params} gives us
	\[ \alpha^2 = -\frac{2(x^3 - x^2 - x - 1)}{(x-1)^2(x+1)^2}, \]
and setting $z = \alpha(x-1)(x+1)$ gives us the second equation in \eqref{eq:1_3and_all1_2curve}.
\end{proof}

\begin{thm}\label{thm:1_2and1_3curve_pts}
Let $Y$ be the genus 5 affine curve defined by \eqref{eq:1_3and_all1_2curve}. Then
	\[ Y(\bbQ,2) = Y(\bbQ) = \{ (\pm 1, \pm 2, \pm 2) \}. \]
\end{thm}

Note that Theorem~\ref{thm:1_3and_all1_2} is an immediate consequence of Theorem~\ref{thm:1_2and1_3curve_pts}, since $Y(\bbQ,2)$ does not intersect the open set $U$ defined by \eqref{eq:1_3and_all1_2cusps}.

\begin{proof}[Proof of Theorem~\ref{thm:1_2and1_3curve_pts}]
Let
	\begin{align*}
		f(x) &:= -(x^2 - 3)(x^2 + 1),\\
		g(x) &:= -2(x^3 - x^2 - x - 1),
	\end{align*}
and let $C_1$ and $C_2$ be the affine curves defined by $y^2 = f(x)$ and $z^2 = g(x)$, respectively. As indicated in \cite[p. 19]{poonen:1998}, the only rational points on $C_1$ are the four points $(\pm 1, \pm 2)$. Setting $x = \pm 1$ gives $z = \pm 2$ on $C_2$, so the only rational points on $Y$ are those listed in the theorem. To show that $Y$ has no quadratic points, we must consider separately the cases $x \in \bbQ$ and $x \not \in \bbQ$.

\textbf{Case 1:} $x \in \bbQ$. A computation in Magma verifies that the curve $C_1$ (resp., $C_2$) is birational to the elliptic curve 24A4 (resp., 11A3) from \cite{cremona:1997}, which has only four (resp., five) rational points. (The curves 11A3 and 24A4 are isomorphic to the modular curves $\Xell_1(11)$ and $\Xell_1(2,12)$, respectively.) Hence the only rational points on $C_1$ are $(\pm 1, \pm 2)$, and the only rational points on $C_2$ are $(\pm 1, \pm 2)$ (plus the rational point at infinity). Therefore, if $x \not \in \{ \pm 1\}$, then $y,z \not \in \bbQ$, so $y$ and $z$ generate the same quadratic extension $K = \bbQ(\sqrt{d})$ with $d \ne 1$ squarefree. Since $y^2,z^2 \in \bbQ$, we may write $y = u\sqrt{d}$ and $z = v\sqrt{d}$ for some $u,v \in \bbQ$. We may therefore rewrite \eqref{eq:1_3and_all1_2curve} as
	\[
		\begin{cases}
			du^2 = f(x)\\
			dv^2 = g(x).
		\end{cases}
	\]
If $p \in \bbZ$ is a prime dividing $d$, then the equation $du^2 = -(x^2 - 3)(x^2 + 1)$ implies that $\ord_p(x) \ge 0$, and therefore $\ord_p(u)$ and $\ord_p(v)$ are nonnegative as well. We reduce modulo $p$ to get
	\[ f(x) \equiv g(x) \equiv 0 \mod p, \]
so $p$ divides $\Res(f,g) = -2^8$. We therefore find that $d \in \{-1, \pm 2\}$. However, the curves $\pm 2u^2 = f(x)$ do not have points over $\bbQ_2$, so we must have $d = -1$. 

Let $C_2^{(-1)}$ denote the twist of $C_2$ by $d = -1$. Magma verifies that $C_2^{(-1)}$ is birational to elliptic curve 176B1 in \cite{cremona:1997}, which has only a single rational point. The curve $C_2^{(-1)}$ has a rational point at infinity and therefore has no affine rational points. We conclude that $Y$ has no quadratic points with $x \in \bbQ$.

\textbf{Case 2:} $x \not \in \bbQ$. We would like to apply Lemma~\ref{lem:ECquad_pts} to the curves $C_1$ and $C_2$, so we require a cubic model for the curve $C_1$. We find that $C_1$ is birational to the elliptic curve $E$ given by
	\[ Y^2 = X(X^2 - X + 1), \]
with
	\begin{align}
	\label{1_2and1_3a} X = -\frac{y-2}{(x+1)^2} \ , \ Y &= \frac{2y - (x^3 + x^2 - x + 3)}{(x+1)^3} \ ; \\
	\label{1_2and1_3b} x = - \frac{Y + 2X - 1}{Y + 1} \ , \ y &= -X(x+1)^2 + 2.
	\end{align}
Since the curve $C_1$ has only four rational points, the only rational points on $E$ are $(0,0)$, $(1,\pm 1)$, and the point at infinity.

By Lemma~\ref{lem:ECquad_pts}, there exist $(x_0,z_0) \in C_2(\bbQ)$ and $v \in \bbQ$ such that the minimal polynomial of $x$ is given by
	\begin{equation}\label{1_2and1_3min_poly}
		p(t) = t^2 + \frac{2x_0 + v^2 - 2}{2}t + \frac{2x_0^2 - v^2x_0 - 2x_0 + 2z_0v - 2}{2}.
	\end{equation}
We now use the curve $E$ and the relations in \eqref{1_2and1_3a} and \eqref{1_2and1_3b} to find another expression for the minimal polynomial of $x$. We handle the cases $X \in \bbQ$ and $X \not \in \bbQ$ separately.

\textbf{Case 2a:} $X \in \bbQ$. Substituting $y = -X(x+1)^2 + 2$ into the equation $y^2 = f(x)$ yields
	\[ (x+1)^2\Big( (X^2 + 1)x^2 + 2(X+1)(X-1)x + (X^2 - 4X + 1) \Big) = 0. \]
Since $x \not \in \bbQ$, we have $(x+1) \ne 0$, so $x$ must have minimal polynomial
	\begin{equation}\label{eq:1_2and1_3min_poly_2}
		p(t) = t^2 + \frac{2(X+1)(X-1)}{X^2 + 1}t + \frac{X^2 - 4X + 1}{X^2 + 1}.
	\end{equation}
Equating the coefficients of \eqref{1_2and1_3min_poly} and \eqref{eq:1_2and1_3min_poly_2} yields the system
	\[
		\begin{cases}
			\hfill (2x_0 + v^2 - 2)(X^2 + 1) &= 4(X+1)(X-1)\\
			(2x_0^2 - v^2x_0 - 2x_0 + 2z_0v - 2)(X^2 + 1) &= 2(X^2 - 4X + 1).
		\end{cases}
	\]
For each pair $(x_0,z_0) \in \{ (\pm 1, \pm 2) \}$, this system defines a zero-dimensional scheme for which Magma can compute the set of all rational points $(X,v)$. For $(x_0,z_0) = (1,2)$, the only rational point is $(X,v) = (0,2)$, which yields $p(t) = (t-1)^2$. Similarly, for $(x_0, z_0) = (1,-2)$, the only rational point is $(X,v) = (0,-2)$, which again yields $p(t) = (t-1)^2$. For $(x_0, z_0) \in \{ (-1, \pm 2) \}$ the scheme has no rational points. Since the only quadratic polynomials $p(t)$ which have arisen in this way are reducible, we conclude that there are no quadratic points with $x \not \in \bbQ$ and $X \in \bbQ$.

\textbf{Case 2b:} $X \not \in \bbQ$. Again using Lemma~\ref{lem:ECquad_pts}, there exist $(X_0,Y_0) \in E(\bbQ)$ and $w \in \bbQ$ such that
	\begin{equation}\label{1_2and1_3min_poly2}
		\begin{gathered}
		X^2 + (X_0 - w^2 - 1)X + (X_0^2 + w^2X_0 - X_0 - 2Y_0w + 1) = 0 , \\
		Y = Y_0 + w(X - X_0).
		\end{gathered}
	\end{equation}
For each of the points $(X_0,Y_0) \in \{ (0,0), (1,\pm 1) \}$, we combine \eqref{1_2and1_3min_poly2} with \eqref{1_2and1_3b} to find another expression for the minimal polynomial of $x$.

For $(X_0,Y_0) = (0,0)$, \eqref{1_2and1_3min_poly2} becomes
	\[ X^2 - (w^2 + 1)X + 1 = 0  , \ Y = wX. \]
Combining this with \eqref{1_2and1_3b} tells us that
	\[ (w+1)\left( (w^2 + 1)x^2 + 4wx - (w^2 - 3) \right) = 0. \]
We cannot have $w = -1$, since \eqref{1_2and1_3min_poly2} would then imply that $X^2 - 2X + 1 = 0$, contradicting the assumption that $X \not \in \bbQ$. Therefore, the minimal polynomial of $x$ is given by
	\[ p(t) = t^2 + \frac{4w}{w^2 + 1}t - \frac{w^2 - 3}{w^2 + 1} = 0. \]
Equating coefficients of this polynomial with those in \eqref{1_2and1_3min_poly} yields the following system:
	\[
		\begin{cases}
			\hfill 8w &= (2x_0 + v^2 - 2)(w^2 + 1)\\
			-2(w^2 - 3) &= (2x_0^2 - v^2x_0 - 2x_0 + 2z_0v - 2)(w^2 + 1),
		\end{cases}
	\]
where $(x_0,z_0) \in \{(\pm 1, \pm 2)\}$. For each pair $(x_0, z_0)$, the system defines a zero-dimensional scheme, and again Magma can compute the set of rational points on each scheme. In each case, there is exactly one rational point. For $(x_0,z_0) = (1, \pm 2)$, we have $(v,w) = (\pm 2, 1)$, and each of $(x_0,z_0) = (-1, \pm 2)$ yields $(v,w) = (0,-1)$. In any case, the corresponding polynomial $p(t)$ equals $t^2 \pm 2t + 1$ and is therefore reducible over $\bbQ$.

We perform a similar computation for $(X_0,Y_0) = (1,1)$, and in this case we find only the polynomial $p(t) = t^2 - 1$, which is again reducible over $\bbQ$. For $(X_0,Y_0) = (1,-1)$, \eqref{1_2and1_3b} and \eqref{1_2and1_3min_poly2} combine to give the equation $w + 1 = 0$, which we have already ruled out. Therefore there are no quadratic points on $X$ with $x \not \in \bbQ$, which concludes the proof.
\end{proof}

\subsection{Period 2}

\begin{figure}[h]
	\includegraphics[scale=.55]{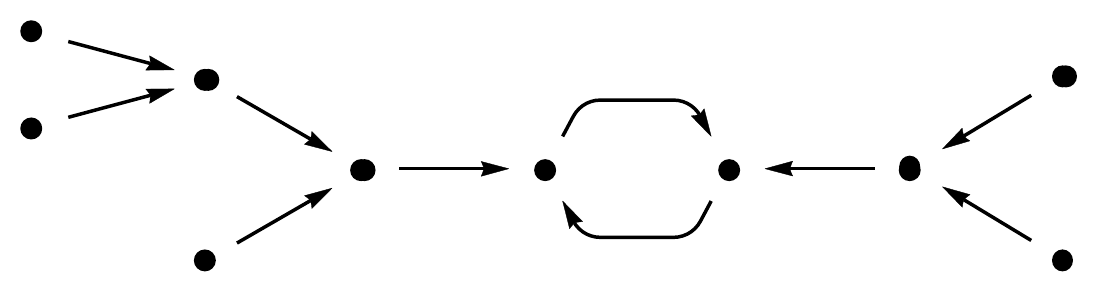}
\caption{The graph $G_2$}
\label{fig:graph12_2_minus2pts}
\end{figure}

\begin{figure}[h]
	\includegraphics[scale=.5]{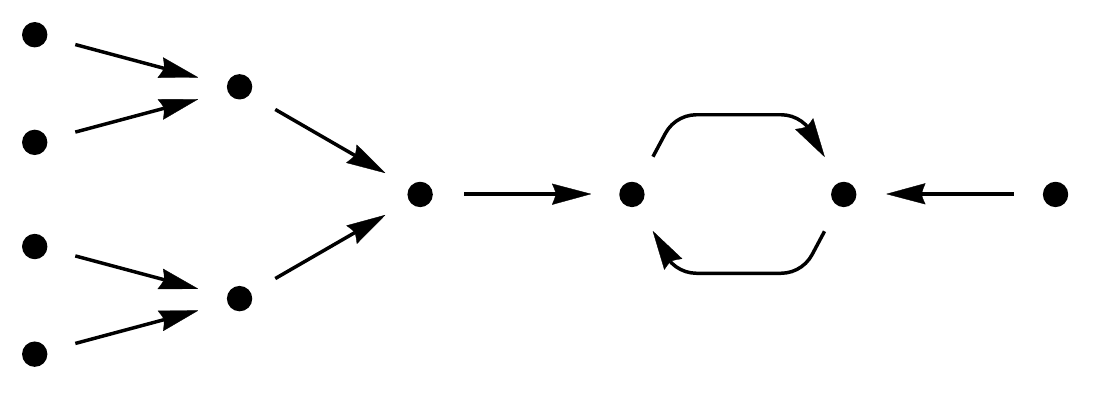}
\caption{The graph $G_3$}
\label{fig:subgraph_of_14_211}
\end{figure}

Among admissible graphs with cycle structure (2), only 4(2) has four vertices, only 6(2) has six vertices, only 8(2)a/b have eight vertices, and only 10(2), $G_2$, and $G_3$ have ten vertices, where $G_2$ and $G_3$ are shown in Figures~\ref{fig:graph12_2_minus2pts} and \ref{fig:subgraph_of_14_211}, respectively.

It follows, then, that if $G(f_c,K)$ is an admissible graph with cycle structure (2) which does not appear in Appendix~\ref{app:graphs}, then $G(f_c,K)$ either contains $G_2$ or $G_3$ or properly contains 10(2).

In \cite[\textsection 3.13]{doyle/faber/krumm:2014}, we show that in addition to the two known pairs $(K,c) = (\bbQ(\sqrt{2}), -15/8)$ and $(K,c) = (\bbQ(\sqrt{57}),-55/48))$, there are at most four additional quadratic pairs $(K,c)$ for which $G(f_c,K)$ contains the graph 12(2), and for each such pair we have $c \in \bbQ$. However, the proof actually involved showing that there were at most four additional quadratic pairs $(K,c)$ --- each with $c \in \bbQ$ --- for which $G(f_c,K)$ contains $G_2$, which is a subgraph of 12(2).

In a similar way, the analysis of the graph 14(2,1,1) in \cite{doyle/faber/krumm:2014} was actually done by studying the graph $G_3$, a subgraph of 14(2,1,1). We showed in \cite[\textsection 3.18]{doyle/faber/krumm:2014} that there is precisely one quadratic pair $(K,c) = (\bbQ(\sqrt{17}),-21/16)$ for which $G(f_c,K)$ contains $G_3$, and in this case $G(f_c,K)$ is equal to 14(2,1,1).

We therefore have the following:

\begin{prop}\label{prop:(2)}
Let $(K,c)$ be a quadratic pair such that $G(f_c,K)$ is admissible with cycle structure (2). If $G(f_c,K)$ does not appear in Appendix~\ref{app:graphs}, then $G(f_c,K)$ contains $G_2$ or properly contains 10(2). Moreover, there are at most four quadratic pairs $(K,c)$ for which $G(f_c,K)$ contains $G_2$, and for each such pair we must have $c \in \bbQ$.
\end{prop}

\subsection{Period 3}

Among admissible graphs with cycle structure (3), only 6(3) has six vertices, only 8(3) has eight vertices, and only 10(3)a/b have ten vertices. We are therefore able to make the following statement.

\begin{prop}\label{prop:(3)}
Let $(K,c)$ be a quadratic pair such that $G(f_c,K)$ is admissible with cycle structure (3). If $G(f_c,K)$ does not appear in Appendix~\ref{app:graphs}, then $G(f_c,K)$ properly contains 10(3)a or 10(3)b.
\end{prop}

\subsection{Period 4}

Among admissible graphs with cycle structure (4), only 8(4) has eight vertices. Recall from Proposition~\ref{prop:4cycle_rat_c} that if $(K,c)$ is a quadratic pair for which $G(f_c,K)$ contains a 4-cycle, then $c \in \bbQ$ and $\sigma(\alpha) = f_c^2(\alpha)$ for each point $\alpha$ of period 4, where $\sigma$ is the nontrivial element of $\Gal(K/\bbQ)$. Now suppose $G(f_c,K)$ contains a point $\beta$ of type $4_2$, with $f_c^2(\beta) = \alpha$ a point of period 4. Then $\sigma(\beta)$ must necessarily be a point of type $4_2$, with $f_c^2(\sigma(\beta)) = \sigma(f_c^2(\beta)) = \sigma(\alpha) = f_c^2(\alpha)$. Therefore if $G(f_c,K)$ contains 8(4) with just one additional vertex, then, by admissibility and the action of $\sigma$, $G(f_c,K)$ must contain 12(4). In \cite[Cor. 3.44]{doyle/faber/krumm:2014}, we showed that there are at most five unknown quadratic pairs $(K,c)$ for which $G(f_c,K)$ contains 12(4), and that for each such pair we must have $c \in \bbQ$. We therefore have the following:

\begin{prop}\label{prop:(4)}
Let $(K,c)$ be a quadratic pair such that $G(f_c,K)$ is admissible with cycle structure (4). If $G(f_c,K)$ does not appear in Appendix~\ref{app:graphs}, then $G(f_c,K)$ properly contains 12(4). Moreover, there are at most five such quadratic pairs, and each such quadratic pair must have $c \in \bbQ$.
\end{prop}

\subsection{Periods 1 and 2}

Among admissible graphs with cycle structure (1,1,2), only 8(2,1,1) has eight vertices, only 10(2,1,1)a/b have ten vertices, and only 12(2,1,1)a/b, $G_4$, $G_5$, and $G_6$ have twelve vertices, where $G_4$, $G_5$, and $G_6$ are shown in Figures~\ref{fig:1and2_3}, \ref{fig:2and_all1_2}, and \ref{fig:1_2and2_2}, respectively. Therefore, a new admissible graph with cycle structure (1,1,2) must necessarily contain $G_4$, $G_5$, or $G_6$ or properly contain 12(2,1,1)a/b. However, we can say more:

\begin{prop}\label{prop:(1,2)}
Let $(K,c)$ be a quadratic pair such that $G(f_c,K)$ is admissible with cycle structure (1,1,2). If $G(f_c,K)$ does not appear in Appendix~\ref{app:graphs}, then $G(f_c,K)$ either contains $G_4$ or properly contains 12(2,1,1)b. Moreover, there are at most three (resp., four) such quadratic pairs for which $G(f_c,K)$ contains $G_4$ (resp., properly contains 12(2,1,1)b), and for each such pair we must have $c \in \bbQ$.
\end{prop}

It is shown in \cite[Cor. 3.36]{doyle/faber/krumm:2014} that there is only one quadratic pair $(K,c) = (\bbQ(\sqrt{-7}),-5/16)$ for which $G(f_c,K)$ contains (in fact, is equal to) 12(2,1,1)a. Also, \cite[Cor. 3.40]{doyle/faber/krumm:2014} says that there are at most four unknown quadratic pairs $(K,c)$ for which $G(f_c,K)$ contains 12(2,1,1)b, and for each such pair we must have $c \in \bbQ$. To prove Proposition~\ref{prop:(1,2)}, it therefore remains to show that $G(f_c,K)$ cannot contain $G_5$ or $G_6$, and that there are at most three quadratic pairs $(K,c)$ --- each with $c \in \bbQ$ --- such that $G(f_c,K)$ contains $G_4$ but is not equal to 14(2,1,1), the only graph appearing in Appendix~\ref{app:graphs} that contains $G_4$. This is the content of Theorems~\ref{thm:1and2_3}, \ref{thm:2and_all1_2}, and \ref{thm:1_2and2_2} below.

\subsubsection{The graph $G_4$}\label{sub:1and2_3}

\begin{figure}[h]
\centering
\begin{overpic}[scale=.5]{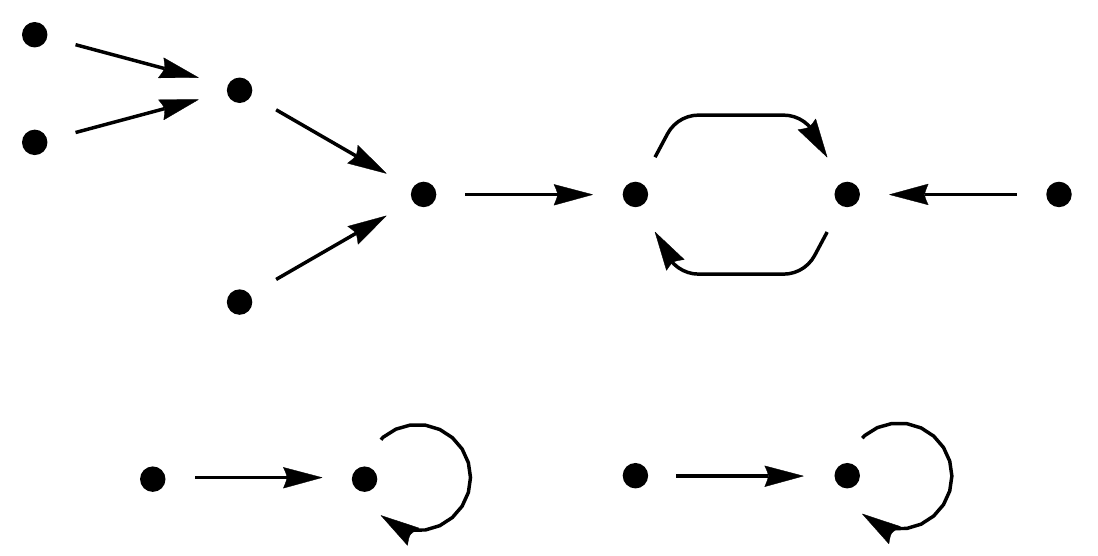}
	\put(45,0){$\alpha$}
	\put(110,0){$\alpha'$}
	\put(-5,52){$\beta$}
\end{overpic}
\caption{The graph $G_4$ generated by fixed points $\alpha$ and $\alpha'$ and a point $\beta$ of type $2_3$}
\label{fig:1and2_3}
\end{figure}
Let $G_4$ be the graph shown in Figure~\ref{fig:1and2_3}. Then $G_4$ is a subgraph of the graph 14(2,1,1) appearing in Appendix~\ref{app:graphs}. In \cite[\textsection 3.18]{doyle/faber/krumm:2014}, we showed that 14(2,1,1) occurs as a subgraph of $G(f_c,K)$ for a single quadratic pair $(K,c) = (\bbQ(\sqrt{17}),-21/16)$, in which case we actually have $G(f_c,K) \cong 14(2,1,1)$. Since $G_4$ is a subgraph of 14(2,1,1), this gives one quadratic pair $(K,c)$ for which $G_4$ is a subgraph of $G(f_c,K)$. In this section, we show that there are at most three additional quadratic pairs with this property:

\begin{thm}\label{thm:1and2_3}
In addition to the pair $(K,c) = (\bbQ(\sqrt{17}),-21/16)$, which satisfies $G(f_c,K) \cong 14(2,1,1)$ and is the unique quadratic pair satisfying $G(f_c,K) \supseteq 14(2,1,1) \supsetneq G_4$, there are at most three quadratic pairs $(K,c)$ for which $G(f_c,K)$ contains a subgraph isomorphic to $G_4$. For each such pair, we must have $c \in \bbQ$.
\end{thm}

As mentioned when we discussed this example following Remark~\ref{rem:redundant}, it suffices for our purposes to replace $G_4$ with the subgraph $G_4'$ generated by $\alpha$ and $\beta$, so we consider the curve
	\[
		U_1(G_4') = \{(\alpha,\beta,c) \in \bbA^3 : \mbox{ $\alpha$ is a fixed point and $\beta$ has type $2_3$ for $f_c$} \}.
	\]
\begin{prop}\label{prop:1and2_3}
Let $Y$ be the affine curve of genus 5 defined by the equation
	\begin{equation}\label{eq:1and2_3curve}
		\begin{cases}
			y^2 &= 2(x^3 + x^2 - x + 1)\\
			z^2 &= 5x^4 + 8x^3 + 6x^2 - 8x + 5,
		\end{cases}
	\end{equation}
and let $U$ be the open subset of $Y$ defined by
	\begin{equation}\label{eq:1and2_3cusps}
		x(x-1)(x+1)(x^2 + 4x - 1) \ne 0.
	\end{equation}
Consider the morphism $\Phi : U \to \bbA^3$, $(x,y,z) \mapsto (\alpha,\beta,c)$, given by
	\[
		\alpha = \frac{x^2 - 1 + z}{2(x-1)(x+1)} , \ \beta = \frac{y}{(x-1)(x+1)} , \ c = -\frac{x^4 + 2x^3 + 2x^2 - 2x + 1}{(x-1)^2(x+1)^2}.
	\]
Then $\Phi$ maps $U$ isomorphically onto $U_1(G_4')$, with the inverse map given by
	\begin{equation}\label{eq:1and2_3inverse}
		x = \frac{f_c(\beta) - 1}{f_c^2(\beta)} , \ y = \beta(x-1)(x+1)  , \ z = \frac{2\alpha - 1}{(x-1)(x+1)}.
	\end{equation}
\end{prop}

\begin{proof}
The condition that $(x-1)(x+1) \ne 0$ ensures that $\Phi$ is a well-defined morphism to $\bbA^3$, and it is not difficult to show that \eqref{eq:1and2_3inverse} provides a left inverse for $\Phi$, so that $\Phi$ is injective.

Now let $(x,y,z) \in Y$ and $(\alpha,\beta,c) = \Phi(x,y,z)$. A computation in Magma shows that $f_c(\alpha) = \alpha$, so that $\alpha$ is a fixed point for $f_c$. Also, a computation shows that $f_c^5(\beta) = f_c^3(\beta)$, and that
	\[
		f_c^5(\beta) - f_c^4(\beta) = -\frac{x^2 + 4x - 1}{(x-1)(x+1)} , \ f_c^4(\beta) - f_c^2(\beta) = \frac{4x}{(x-1)(x+1)}.
	\]
Since the above expressions are nonzero by \eqref{eq:1and2_3cusps}, it follows that $\beta$ is a point of type $2_3$. Therefore $\Phi$ maps $U$ into $U_1(G_4')$. It remains to prove surjectivity.

Let $(\alpha,\beta,c) \in U_1(G_4')$. Since $\beta$ is a point of type $2_3$, we have from \cite[p. 23]{poonen:1998} that there exists $x$ for which
	\[ \beta^2 = \frac{2(x^3 + x^2 - x + 1)}{(x-1)^2(x+1)^2} , \ c = -\frac{x^4 + 2x^3 + 2x^2 - 2x + 1}{(x-1)^2(x+1)^2}. \]
This gives us the correct expression for $c$, and setting $y = \beta(x^2 - 1)$ gives us the first equation in \eqref{eq:1and2_3curve} and the correct expression for $\beta$. Finally, since $\alpha$ is a fixed point, we have
	\[ 0 = f_c(\alpha) - \alpha = \alpha^2 - \alpha - \frac{x^4 + 2x^3 + 2x^2 - 2x + 1}{(x-1)^2(x+1)^2}. \]
The discriminant of this polynomial (in $\alpha$) is
	\[ \Delta := \frac{5 x^4+8 x^3+6 x^2-8 x+5}{(x-1)^2(x+1)^2}. \]
Thus $\alpha = 1/2 + z/(2(x-1)(x+1))$, where $z$ satisfies
	\[ z^2 = \Delta(x-1)^2(x+1)^2 = 5 x^4+8 x^3+6 x^2-8 x+5, \]
which completes the proof.
\end{proof}

To prove Theorem~\ref{thm:1and2_3}, we require an upper bound for the number of quadratic points on $U_1(G_4)$.

\begin{thm}\label{thm:1and2_3quad_pts}
Let $Y$ be the genus 5 affine curve defined by \eqref{eq:1and2_3curve}. Then
	\begin{align*}
	Y(\bbQ,2) &= Y(\bbQ) \cup \{ (x,\pm 2(2x+1), \pm 20x) : x^2 - 8x - 1 = 0 \} \cup \calS \\
		&= \{ (\pm 1, \pm 2, \pm 4)\} \cup \{ (x,\pm 2(2x+1), \pm 20x) : x^2 - 8x - 1 = 0 \} \cup \calS,
	\end{align*}
where $\#\calS \in \{0,4,8,12\}$. Moreover, if $(x,y,z) \in \calS$, then $x \in \bbQ$, and $(x, \pm y, \pm z)$ are four distinct points in $\calS$.
\end{thm}

We first prove the following lemma:

\begin{lem}\label{lem:1and2_3aux}
Let $C_3$ be the genus 3 hyperelliptic curve defined by
	\begin{equation}\label{eq:1and2_3aux}
		w^2 = 2(x^3 + x^2 - x + 1)(5x^4 + 8x^3 + 6x^2 - 8x + 5).
	\end{equation}
Then
	\[
		C_3(\bbQ) = \{(\pm 1, \pm 8), \infty\} \cup \calS_3,
	\]
where $\#\calS_3 \in \{0, 2, 4, 6 \}$. Moreover, if $\calS_3$ is nonempty and $(x,w) \in \calS_3$, then $(x, \pm w)$ are two distinct points on $C_3$.
\end{lem}

\begin{proof}
A quick search for points on $C_3$ yields the five points listed in the lemma. Now let $J_3$ be the Jacobian of $C_3$. Applying Magma's \texttt{RankBound} function to $J_3$ shows that $\rk J(\bbQ) \le 1$. (Since the point $[(1,8) - \infty] \in J_3(\bbQ)$ has infinite order, we actually have $\rk J(\bbQ) = 1$.) We may therefore apply Theorem~\ref{thm:stoll} to bound $\#C_3(\bbQ)$. The prime $p = 7$ is a prime of good reduction for $C_3$, and $\#C_3(\bbF_7) = 10$, so
	\[ \#C_3(\bbQ) \le \#C_3(\bbF_7) + 2 \cdot \rk J_3(\bbQ) = 12. \]
The point $\infty$ is a rational Weierstrass point on $C_3$, since \eqref{eq:1and2_3aux} has odd degree. Also, since the cubic and quartic polynomials appearing in \eqref{eq:1and2_3aux} are irreducible over $\bbQ$, $C_3$ has no finite rational Weierstrass points. Non-Weierstrass points come in pairs $(x, \pm w)$, so $\#C_3(\bbQ)$ must be odd, and therefore $\#C_3(\bbQ) \le 11$. Since we already have five rational points on $C_3$, there are at most six additional points.
\end{proof}

\begin{proof}[Proof of Theorem~\ref{thm:1and2_3quad_pts}]
Let
	\begin{align*}
		f(x) &:= 2(x^3 + x^2 - x + 1);\\
		g(x) &:= 5x^4 + 8x^3 + 6x^2 - 8x + 5;
	\end{align*}
and let $C_1$ and $C_2$ be the curves $y^2 = f(x)$ and $z^2 = g(x)$, respectively. Magma verifies that $C_1$ (resp., $C_2$) is birational to elliptic curve 11A3 (resp., 17A4) in \cite{cremona:1997}, which has only five (resp., four) rational points. (As mentioned previously, the curve 11A3 is isomorphic to $X_1^{\Ell}(11)$.) Hence $C_1(\bbQ)$ consists of the four points $(\pm 1, \pm 2)$ (plus the rational point at infinity), and $C_2(\bbQ)$ consists of the four points $(\pm 1, \pm 4)$. It follows that the only rational points on $Y$ are those listed in the theorem. We now suppose $(x,y,z)$ is a quadratic point on $X$, and we consider separately the cases $x \in \bbQ$ and $x \not \in \bbQ$.

\textbf{Case 1:} $x \in \bbQ$. If $x \in \{\pm 1\}$, then we already know that $(x,y,z)$ is a rational point. On the other hand, if $x \in \bbQ \setminus \{\pm 1\}$, then it follows from the previous paragraph that $y,z \not \in \bbQ$. Since $y^2,z^2 \in \bbQ$, we find that setting $w := yz$ yields a finite rational point $(x,w)$ on the curve $C_3$ defined by \eqref{eq:1and2_3aux}. Since $x \not \in \{\pm 1\}$, the point $(x,w)$ must lie in $\calS_3$; in particular, there are at most six such points by Lemma~\ref{lem:1and2_3aux}. Since the map $Y \to C_3$ given by $(x,y,z) \mapsto (x,yz)$ has degree two, this implies that there are at most 12 additional quadratic points on $Y$ with $x \in \bbQ$. Finally, since the polynomials $f$ and $g$ are irreducible over $\bbQ$, we have $y,z \ne 0$, and therefore $(x,\pm y, \pm z)$ are four distinct points on $Y$.

\textbf{Case 2:} $x \not \in \bbQ$. We would like to apply Lemma~\ref{lem:ECquad_pts} to the curve $C_1$ and $C_2$, just as we did to prove Theorem~\ref{thm:1_2and1_3curve_pts}. We therefore require a cubic model for $C_2$, and we find that $C_2$ is birational to the elliptic curve $E$ defined by
	\[
		Y^2 = (X - 4)(X^2 + X - 4),
	\]
with
	\begin{align}
		X = \frac{2(x^2 + 3 - z)}{(x + 1)^2} \ &, \ Y = \frac{2(3x^3 + 3x^2 + 9x - 7 - (x-3)z)}{(x+1)^3} \label{eq:1and2_3X} \\
		x = \frac{3X - 4 + Y}{X + 4 - Y} \ &, \ z = -\frac{1}{2}X(x+1)^2 + x^2 + 3. \label{eq:1and2_3x}
	\end{align}
Since $C_2$ has only four rational points, we have $E(\bbQ) = \{(0,\pm 4), (4,0) , \infty\}$.

By Lemma~\ref{lem:ECquad_pts}, there exist $(x_0, y_0) \in C_1(\bbQ)$ and $v \in \bbQ$ such that the minimal polynomial of $x$ is given by
	\begin{equation}\label{eq:1and2_3min_poly_a}
		p(t) = t^2 + \frac{2x_0 - v^2 + 2}{2}t + \frac{2x_0^2 + v^2x_0 + 2x_0 - 2y_0v - 2}{2}.
	\end{equation}
We now use the curve $E$ and the relations in \eqref{eq:1and2_3X} and \eqref{eq:1and2_3x} to find another expression for the minimal polynomial of $x$. We handle the cases $X \in \bbQ$ and $X \not \in \bbQ$ separately.

\textbf{Case 2a:} $X \in \bbQ$. Substituting $z = -\frac{1}{2}X(x+1)^2 + x^2 + 3$ into the equation $z^2 = g(x)$ yields
	\[
		\frac{1}{4}(x+1)^2\Big( (X^2 - 4X - 16)x^2 + 2X^2 x + (X^2 - 12X + 16) \Big) = 0.
	\]
Since $x \not \in \bbQ$, $(x+1) \ne 0$, so $x$ must have minimal polynomial
	\begin{equation}\label{eq:1and2_3min_poly_b}
		p(t) = t^2 + \frac{2X^2}{X^2 - 4X - 16}t + \frac{X^2 - 12X + 16}{X^2 - 4X - 16}.
	\end{equation}
Equating the coefficients of \eqref{eq:1and2_3min_poly_a} and \eqref{eq:1and2_3min_poly_b} yields the system
	\[
		\begin{cases}
			\hfill (2x_0 - v^2 + 2)(X^2 - 4X - 16) &= 4X^2\\
			(2x_0^2 + v^2x_0 + 2x_0 - 2y_0v - 2)(X^2 - 4X - 16) &= 2(X^2 - 12X + 16)
		\end{cases}
	\]
For each pair $(x_0,y_0) \in \{ (\pm 1, \pm 2) \}$, Magma computes all rational solutions to the given system. Just as in case 2a of the proof of Theorem~\ref{thm:1_2and1_3curve_pts}, we find that for each such rational solution, the corresponding polynomial $p(t)$ is reducible over $\bbQ$. Therefore there are no quadratic points with $x \not \in \bbQ$ and $X \in \bbQ$.

\textbf{Case 2b:} $X \not \in \bbQ$. Again using Lemma~\ref{lem:ECquad_pts}, there exist $(X_0,Y_0) \in E(\bbQ)$ and $w \in \bbQ$ such that
	\begin{equation}\label{eq:1and2_3min_poly_X} 
		\begin{gathered}
			X^2 + (X_0 - w^2 - 3)X + (X_0^2 + w^2X_0 - 3X_0 - 2Y_0w - 8) = 0, \\
			Y = Y_0 + w(X - X_0) .
		\end{gathered}
	\end{equation}
For each of the points $(X_0,Y_0) \in \{ (0, \pm 4), (4,0) \}$, we combine \eqref{eq:1and2_3min_poly_X} with \eqref{eq:1and2_3x} to find another expression for the minimal polynomial of $x$. More precisely, for each $(x_0,y_0) \in \{ (\pm 1, \pm 2) \}$, and for each $(X_0,Y_0) \in \{ (0, \pm 4), (4,0)\}$, we arrive at two expressions for $p(t)$; equating coefficients of these expressions yields a system of equations all of whose rational solutions may be found, either by hand or in Magma.

Consider $(x_0,y_0) = (-1,2)$ and $(X_0,Y_0) = (0, -4)$. We first recall that, by \eqref{eq:1and2_3min_poly_a}, one expression for $p(t)$ is given by
	\[
		p(t) = t^2 - \frac{v^2}{2}t - \frac{v^2 + 4v + 2}{2}.
	\]
On the other hand, we may write \eqref{eq:1and2_3min_poly_X} as
	\[
		X^2 - (w^2 + 3)X + 8(w - 1) \ , \ Y = wX - 4;
	\]
combining this with \eqref{eq:1and2_3x}, then clearing denominators appropriately, yields
	\[
		p(t) = t^2 + \frac{8(w - 1)}{w^2 - 5}t - 1.
	\]
Equating the coefficients of these two expressions for $p(t)$ yields the system
	\[
		\begin{cases}
			\hfill v^2(w^2 - 5) &= -16(w - 1)\\
			v^2 + 4v + 2 &= 2.
		\end{cases}
	\]
One can verify by hand that the only solutions are $(v,w) \in \{(0,1),(-4,-3),(-4,2)\}$. The pair $(0,1)$ yields $p(t) = t^2 - 1$, which is reducible over $\bbQ$; on the other hand, the pairs with $v = -4$ yield $p(t) = t^2 - 8t - 1$, which gives us precisely the known quadratic points listed in the statement of the theorem.

For every other choice of $(x_0,y_0)$ and $(X_0,Y_0)$, a similar computation finds either that $p(t)$ is reducible over $\bbQ$ or that $p(t) = t^2 - 8t - 1$. It therefore follows that the only quadratic points on the curve $Y$ with $x \not \in \bbQ$ are those listed in the statement of the theorem, concluding the proof.
\end{proof}

We now complete the proof of Theorem~\ref{thm:1and2_3}.

\begin{proof}[Proof of Theorem~\ref{thm:1and2_3}]
Observe that, by Theorem~\ref{thm:1and2_3quad_pts}, all rational points on $Y$ lie outside the open subset $U$ defined in Proposition~\ref{prop:1and2_3}. Therefore, every point in $U(\bbQ,2)$ must either satisfy $x^2 - 8x - 1 = 0$ or lie in $\calS$. If $x^2 - 8x - 1 = 0$, then we recover the quadratic pair $(\bbQ(\sqrt{17}),-21/16)$, which was shown in \cite[Cor. 3.52]{doyle/faber/krumm:2014} to be the unique quadratic pair $(K,c)$ for which $G(f_c,K)$ is isomorphic to 14(2,1,1). Note that $G_4$ is a proper subgraph of 14(2,1,1).

Now suppose $(K,c)$ is another quadratic pair whose preperiodic graph contains a subgraph isomorphic to $G_4$. Let $\alpha$ be a $K$-rational fixed point for $f_c$, and let $\beta$ be a $K$-rational point of type $2_3$ for $f_c$. Then there is a point $(x,y,z) \in Y(K)$ such that
	\[
		\alpha = \frac{x^2 - 1 + z}{2(x^2 - 1)} , \ \beta = \frac{y}{x^2 - 1} , \ c = -\frac{x^4 + 2x^3 + 2x^2 - 2x + 1}{(x^2 - 1)^2}.
	\]
Moreover, this point $(x,y,z)$ --- in fact, all four points $(x,\pm y, \pm z)$ --- must lie in $\calS$. In particular, $x \in \bbQ$. Since $c$ is defined only in terms of $x$, it follows that $c \in \bbQ$ and that all four points $(x, \pm y, \pm z)$ yield the same value of $c$. Since there are at most twelve points in $\calS$ from Theorem~\ref{thm:1and2_3quad_pts}, and since four such points yield the same value of $c$, it follows that there are at most three quadratic pairs $(K,c)$ --- each with $c \in \bbQ$ --- for which $G_4 \subseteq G(f_c,K)$.
\end{proof}

\subsubsection{The graph $G_5$}\label{sub:2and_all1_2}

\begin{figure}[h]
\centering
	\begin{overpic}[scale = .5]{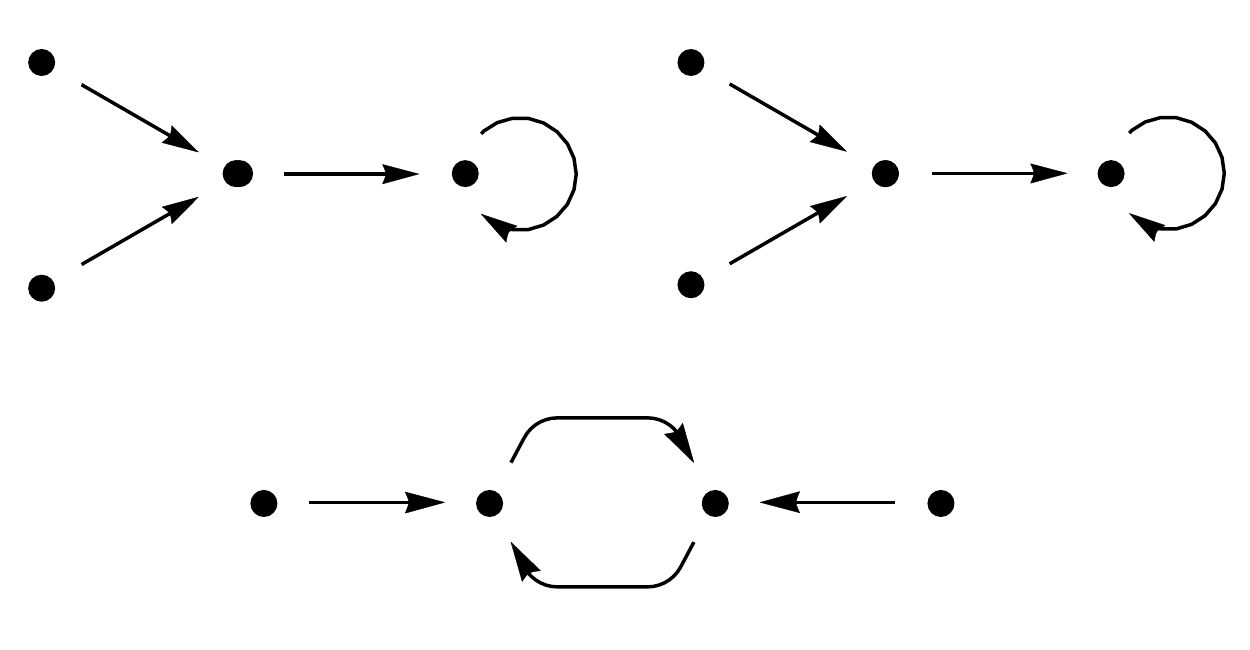}
		\put(-8,90){$\alpha_1$}
		\put(86,90){$\alpha_2$}
		\put(62,30){$\beta$}
	\end{overpic}
\caption{The graph $G_5$ generated by two points $\alpha_1$ and $\alpha_2$ of type $1_2$ with disjoint orbits and a point $\beta$ of period 2}
\label{fig:2and_all1_2}
\end{figure}

For the graph $G_5$ shown in Figure~\ref{fig:2and_all1_2}, we have the following:

\begin{thm}\label{thm:2and_all1_2}
Let $K$ be a quadratic field and let $c \in K$. Then $G(f_c,K)$ does not contain a subgraph isomorphic to $G_5$.
\end{thm}

To prove the theorem, we show that
	\begin{align*}
	U_1(G_5) = \{(\alpha_1,\alpha_2,\beta,c) \in \bbA^4 : \ &\alpha_1,\alpha_2 \mbox{ are points of type $1_2$ and}\\
		&\beta \mbox{ is a point of period 2 for $f_c$, with $f_c(\alpha_1) \ne f_c(\alpha_2)$} \}.
	\end{align*}
has no quadratic points.
	
\begin{prop}\label{prop:2and_all1_2curve}
Let $Y$ be the affine curve of genus 5 defined by the equation
	\begin{equation}\label{eq:2all1_2curve}
		\begin{cases}
			y^2 &= (5x^2 - 1)(x^2 + 3)\\
			z^2 &= -(3x^2 + 1)(x^2 - 5),
		\end{cases}
	\end{equation}
and let $U$ be the open subset of $Y$ defined by
	\begin{equation}\label{eq:2all1_2cusps}
		x(x - 1)(x + 1)(x^2 + 1)(x^2 + 3)(3x^2 + 1) \ne 0.
	\end{equation}
Consider the morphism $\Phi : Y \to \bbA^4$, $(x,y,z) \mapsto (\alpha_1,\alpha_2,\beta,c)$, given by
	\[ \alpha_1 = \frac{y}{2(x-1)(x+1)} , \ \alpha_2 = \frac{z}{2(x-1)(x+1)} , \ \beta = -\frac{x^2 - 4x - 1}{2(x-1)(x+1)} , \ c = -\frac{(x^2 + 3)(3x^2 + 1)}{4(x-1)^2(x+1)^2}. \]
Then $\Phi$ maps $U$ isomorphically onto $U_1(G_5)$, with the inverse map given by
	\begin{equation}\label{eq:2all1_2inverse}
	x = \frac{2f_c(\alpha_1) + 3}{2\beta + 1} , \ y = 2\alpha_1(x-1)(x+1)  , \ z = 2\alpha_2(x-1)(x+1).
	\end{equation}
\end{prop}

\begin{proof}
The condition $(x - 1)(x + 1) \ne 0$ implies that $\Phi$ defines a morphism to $\bbA^4$. We verify using Magma that \eqref{eq:2all1_2inverse} provides a left inverse for $\Phi$, so $\Phi$ is injective.

Suppose $(x,y,z) \in U$ and $(\alpha_1,\alpha_2,\beta,c) = \Phi(x,y,z)$. It is easy to check that $f_c^2(\beta) = \beta$ and $f_c^3(\alpha_i) = f_c^2(\alpha_i)$ for each $i \in \{1,2\}$, and that
	\begin{align*}
		f_c^2(\alpha_1) - f_c(\alpha_1) &= -\frac{x^2 + 3}{(x-1)(x+1)}, \\
		f_c^2(\alpha_2) - f_c(\alpha_2) &= \frac{3x^2 + 1}{(x-1)(x+1)}, \\
		f_c(\beta) - \beta &= -\frac{4x}{(x-1)(x+1)},
	\end{align*}
all of which are nonzero by \eqref{eq:2all1_2cusps}. Hence $\alpha_1$ and $\alpha_2$ are points of type $1_2$ and $\beta$ is a point of period 2 for $f_c$. Finally, we see that
	\[ f_c(\alpha_1) - f_c(\alpha_2) = \frac{2(x^2 + 1)}{(x-1)(x+1)} \ne 0, \]
so $\alpha_1$ and $\alpha_2$ have disjoint orbits under $f_c$. Therefore $\Phi$ maps $U$ into $U_1(G_5)$.

It remains to show that $\Phi$ maps $U$ onto $U_1(G_5)$. Let $\alpha_1$ and $\alpha_2$ be points of type $1_2$ for $f_c$ with disjoint orbits, and let $\beta$ be a point of period 2 for $f_c$. Since $f_c(\alpha_1)$ and $f_c(\alpha_2)$ are distinct points of type $1_1$, we know that $f_c(\alpha_1)$ and $f_c(\alpha_2)$ must be the negatives of distinct fixed points for $f_c$. Therefore, by Proposition~\ref{prop:1or2or3}, there exist $r$ and $s$ such that
	\begin{equation}\label{eq:2all1_2curve_proof}
	f_c(\alpha_1) = -\left(\frac{1}{2} + r\right) , \ f_c(\alpha_2) = -\left(\frac{1}{2} - r\right) , \ \beta = -\frac{1}{2} + s , \ c = \frac{1}{4} - r^2 = -\frac{3}{4} - s^2.
	\end{equation}
Following the proof of \cite[Thm. 2]{poonen:1998}, we see that setting $x = (1 - r)/s$ yields
	\[ r = -\frac{x^2 + 1}{(x-1)(x+1)}  , \ s = \frac{2x}{(x-1)(x+1)} , \ c = -\frac{(x^2 + 3)(3x^2 + 1)}{4(x-1)^2(x+1)^2}. \]
This gives the correct expression for $c$, and combining this with \eqref{eq:2all1_2curve_proof} yields
	\[ \alpha_1^2 = \frac{(5x^2 - 1)(x^2 + 3)}{4(x-1)^2(x+1)^2} , \ \alpha_2^2 = -\frac{(3x^2 + 1)(x^2 - 5)}{4(x-1)^2(x+1)^2} , \ \beta = -\frac{x^2 - 4x - 1}{2(x-1)(x+1)}. \]
Setting $y = 2\alpha_1(x-1)(x+1)$ and $z = 2\alpha_2(x-1)(x+1)$ gives a point $(x,y,z) \in U$ for which $(\alpha_1,\alpha_2,\beta,c) = \Phi(x,y,z)$, completing the proof.
\end{proof}

We now show that the curve $U_1(G_5)$ has no quadratic points. Theorem~\ref{thm:2and_all1_2} is an immediate consequence of the following result, which shows that the only points in $Y(\bbQ,2)$ lie outside the open set $U$ defined by \eqref{eq:2all1_2cusps}.

\begin{thm}\label{thm:2all1_2}
Let $Y$ be the affine curve of genus 5 defined by \eqref{eq:2all1_2curve}. Then
	\[ Y(\bbQ,2) = Y(\bbQ) = \{ (\pm 1, \pm 4, \pm 4) \}. \]
\end{thm}

\begin{proof}
Let
	\begin{align*}
		f(x) &:= (5x^2 - 1)(x^2 + 3),\\
		g(x) &:= -(3x^2 + 1)(x^2 - 5),
	\end{align*}
and let $C_1$ and $C_2$ be the affine curves defined by $y^2 = f(x)$ and $z^2 = g(x)$, respectively. We observe that $C_1$ is birational to $C_2$ via the map $(x,y) \mapsto (1/x,y/x^4)$, and a computation in Magma verifies that both curves are birational to $X_1^{\Ell}(15)$ (labeled 15A8 in \cite{cremona:1997}), which has only four rational points. Hence $(\pm 1, \pm 4)$ are the only four rational points on $C_1$ (resp., $C_2$), and therefore $Y(\bbQ)$ is as stated in the theorem.

For the sake of contradiction, suppose $(x,y,z)$ is a quadratic point on $Y$. We consider separately the cases $x \in \bbQ$ and $x \not \in \bbQ$.

\textbf{Case 1:} $x \in \bbQ$. First, we note that $x$ cannot be equal to $\pm 1$, since $(x,y,z)$ would then be a rational point on $Y$. Moreover, if $x \ne \pm 1$, then $y,z \not \in \bbQ$, though certainly $y^2,z^2 \in \bbQ$. We may therefore write $y = u\sqrt{d}, z = v\sqrt{d}$ for some $u,v \in \bbQ$ and $d \ne 1$ squarefree. We then rewrite \eqref{eq:2all1_2curve} as
	\begin{equation}\label{eq:2all1_2twists}
		\begin{cases}
			du^2 &= f(x)\\
			dv^2 &= g(x).
		\end{cases}
	\end{equation}
If $p$ is a prime dividing $d$, then arguing in the usual way we find that $\ord_p(x)$, $\ord_p(y)$, and $\ord_p(z)$ are nonnegative, so we may reduce modulo $p$ to find
	\[ f(x) \equiv g(x) \equiv 0 \mod p. \]
Hence $p$ divides $\Res(f,g) = 2^{24} \cdot 3^2$, so we have $p \in \{2,3\}$, and therefore $d \in \{-1, \pm 2, \pm 3 , \pm 6\}$. The system \eqref{eq:2all1_2twists} has no 2-adic solutions for $d \in \{-1,\pm 2, 3, \pm 6\}$, so we need only consider the case $d = -3$.

Let $C_1^{(-3)}$ denote the twist of $C_1$ by $-3$. This curve is birational to curve 45A1 in \cite{cremona:1997}, which has only two rational points. Therefore, the only rational points on $C_1^{(-3)}$ are the points $(0,\pm 1)$. However, when $x = 0$, we have $-3v^2 = 5$, which implies $v \not \in \bbQ$. We conclude, therefore, that $Y$ has no quadratic points with $x \in \bbQ$.

\textbf{Case 2:} $x \not \in \bbQ$. In order to apply Lemma~\ref{lem:ECquad_pts}, we require cubic models for each of genus one curves $C_1$ and $C_2$. As mentioned above, $C_1$ is birational to $C_2$, so we will use the same cubic model for both. We find that $C_1$ and $C_2$ are birational to the elliptic curve $E$ given by the equation
	\[ Y^2 = X^3 + \frac{5}{4}X^2 + \frac{1}{2}X + \frac{1}{4}.\]
Since $E$ is birational to $X^{\Ell}_1(15)$, $E$ has only four rational points --- $\infty$, $(0,\pm 1/2)$, and $(-1,0)$.
The birational maps between $C_1$ and $E$ are given by

	\begin{align}
		X = \frac{x^2 - 4x - 1 - y}{2(x+1)^2}  &, \ Y = \frac{5x^3 - 7x^2 + 7x + 3 - 3xy + y}{4(x+1)^3} \label{eq:2all1_2cubic1} \ ;\\
		x = \frac{X - 1 + 2Y}{3X + 1 - 2Y}  &, \ y = -2(x+1)^2X + x^2 - 4x - 1, \label{eq:2all1_2cubic2}
	\end{align}
and the birational maps between $C_2$ and $E$ are given by
	\begin{align}
		X' = -\frac{x^2 - 4x - 1 - z}{2(x-1)^2}  &, \ Y' = \frac{3x^3 - 7x^2 - 7x - 5 - xz - 3z}{4(x-1)^3} \label{eq:2all1_2cubic3} \ ;\\
		x = -\frac{3X' + 1 - 2Y'}{X' - 1 + 2Y'}  &, \ z = 2(x-1)^2X' + x^2 - 4x - 1. \label{eq:2all1_2cubic4}
	\end{align}
	
We must consider separately the different cases depending on whether $X$ and $X'$ are rational. The method is the same as that used in the proofs of Theorems~\ref{thm:1_2and1_3curve_pts} and \ref{thm:1and2_3quad_pts}, so we only summarize the main steps and omit the details of the computations, which may be found in \cite[\textsection 5.2]{doyle:thesis}.

If $X \in \bbQ$, then substituting $y = -2(x+1)^2X + x^2 - 4x - 1$ into the equation $y^2 - f(x)$ yields the following minimal polynomial of $x$:
	\[
		p_1(t) = t^2 + \frac{2X(X+2)}{X^2 - X - 1}t + \frac{X^2 + X + 1}{X^2 - X - 1}.
	\]
On the other hand, if $X \not \in \bbQ$, then Lemma~\ref{lem:ECquad_pts} says that there exists $(X_0,Y_0) \in E(\bbQ)$ and $v \in \bbQ$ for which 
	\[
		\begin{gathered}
			X^2 + (X_0 - v^2 + 5/4)X + (X_0^2 + v^2X_0 + 5/4X_0 - 2Y_0v + 1/2),\\
			Y = Y_0 + v(X - X_0).
		\end{gathered}
	\]
For each point $(X_0,Y_0)$, combining these equations with \eqref{eq:2all1_2cubic2} yields an expression for the minimal polynomial $p_1(t)$ of $x$ in terms of $v$.

One performs a similar computation in the cases $X' \in \bbQ$ and $X' \not \in \bbQ$, and doing so yields a short list of polynomials $p_2(t)$ among which the minimal polynomial of $x$ must be found. For each pair of polynomials $(p_1(t),p_2(t))$ arising in this way, one equates coefficients of $p_1$ and $p_2$ to obtain a system of equations. In all cases but one, the resulting system of equations defines a zero-dimensional scheme, and each rational point on the given scheme yields a polynomial $p_1(t) = p_2(t)$ that is reducible over $\bbQ$. In the remaining case, the system consists of a single equation that defines a genus one curve with Mordell-Weil rank zero (and small conductor), so all rational points can easily be found; once again, all such rational points yield an expression for $p_1(t) = p_2(t)$ that is reducible over $\bbQ$. It follows that $Y$ has no quadratic points, completing the proof of the theorem.
\end{proof}

\subsubsection{The graph $G_6$}

\begin{figure}[h]
\centering
	\begin{overpic}[scale=.5]{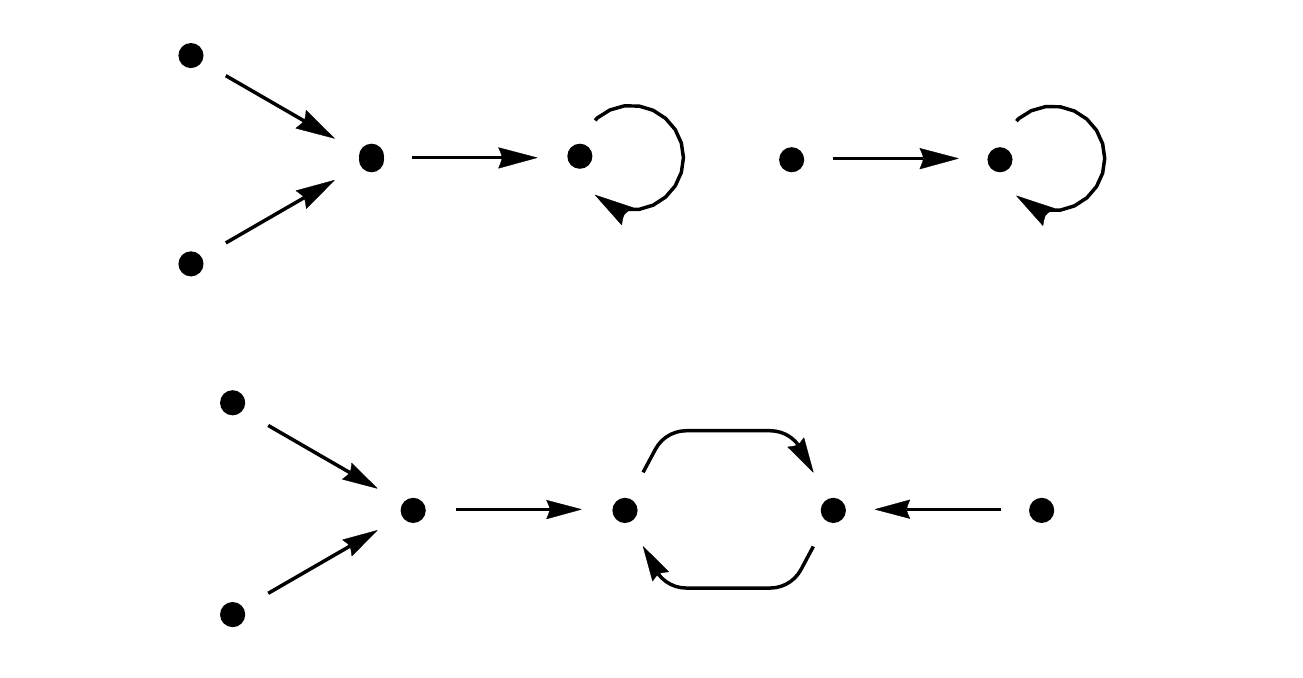}
		\put(17,88){$\alpha$}
		\put(135,79){$\alpha'$}
		\put(23,37){$\beta$}
	\end{overpic}
\caption{The graph $G_6$ minimally generated by a point $\alpha$ of type $1_2$, a fixed point $\alpha'$, and a point $\beta$ of type $2_2$}
\label{fig:1_2and2_2}
\end{figure}

Let $G_6$ be the graph appearing in Figure~\ref{fig:1_2and2_2}. The main result of this section is the following:

\begin{thm}\label{thm:1_2and2_2}
Let $K$ be a quadratic field and let $c \in K$. Then $G(f_c,K)$ does not contain a subgraph isomorphic to $G_6$.
\end{thm}

Because the generator $\alpha'$ is a fixed point, it suffices by Remark~\ref{rem:redundant} to work with the curve
	\[ U_1(G_6') = \{ (\alpha,\beta,c) \in \bbA^3 : \mbox{$\alpha$ is of type $1_2$ and $\beta$ is of type $2_2$ for $f_c$} \}, \] 
where $G_6'$ is the admissible subgraph of $G_6$ generated by $\alpha$ and $\beta$.

\begin{prop}\label{prop:1_2and2_2curve}
Let $Y$ be the affine curve of genus 5 defined by the equation
	\begin{equation}\label{eq:1_2and2_2curve}
		\begin{cases}
		y^2 &= (5q^2 - 1)(q^2 + 3)\\
		z^2 &= 5q^4 - 8q^3 + 6q^2 + 8q + 5,
		\end{cases}
	\end{equation}
and let $U$ be the open subset of $Y$ defined by
	\begin{equation}\label{eq:1_2and2_2cusps}
	q(q-1)(q+1)(q^2 + 3)(q^2 - 4q - 1) \ne 0.
	\end{equation}
Consider the morphism $\Phi : U \to \bbA^3$, $(q,y,z) \mapsto (\alpha,\beta,c)$, given by
	\[ \alpha = \frac{y}{2(q-1)(q+1)} , \ \beta = \frac{z}{2(q-1)(q+1)} , \ c = -\frac{(q^2 + 3)(3q^2 + 1)}{4(q-1)^2(q+1)^2}.\]
Then $\Phi$ maps $U$ isomorphically onto $U_1(G_6')$, with the inverse map given by
	\begin{equation}\label{eq:1_2and2_2inverse}
	q = -\frac{3 + 2f_c(\alpha)}{1 + 2f_c^2(\beta)} , \ y = 2\alpha(q-1)(q+1) , \ z = 2\beta(q-1)(q+1).
	\end{equation}
\end{prop}

\begin{proof}
The condition $(q-1)(q+1) \ne 0$ shows that $\Phi$ is well-defined, and since \eqref{eq:1_2and2_2inverse} provides a left inverse to $\Phi$, $\Phi$ must be injective.

Now let $(\alpha,\beta,c) = \Phi(q,y,z)$ for some point $(q,y,z)$ on $U$. Then a simple computation verifies that $f_c^3(\alpha) = f_c^2(\alpha)$, so $f_c^2(\alpha)$ is a fixed point, and
	
	\[ f_c^2(\alpha) - f_c(\alpha) = -\frac{q^2 + 3}{(q-1)(q+1)} \ne 0, \]
so $f_c(\alpha)$ is not a fixed point. Hence $\alpha$ is of type $1_2$.

One can also verify that $f_c^4(\beta) = f_c^2(\beta)$ and
	\[ f_c^3(\beta) - f_c^2(\beta) = \frac{4q}{(q-1)(q+1)} \ne 0, \]
so $f_c^2(\beta)$ is a point of period 2, and
	\[ f_c^3(\beta) - f_c(\beta) = - \frac{q^2 - 4q - 1}{(q-1)(q+1)} \ne 0,\]
so $f_c(\beta)$ is not a point of period 2. Therefore $\beta$ is of type $2_2$. It follows that $\Phi$ maps $U$ into $U_1(G_6')$.

Finally, suppose $(\alpha,\beta,c)$ lies on $U_1(G_6')$. By Proposition~\ref{prop:1and2etc}(A), we may write
	\begin{equation}\label{eq:c,r,s_by_q}
		c = -\frac{(q^2 + 3)(3q^2 + 1)}{4(q-1)^2(q+1)^2} , \ r = -\frac{q^2 + 1}{(q-1)(q+1)} , \ s = \frac{2q}{(q-1)(q+1)},
	\end{equation}
where $r$ and $s$ are the parameters appearing in Proposition~\ref{prop:1or2or3}. Since $\alpha$ and $\beta$ are of types $1_2$ and $2_2$, respectively, for $f_c$, then $f_c(\alpha)$ (resp. $f_c(\beta)$) is a point of type $1_1$ (resp. $2_1$). This means that $f_c(\alpha)$ is the negative of a fixed point, and $f_c(\beta)$ is the negative of a point of period 2. We may therefore assume without loss of generality that
	\[ \alpha^2 + c = -\left( \frac{1}{2} + r\right) , \ \beta^2 + c = -\left( -\frac{1}{2} + s\right). \]
Substituting the expressions for $c$, $r$, and $s$ from \eqref{eq:c,r,s_by_q} gives the equations
	\[
		\begin{cases}
		\alpha^2 &= \dfrac{\left(5 q^2-1\right)\left(q^2+3\right)}{4 (q-1)^2 (q+1)^2}\\
		\\
		\beta^2 &= \dfrac{5 q^4-8 q^3+6 q^2+8 q+5}{4 (q-1)^2 (q+1)^2}.
		\end{cases}
	\]
Finally, setting $y = 2\alpha(q-1)(q+1)$ and $z = 2\beta(q-1)(q+1)$ gives us \eqref{eq:1_2and2_2curve}.
\end{proof}

Theorem~\ref{thm:1_2and2_2} is an immediate consequence of the following result, which says that all elements of $Y(\bbQ,2)$ lie outside the open subset $U$ defined by \eqref{eq:1_2and2_2cusps}.

\begin{thm}\label{thm:1_2and2_2curve_pts}
Let $Y$ be the affine curve of genus 5 defined by \eqref{eq:1_2and2_2curve}. Then
	\[ Y(\bbQ,2) = Y(\bbQ) = \{ (\pm 1, \pm 4, \pm 4) \}. \]
\end{thm}

\begin{proof}
Let
	\begin{align*}
		f(q) &:= (5q^2 - 1)(q^2 + 3)\\
		g(q) &:= 5q^4 - 8q^3 + 6q^2 + 8q + 5,
	\end{align*}
and let $C_1$ and $C_2$ be the affine curves defined by $y^2 = f(q)$ and $z^2 = g(q)$, respectively. As noted in \cite[pp. 21--22]{poonen:1998}, the curves $C_1$ and $C_2$ are birational to elliptic curves 15A8 and 17A4, respectively, in \cite{cremona:1997}. (As mentioned previously, curve 15A8 is isomorphic to $\Xell_1(15)$.) Each of these two elliptic curves has four rational points, so it follows that the only rational points on $C_1$ (resp., $C_2$) are $(\pm 1, \pm 4)$. Therefore the only rational points on $Y$ are the eight points listed in the theorem. It remains to show that $Y$ admits no quadratic points. As usual, we separate the cases $q \in \bbQ$ and $q \not \in \bbQ$.

\textbf{Case 1:} $q \in \bbQ$. We already know by the previous paragraph that we cannot have $q = \pm 1$; moreover, if $q \in \bbQ \setminus \{\pm 1\}$, then $y,z \not \in \bbQ$ and $y^2,z^2 \in \bbQ$. Thus $y$ and $z$ generate the same quadratic field $K = \bbQ(\sqrt{d})$ with $d \ne 1$ squarefree. Writing $y = u\sqrt{d}$ and $z = v\sqrt{d}$ with $u,v \in \bbQ$, we rewrite \eqref{eq:1_2and2_2curve} as
	\begin{equation}\label{eq:1_2and2_2twist}
	\begin{cases}
		du^2 &= f(q)\\
		dv^2 &= g(q).
	\end{cases}
	\end{equation}
If $p \ne 5$ is a prime dividing $d$, then each of $q$, $u$, and $v$ is integral at $p$. For such $p$, we find that
	\[ f(q) \equiv g(q) \equiv 0 \mod p. \]
Therefore $p$ divides $\Res(f,g)  = 2^{24} \cdot 5$; hence the only primes that may divide $d$ are $2$ and $5$, so $d \in \{-1, \pm 2, \pm 5, \pm 10\}$. Since $g(q)$ is positive for all real $q$, we must have $d \in \{2,5,10\}$. Furthermore, for $d \in \{2,10\}$, \eqref{eq:1_2and2_2twist} has no 2-adic solutions. We are left, therefore, with the case $d = 5$.

Let $C_1^{(5)}$ denote the twist of $C_1$ by $d = 5$. A Magma computation verifies that $C_1^{(5)}$ is birational to the elliptic curve labeled 75B1 in \cite{cremona:1997}, which has only two rational points. The projective closure of $C_1^{(5)}$ has a singular point at infinity which splits into two rational points upon desingularization, so the affine curve $C_1^{(5)}$ has no rational points. We conclude that $Y$ has no quadratic points with $q \in \bbQ$.

\textbf{Case 2:} $q \not \in \bbQ$. This case is handled in the same way as case 2 of the proof of Theorem~\ref{thm:2all1_2}, and we refer the reader to \cite[\textsection 5.5]{doyle:thesis} for the details of the computation. The conclusion in this case is also the same as in the proof of Theorem~\ref{thm:2all1_2}: the curve $Y$ has no quadratic points $(q,y,z)$ with $q \not \in \bbQ$, and therefore $Y$ has no quadratic points.
\end{proof}

\subsection{Periods 1 and 3}

\begin{figure}[h]
	\includegraphics[scale=.58]{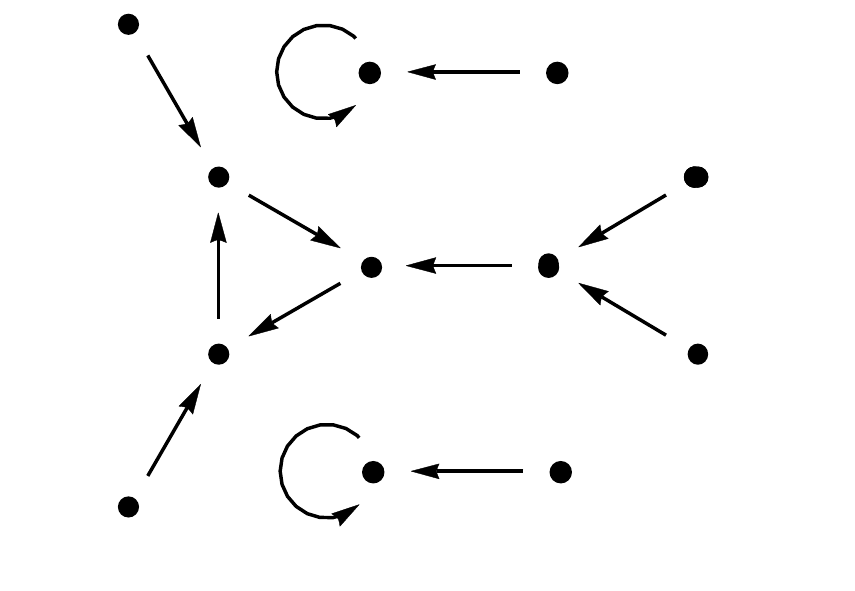}
\caption{The graph $G_7$}
\label{fig:1and3_2}
\end{figure}

\begin{figure}[h]
\centering
	\begin{overpic}[scale=.5]{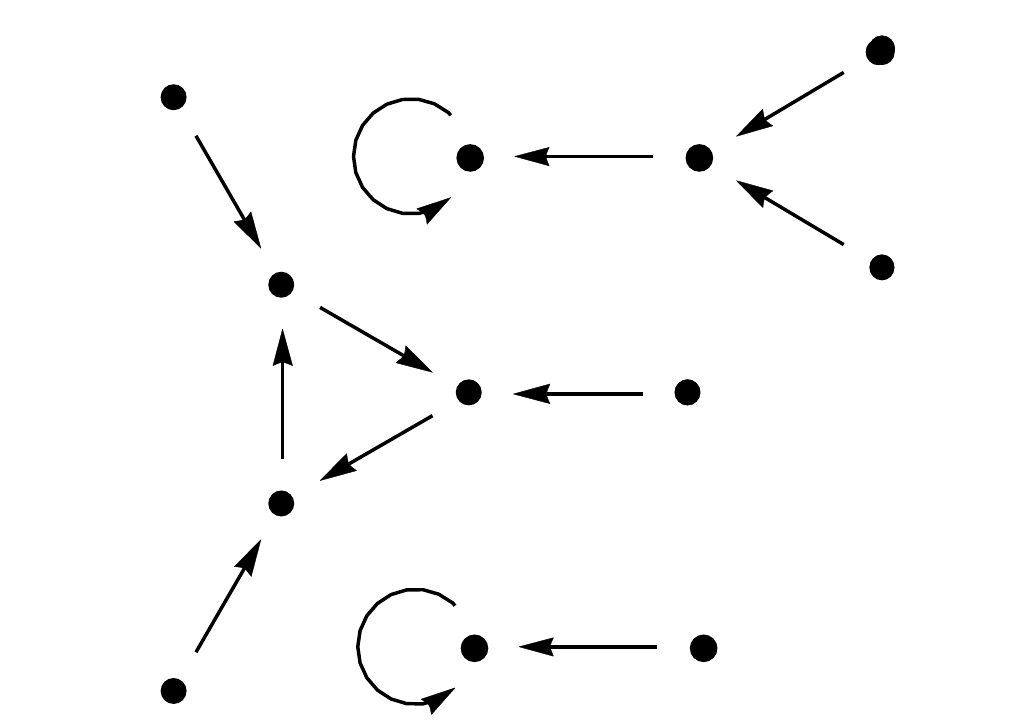}
		\put(133,100){$\alpha$}
		\put(70,17){$\alpha'$}
		\put(30,30){$\beta$}
	\end{overpic}
\caption{The graph $G_8$ minimally generated by a point $\alpha$ of type $1_2$, a fixed point $\alpha'$, and a point $\beta$ of period 3}
\label{fig:1_2and3}
\end{figure}

Among admissible graphs with cycle structure (1,1,3), only 10(3,1,1) has ten vertices, and only $G_7$ and $G_8$ have twelve vertices, where $G_7$ and $G_8$ are shown in Figures~\ref{fig:1and3_2} and \ref{fig:1_2and3}, respectively.

\begin{prop}\label{prop:(1,3)}
Let $(K,c)$ be a quadratic pair such that $G(f_c,K)$ is admissible with cycle structure (1,1,3). Then $G(f_c,K)$ appears in Appendix~\ref{app:graphs}.
\end{prop}

In \cite[\textsection 3.19]{doyle/faber/krumm:2014}, we show that $(K,c) = (\bbQ(\sqrt{33}),-29/16)$ is the only quadratic pair for which $G(f_c,K)$ contains $G_7$, and for that pair we actually have $G(f_c,K) \cong 14(3,1,1)$. Therefore, to prove Proposition~\ref{prop:(1,3)} it remains to prove the following:

\begin{thm}\label{thm:1_2and3}
Let $K$ be a quadratic field and let $c \in K$. Then $G(f_c,K)$ does not contain a subgraph isomorphic to $G_8$.
\end{thm}

Just as we did for $G_4$ and $G_6$, we can effectively ignore the fixed point $\alpha'$, looking instead at the graph $G_8'$ generated by $\alpha$ and $\beta$. We therefore consider
	\begin{align*}
	U_1(G_8') = \{(\alpha,\beta,c) \in \bbA^3 : \mbox{ $\alpha$ is of type $1_2$ and $\beta$ has period 3 for $f_c$} \}.
	\end{align*}

\begin{prop}\label{prop:1_2and3curve}
Let $Y$ be the affine curve of genus 9 defined by the equation
	\begin{equation}\label{eq:1_2and3curve}
	\begin{cases}
		y^2 &= t^6 + 2t^5 + 5t^4 + 10t^3 + 10t^2 + 4t + 1\\
		z^2 &= t^6+2 t^5+2 t^4+4 t^3+7 t^2+4 t+1 - 2t(t+1)y,
	\end{cases}
	\end{equation}
and let $U$ be the open subset of $Y$ defined by
	\begin{equation}\label{eq:1_2and3cusps}
		t(t+1)(t^2 + t + 1)(y + t^2 + t) \ne 0.
	\end{equation}
Consider the morphism $\Phi: U \to \bbA^3$, $(t,y,z) \mapsto (\alpha,\beta,c)$, given by
	\[ \alpha = \frac{z}{2 t (t+1)} , \ \beta = \frac{t^3+2 t^2+t+1}{2 t (t+1)} , \ c = -\frac{t^6+2 t^5+4 t^4+8 t^3+9 t^2+4 t+1}{4 t^2 (t+1)^2}.\]
Then $\Phi$ maps $U$ isomorphically onto $U_1(G_8')$, with the inverse map given by
	\begin{equation}\label{eq:1_2and3inverse}
		t = \beta + f_c(\beta) , \ y = -t(t+1)(2f_c(\alpha) + 1) , \ z = 2t(t+1)\alpha.
	\end{equation}
\end{prop}

\begin{proof}
The condition $t(t+1) \ne 0$ implies that $\Phi$ is well-defined, and one can verify that the relations in \eqref{eq:1_2and3inverse} provide a left inverse to $\Phi$, so $\Phi$ is injective.

Let $(t,y,z)$ lie in $U$. One can check that $f_c^3(\alpha) = f_c^2(\alpha)$ and $f_c^3(\beta) = \beta$; furthermore,
	\[ f_c(\alpha) - f_c^2(\alpha) = -\frac{y + t^2 + t}{t(t+1)} , \ f_c(\beta) - \beta = \frac{t^2 + t + 1}{t(t+1)}, \]
which are nonzero by hypothesis, so $\alpha$ and $\beta$ are points of type $1_2$ and period 3, respectively. Hence $\Phi$ maps $U$ into $U_1(G_8')$.

To prove surjectivity, suppose $(\alpha,\beta,c)$ lies on $U_1(G_8')$. Since $\alpha$ is of type $1_2$, $f_c(\alpha)$ is of type $1_1$, which means that $-f_c(\alpha)$ is a fixed point for $f_c$. By Proposition~\ref{prop:1and2etc}(B), there exists $(t,y)$ satisfying the first equation in \eqref{eq:1_2and3curve} for which
	\[
		-f_c(\alpha) = \frac{t^2+t+y}{2 t (t+1)} , \ \beta = \frac{t^3+2 t^2+t+1}{2 t (t+1)} , \ c = -\frac{t^6+2 t^5+4 t^4+8 t^3+9 t^2+4 t+1}{4 t^2 (t+1)^2}.
	\]
Writing $f_c(\alpha) = \alpha^2 + c$ yields
	\begin{align*}
	\alpha^2 &= -\frac{t^2+t+y}{2 t (t+1)} + \frac{t^6+2 t^5+4 t^4+8 t^3+9 t^2+4 t+1}{4 t^2 (t+1)^2} \\
		&= \frac{t^6+2 t^5+2 t^4+4 t^3+7 t^2+4 t+1 - 2t(t+1)y}{4 t^2 (t+1)^2}.
	\end{align*}
Setting $z := 2t(t+1)\alpha$ completes the proof.
\end{proof}

\begin{thm}\label{thm:1_2and3pts}
Let $Y$ be the affine curve of genus 9 defined by \eqref{eq:1_2and3curve}.
Then
	\[ Y(\bbQ,2) = Y(\bbQ) = \{ (0,\pm 1,\pm 1) , (-1, \pm 1, \pm 1 )\}. \]
\end{thm}

Theorem~\ref{thm:1_2and3} now follows from the fact that $Y(\bbQ,2)$ is disjoint from the open subset $U \subset Y$ defined in \eqref{eq:1_2and3cusps}. In order to prove Theorem~\ref{thm:1_2and3pts}, we require the following lemma, which determines the set of rational points on a certain auxiliary curve.

\begin{lem}\label{lem:1_2and3pts}
Let $C$ denote the hyperelliptic curve of genus 5 given by the equation
	\begin{equation}\label{eq:1_2and3norms}
		w^2 = (t^3 - 3t - 1)(t^3 + 2t^2 - t - 1)(t^6 + 2t^5 + 4t^4 + 8t^3 + 9t^2 + 4t + 1).
	\end{equation}
Then
	\[ C(\bbQ) = \{ (0,\pm 1), (-1,\pm 1), \infty^{\pm} \}. \]
\end{lem}

\begin{proof}
Let $J$ denote the Jacobian of $C$. Applying Magma's \texttt{RankBound} function, we find that $\rk J(\bbQ) \le 2$, so we may apply Theorem~\ref{thm:stoll} to bound the number of rational points on $C$. The curve $C$ has good reduction at $p = 5$, and a simple computation verifies that $\#C(\bbF_5) = 6$. By Theorem~\ref{thm:stoll}, we have
	\[ \#C(\bbQ) \le \#C(\bbF_5) + 2 \cdot \rk J(\bbQ) + \left\lfloor \frac{2 \cdot \rk J(\bbQ)}{3} \right\rfloor \le 11. \]
We will now show that $\#C(\bbQ)$ is divisible by six, which will then imply that $\#C(\bbQ) = 6$ and, therefore, $C(\bbQ)$ is precisely the set listed in the lemma.	

Considering the remark following Proposition~\ref{prop:1or2or3}, one might expect an automorphism of $C$ that takes $t \mapsto -(t+1)/t$. We find that
	\[ \sigma(t,w) := \left( -\frac{t+1}{t} , -\frac{w}{t^6} \right) \]
is an automorphism of $C$ of order six, which we verify by describing $\sigma^k$ for each $k = 1, \ldots, 6$:
	\begin{align*}
		\sigma : (t,w) &\mapsto \left( -\frac{t+1}{t} , -\frac{w}{t^6} \right);\\
		\sigma^2 : (t,w) &\mapsto \left( -\frac{1}{t+1} , \frac{w}{(t+1)^6} \right);\\
		\iota = \sigma^3 : (t,w) &\mapsto \left( t, - w \right);\\
		\sigma^4 : (t,w) &\mapsto \left( -\frac{t+1}{t} , \frac{w}{t^6} \right);\\	
		\sigma^5 : (t,w) &\mapsto \left( -\frac{1}{t+1} , -\frac{w}{(t+1)^6} \right);\\		
		\id = \sigma^6 : (t,w) &\mapsto (t,w).
	\end{align*}

We now claim that no rational point may be fixed by $\sigma^k$ for any $k$, which implies that, for a rational point $(t,w)$, the set
	\[ \{\sigma^k(t,w) : k \in \{0,\ldots,5\}\} \]
contains six distinct points. In particular, this would show that rational points come in multiples of six.

If $(t,w)$ is fixed by any of $\sigma$, $\sigma^2$, $\sigma^4$, or $\sigma^5$, then comparing $t$-coordinates shows that $t^2 + t + 1 = 0$, so that $t \not \in \bbQ$. If $(t,w)$ is fixed by $\sigma^3 = \iota$, then $w = 0$. Since the cubic and sextic polynomials appearing in \eqref{eq:1_2and3norms} are irreducible over $\bbQ$, this implies that $t \not \in \bbQ$. Therefore no rational point may be fixed by these automorphisms, completing the proof of the lemma.
\end{proof}

\begin{proof}[Proof of Theorem~\ref{thm:1_2and3pts}]
By Lemma~\ref{lem:1and3/2and3curves}(A), the only rational solutions to the first equation in \eqref{eq:1_2and3curve} satisfy $t \in \{-1,0\}$, so the points listed are the only rational points on $Y$.

A simple calculation shows that if $t^2 + t + 1 = 0$, then $[\bbQ(z):\bbQ] > 2$. It therefore follows from Lemma~\ref{lem:1and3/2and3curves}(A) that if $(t,y,z)$ is a quadratic point on $Y$, then $t \in \bbQ \setminus \{-1,0\}$, $y \not \in \bbQ$, and $y^2 \in \bbQ$. Let $K$ be the quadratic field of definition of the point $(t,x,y)$. If we take norms of both sides of the second equation in \eqref{eq:1_2and3curve} and set $w := N_{K/\bbQ}(z) \in \bbQ$, we get
	\[ w^2 = (t^6+2 t^5+2 t^4+4 t^3+7 t^2+4 t+1)^2 - 4t^2(t+1)^2y^2. \]
Using \eqref{eq:1_2and3curve} to rewrite $y^2$ in terms of $t$, this equation is equivalent to
	\[
		w^2 = (t^3 - 3t - 1)(t^3 + 2t^2 - t - 1)(t^6 + 2t^5 + 4t^4 + 8t^3 + 9t^2 + 4t + 1).
	\]
Thus $(t,w)$ is a rational point on the curve $C$ defined by \eqref{eq:1_2and3norms}. By Lemma~\ref{lem:1_2and3pts}, the only finite rational points on $C$ have $t \in \{-1,0\}$. We have already shown that a quadratic point on $Y$ cannot have $t \in \{-1,0\}$, so $Y$ has no quadratic points.
\end{proof}

\subsection{Periods 2 and 3}

\begin{figure}[h]
	\includegraphics[scale=.55]{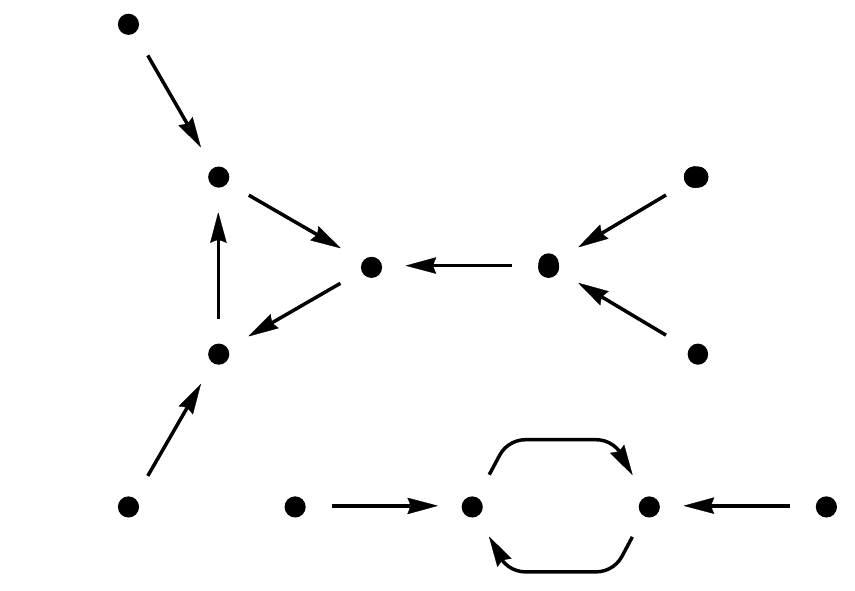}
\caption{The graph $G_9$}
\label{fig:2and3_2}
\end{figure}

\begin{figure}[h]
\centering
	\begin{overpic}[scale=.45]{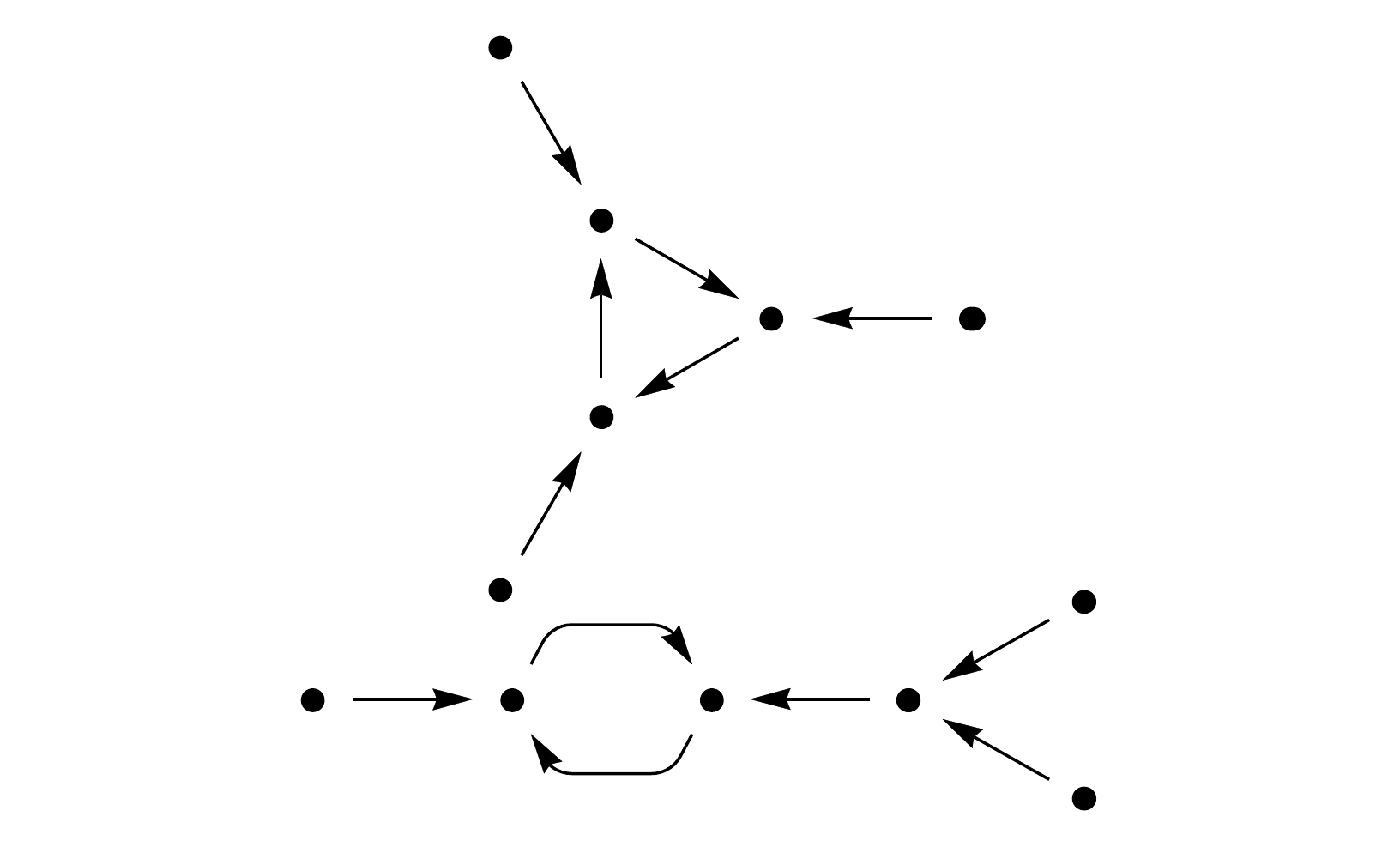}
		\put(168,39){$\alpha$}
		\put(81,66){$\beta$}
	\end{overpic}
\caption{The graph $G_{10}$ generated by a point $\alpha$ of type $2_2$ and a point $\beta$ of period 3}
\label{fig:2_2and3}
\end{figure}

Among admissible graphs with cycle structure (2,3), only 10(3,2) has ten vertices, and only $G_9$ and $G_{10}$ have twelve vertices, where $G_9$ and $G_{10}$ are shown in Figures~\ref{fig:2and3_2} and \ref{fig:2_2and3}, respectively.

\begin{prop}\label{prop:(2,3)}
Let $(K,c)$ be a quadratic pair such that $G(f_c,K)$ is admissible with cycle structure (2,3). Then $G(f_c,K)$ appears in Appendix~\ref{app:graphs}.
\end{prop}

In \cite[\textsection 3.20]{doyle/faber/krumm:2014}, we show that $(K,c) = (\bbQ(\sqrt{-17}),-29/16)$ is the only quadratic pair for which $G(f_c,K)$ contains $G_9$, and for that pair we actually have $G(f_c,K) \cong 14(3,2)$. Proposition~\ref{prop:(2,3)} is therefore a consequence of the following:

\begin{thm}\label{thm:2_2and3}
Let $K$ be a quadratic field and let $c \in K$. Then $G(f_c,K)$ does not contain a subgraph isomorphic to $G_{10}$.
\end{thm}

Since the graph $G_{10}$ is minimally generated by a point of type $2_2$ and a point of period 3, we write
	\begin{align*}
	U_1(G) = \{(\alpha,\beta,c) \in \bbA^3 : \mbox{ $\alpha$ is of type $2_2$ and $\beta$ has period 3 for $f_c$} \}.
	\end{align*}

\begin{prop}\label{prop:2_2and3curve}
Let $Y$ be the affine curve of genus 9 defined by the equation
	\begin{equation}\label{eq:2_2and3curve}
		\begin{cases}
			y^2 &= t^6 + 2t^5 + t^4 + 2t^3 + 6t^2 + 4t + 1\\
			z^2 &= t^6 + 2t^5 + 6t^4 + 12t^3 + 11t^2 + 4t + 1 - 2t(t+1)y,
		\end{cases}
	\end{equation}
and let $U$ be the open subset of $X$ defined by
	\begin{equation}\label{eq:2_2and3cusps}
		t(t+1)(t^2 + t + 1)y(y - t^2 - t) \ne 0.
	\end{equation}
Consider the morphism $\Phi: U \to \bbA^3$, $(t,y,z) \mapsto (\alpha,\beta,c)$, given by
	\[ \alpha = \frac{z}{2 t (t+1)} , \ \beta = \frac{t^3+2 t^2+t+1}{2 t (t+1)} , \ c = -\frac{t^6+2 t^5+4 t^4+8 t^3+9 t^2+4 t+1}{4 t^2 (t+1)^2}.\]
Then $\Phi$ maps $U$ isomorphically onto $U_1(G_{10})$, with the inverse map given by
	\begin{equation}\label{eq:2_2and3inverse}
		t = \beta + f_c(\beta) , \ y = -t(t+1)(2f_c(\alpha) - 1) , \ z = 2t(t+1)\alpha.
	\end{equation}
\end{prop}

\begin{proof}
The condition $t(t+1) \ne 0$ implies that $\Phi$ is well-defined. A Magma computation verifies that the relations given in \eqref{eq:2_2and3inverse} serve as a left inverse for $\Phi$, so we have that $\Phi$ is injective.

If $(t,y,z)$ lies on $U$, then one can verify that $f_c^4(\alpha) = f_c^2(\alpha)$ and $f_c^3(\beta) = \beta$, and that
	\[ f_c^2(\alpha) - f_c^3(\alpha) = -\frac{y}{t(t+1)} , \ f_c(\alpha) - f_c^3(\alpha) = -\frac{y - t^2 - t}{t(t+1)} , \ \beta - f_c(\beta) = \frac{t^2 + t + 1}{t(t+1)}, \]
which are all nonzero by hypothesis. Thus $\alpha$ and $\beta$ are points of type $2_2$ and period 3, respectively, for $f_c$, and therefore $\Phi$ maps $Y$ into $U_1(G_{10})$.

Finally, suppose $(\alpha,\beta,c)$ is a point on $U_1(G_{10})$. Since $f_c(\alpha)$ is a point of type $2_1$, we must have that $-f_c(\alpha)$ is a point of period 2. By Proposition~\ref{prop:1and2etc}, there exists a solution $(t,y)$ to the first equation in \eqref{eq:2_2and3curve} for which
	\[ -f_c(\alpha) = -\frac{t^2 + t - y}{2t(t+1)} , \ \beta = \frac{t^3 + 2t^2 + t + 1}{2t(t+1)} , \ c = -\frac{t^6+2 t^5+4 t^4+8 t^3+9 t^2+4 t+1}{4 t^2 (t+1)^2}. \]
Since $f_c(\alpha) = \alpha^2 + c$, we may write
	\begin{align*}
	\alpha^2 &= \frac{t^2 + t - y}{2t(t+1)} + \frac{t^6+2 t^5+4 t^4+8 t^3+9 t^2+4 t+1}{4 t^2 (t+1)^2}\\ 
		&= \frac{t^6+2 t^5+6 t^4+12 t^3+11 t^2+4 t+1 - 2t(t+1)y}{4 t^2 (t+1)^2}.
	\end{align*}
Setting $z := 2t(t+1)\alpha$ completes the proof.
\end{proof}

Theorem~\ref{thm:2_2and3} is now a direct consequence of the following:

\begin{thm}\label{thm:2_2and3curve_pts}
Let $Y$ be the affine curve of genus 9 defined by \eqref{eq:2_2and3curve}. Then
	\[ Y(\bbQ, 2) = Y(\bbQ) = \{ (0,\pm 1, \pm 1), (-1, \pm 1, \pm 1) \}. \]
\end{thm}

As in the previous section, we require a lemma that determines the complete set of rational points on an auxiliary curve.

\begin{lem}\label{lem:2_2and3pts}
Let $C$ denote the hyperelliptic curve of genus 5 given by the equation
	\begin{equation}\label{eq:2_2and3norms}
		w^2 = t^{12}+4 t^{11}+12 t^{10}+32 t^9+82 t^8+172 t^7+250 t^6+244 t^5+169 t^4+88 t^3+34 t^2+8 t+1.
	\end{equation}
Then
	\[ C(\bbQ) = \{ (0,\pm 1), (-1,\pm 1), \infty^{\pm} \}. \]
\end{lem}

\begin{proof}
Let $J$ denote the Jacobian of $C$. We apply Magma's \texttt{RankBound} function to see that $\rk J(\bbQ) \le 2$, so we may apply the method of Chabauty and Coleman. The curve $C$ has good reduction at $p = 11$, and we find that $\#C(\bbF_{11}) = 6$, so Theorem~\ref{thm:stoll} implies that
	\[ \#C(\bbQ) \le \#C(\bbF_{11}) + 2 \cdot \rk J(\bbQ) \le 10. \]
Just as in the proof of Lemma~\ref{lem:1_2and3pts}, it will suffice to show that $\#C(\bbQ)$ is a multiple of six; motivated in the same way as in the proof of that lemma, we are led to consider the automorphism $\sigma \in \Aut(C)$ given by
	\[ \sigma(t,w) := \left( -\frac{t+1}{t} , -\frac{w}{t^6} \right). \]
We see that $\sigma$ has order six:
	\begin{align*}
		\sigma : (t,w) &\mapsto \left( -\frac{t+1}{t} , -\frac{w}{t^6} \right);\\
		\sigma^2 : (t,w) &\mapsto \left( -\frac{1}{t+1} , \frac{w}{(t+1)^6} \right);\\
		\iota = \sigma^3 : (t,w) &\mapsto \left( t, - w \right);\\
		\sigma^4 : (t,w) &\mapsto \left( -\frac{t+1}{t} , \frac{w}{t^6} \right);\\
		\sigma^5 : (t,w) &\mapsto \left( -\frac{1}{t+1} , -\frac{w}{(t+1)^6} \right);\\
		\id = \sigma^6 : (t,w) &\mapsto (t,w).
	\end{align*}
Proceeding as in the proof of Lemma~\ref{lem:1_2and3pts}, and noting that the degree 12 polynomial in \eqref{eq:2_2and3norms} is irreducible over $\bbQ$, we find that no rational point may be fixed by $\sigma^k$ for any $k$. Therefore, $6$ divides $\#C(\bbQ)$, completing the proof of the lemma.
\end{proof}

\begin{proof}[Proof of Theorem~\ref{thm:2_2and3curve_pts}]
By Lemma~\ref{lem:1and3/2and3curves}(B), the only rational solutions to the first equation of \eqref{eq:2_2and3curve} are $(-1,\pm 1) $ and $(0, \pm 1)$. Since $t \in \{-1,0\}$ also implies $z \in \{\pm 1\}$, we have found all rational points on $Y$.

It also follows from Lemma~\ref{lem:1and3/2and3curves}(B) that if $(t,y,z)$ is a quadratic point on $Y$, then $t \in \bbQ \setminus \{-1,0\}$, $y \not \in \bbQ$, and $y^2 \in \bbQ$. Let $K$ be the quadratic field generated by $y$. If we take norms of both sides of the second equation in \eqref{eq:2_2and3curve} and set $w := N_{K/\bbQ}(z) \in \bbQ$, we get
	\[ w^2 = (t^6 + 2t^5 + 6t^4 + 12t^3 + 11t^2 + 4t + 1)^2 - 4t^2(t+1)^2y^2. \]
Using \eqref{eq:2_2and3curve} to rewrite $y^2$ in terms of $t$, this equation is equivalent to
	\[
		w^2 = t^{12}+4 t^{11}+12 t^{10}+32 t^9+82 t^8+172 t^7+250 t^6+244 t^5+169 t^4+88 t^3+34 t^2+8 t+1.
	\]
We therefore have a rational point $(t,w)$ on the curve $C$ from Lemma~\ref{lem:2_2and3pts}. However, the only rational points on $C$ have $t \in \{-1,0\}$, and we have already shown that a quadratic point on $Y$ cannot have $t \in \{-1,0\}$. Therefore $Y$ has no quadratic points.
\end{proof}


\subsection{Proof of the main theorem}\label{sec:main_thm}
We can now prove our main result, which we restate here:

\begin{mainthm}
Let $K$ be a quadratic field, and let $c \in K$. Suppose $f_c$ does not admit $K$-rational points of period greater than four, and suppose that $G(f_c,K)$ is not isomorphic to one of the 46 graphs shown in Appendix~\ref{app:graphs}. Then $G(f_c,K)$	\begin{enumerate}
		\item \emph{properly} contains one of the following graphs from Appendix~\ref{app:graphs}:
			\[ \text{10(1,1)b, 10(2), 10(3)a, 10(3)b, 12(2,1,1)b, 12(4), 12(4,2); or} \]
		\item contains one of the following graphs (from Figures \ref{fig:1and4}, \ref{fig:graph12_2_minus2pts}, and \ref{fig:1and2_3}, respectively):
			\[ \text{$G_0$, $G_2$, $G_4$.} \]
	\end{enumerate}
Moreover:
	\renewcommand{\labelenumi}{(\alph{enumi})}
	\begin{enumerate}
		\item There are at most four pairs $(K,c)$ for which $G(f_c,K)$ properly contains the graph 12(2,1,1)b, at most five pairs for which $G(f_c,K)$ properly contains 12(4), and at most one pair for which $G(f_c,K)$ properly contains 12(4,2). For each such pair we must have $c \in \bbQ$.
		\item There is at most one pair $(K,c)$ for which $G(f_c,K)$ contains $G_0$, at most four pairs for which $G(f_c,K)$ contains $G_2$, and at most three pairs for which $G(f_c,K)$ contains $G_4$. For each such pair we must have $c \in \bbQ$.
	\end{enumerate}
	\renewcommand{\labelenumi}{(\Alph{enumi})}
\end{mainthm}

\begin{proof}
By Lemma~\ref{lem:admissible}, the only way that $G(f_c,K)$ may be fail to be strongly admissible is if $c = 1/4$, in which case $f_c = f_{1/4}$ has a single fixed point, or $0$ is preperiodic for $f_c$. In \cite[\textsection 5]{doyle/faber/krumm:2014}, we found that if $(K,c)$ is a quadratic pair with a unique fixed point, then $G(f_c,K)$ is isomorphic to 2(1), 4(1), or 6(2,1), and if $(K,c)$ is a quadratic pair for which $0$ is preperiodic, then $G(f_c,K)$ is isomorphic to one of the graphs in Appendix~\ref{app:graphs} with an odd number of vertices. Therefore, if $(K,c)$ is a quadratic pair for which $G(f_c,K)$ does not appear in Appendix~\ref{app:graphs}, then $G(f_c,K)$ is strongly admissible.

Now suppose $(K,c)$ is a quadratic pair for which $G(f_c,K)$ is strongly admissible, and suppose further that $f_c$ admits no $K$-rational points of period greater than four. Combining Corollary~\ref{cor:number_of_cycles} and Theorems~\ref{thm:3and4} and \ref{thm:1and2and3}, the cycle structure of $G(f_c,K)$ must lie in the following list:
	\[
		(1,1), (2), (3), (4), (1,1,2), (1,1,3), (1,1,4), (2,3), (2,4), (1,1,2,4).
	\]
For each cycle structure $\tau$ listed above --- with the exception of $\tau = (1,1,2,4)$ --- a detailed study of graphs $G(f_c,K)$ with cycle structure $\tau$ may be found in Sections~\ref{sec:periodic} and \ref{sec:preper_pts}. In particular, for each such $\tau$, we have given necessary conditions for $(K,c)$ to be a quadratic pair for which $G(f_c,K)$ has cycle structure $\tau$: the relevant results are Theorems~\ref{thm:1and4} and \ref{thm:2and4} and Propositions~\ref{prop:(1)}, \ref{prop:(2)}, \ref{prop:(3)}, \ref{prop:(4)}, \ref{prop:(1,2)}, \ref{prop:(1,3)}, and \ref{prop:(2,3)}.  Since any graph with cycle structure (1,1,2,4) contains subgraphs with cycle structures (1,1,4) and (2,4), respectively, the compilation of these results from Sections~\ref{sec:periodic} and \ref{sec:preper_pts} completes the proof of Theorem~\ref{thm:main_thm}.
\end{proof}


\appendix
\section[Preperiodic graphs over quadratic extensions]{Preperiodic graphs over quadratic extensions}\label{app:graphs}

In this appendix, we provide a summary of the data obtained in the search for preperiodic graphs for quadratic polynomials over quadratic fields conducted in \cite{doyle/faber/krumm:2014}.

\subsection[List of known graphs]{List of known graphs}\label{sub:known_graphs}
We list here the 46 preperiodic graphs discovered in \cite{doyle/faber/krumm:2014}. The label of each graph is in the form $N(\ell_1,\ell_2,\ldots)$, where $N$ denotes the number of vertices in the graph and $\ell_1,\ell_2,\ldots$ are the lengths of the directed cycles in the graph in nonincreasing order. If more than one isomorphism class of graphs with this data was observed, we add a lowercase Roman letter to distinguish them. For example, the labels 5(1,1)a and 5(1,1)b correspond to the two isomorphism classes of graphs observed that have five vertices and two fixed points. In all figures below we omit the connected component corresponding to the point at infinity.


\begin{center}

\begin{tabularx}{\textwidth}{|L|L|L|} \hline
0 & 2(1) & 3(1,1)
\end{tabularx} \offinterlineskip
\begin{tabularx}{\textwidth}{|H|H|H|}
 & \pic{graph2_1} & \pic{graph3_11} \\ \hline
\end{tabularx} \offinterlineskip
\begin{tabularx}{\textwidth}{|L|L|L|}
3(2) & 4(1) & 4(1,1)
\end{tabularx} \offinterlineskip
\begin{tabularx}{\textwidth}{|H|H|H|}
\pic{graph3_2} & \pic{graph4_1} & \pic{graph4_11} \\ \hline
\end{tabularx} \offinterlineskip
\newpage
\begin{tabularx}{\textwidth}{|L|L|L|} \hline
4(2) & 5(1,1)a & 5(1,1)b \\
\end{tabularx} \offinterlineskip
\begin{tabularx}{\textwidth}{|H|H|H|}
\pic{graph4_2} & \pic{graph5_11a} & \pic{graph5_11b} \\ \hline
\end{tabularx} \offinterlineskip
\begin{tabularx}{\textwidth}{|M|M|}
5(2)a & 5(2)b
\end{tabularx} \offinterlineskip
\begin{tabularx}{\textwidth}{|W|W|}
\pic{graph5_2a} & \pic{graph5_2b} \\ \hline
\end{tabularx} \offinterlineskip
\begin{tabularx}{\textwidth}{|L|L|L|}
6(1,1) & 6(2) & 6(2,1)
\end{tabularx} \offinterlineskip
\begin{tabularx}{\textwidth}{|H|H|H|}
\pic{graph6_11} & \pic{graph6_2} & \pic{graph6_21} \\ \hline
\end{tabularx} \offinterlineskip
\begin{tabularx}{\textwidth}{|M|M|}
6(3) & 7(1,1)a
\end{tabularx} \offinterlineskip
\begin{tabularx}{\textwidth}{|W|W|}
\pic{graph6_3} & \pic{graph7_11a} \\ \hline
\end{tabularx} \offinterlineskip
\begin{tabularx}{\textwidth}{|M|M|}
7(1,1)b & 7(2,1,1)a
\end{tabularx} \offinterlineskip
\begin{tabularx}{\textwidth}{|W|W|}
\pic{graph7_11b} & \pic{graph7_211a} \\ \hline
\end{tabularx} \offinterlineskip
\begin{tabularx}{\textwidth}{|M|M|}
7(2,1,1)b & 8(1,1)a
\end{tabularx} \offinterlineskip
\begin{tabularx}{\textwidth}{|W|W|}
\pic{graph7_211b} & \pic{graph8_11a} \\ \hline
\end{tabularx} \offinterlineskip
\begin{tabularx}{\textwidth}{|M|M|}
8(1,1)b & 8(2)a
\end{tabularx} \offinterlineskip
\begin{tabularx}{\textwidth}{|W|W|}
\pic{graph8_11b} & \pic{graph8_2a} \\ \hline
\end{tabularx} \offinterlineskip
\begin{tabularx}{\textwidth}{|M|M|}
8(2)b & 8(2,1,1)
\end{tabularx} \offinterlineskip
\begin{tabularx}{\textwidth}{|W|W|}
\pic{graph8_2b} & \pic{graph8_211} \\ \hline
\end{tabularx} \offinterlineskip
\newpage
\begin{tabularx}{\textwidth}{|M|M|} \hline
8(3) & 8(4)
\end{tabularx} \offinterlineskip
\begin{tabularx}{\textwidth}{|W|W|}
\pic{graph8_3} & \pic{graph8_4} \\ \hline
\end{tabularx} \offinterlineskip
\begin{tabularx}{\textwidth}{|M|M|}
9(2,1,1) & 10(1,1)a
\end{tabularx} \offinterlineskip
\begin{tabularx}{\textwidth}{|W|W|}
\pic{graph9_211} & \pic{graph10_11a} \\ \hline
\end{tabularx} \offinterlineskip
\begin{tabularx}{\textwidth}{|M|M|}
10(1,1)b & 10(2)
\end{tabularx} \offinterlineskip
\begin{tabularx}{\textwidth}{|W|W|}
\pic{graph10_11b} & \includegraphics[scale=.58]{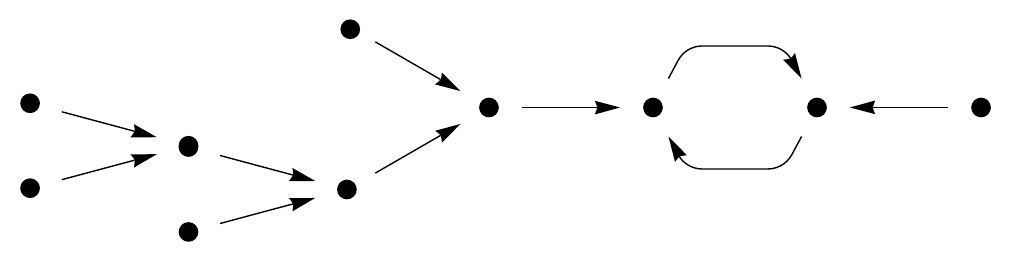} \\ \hline
\end{tabularx} \offinterlineskip
\begin{tabularx}{\textwidth}{|M|M|}
10(2,1,1)a & 10(2,1,1)b
\end{tabularx} \offinterlineskip
\begin{tabularx}{\textwidth}{|W|W|}
\pic{graph10_211a} & \pic{graph10_211b} \\ \hline
\end{tabularx} \offinterlineskip
\begin{tabularx}{\textwidth}{|L|L|L|}
10(3)a & 10(3)b & 10(3,1,1)
\end{tabularx} \offinterlineskip
\begin{tabularx}{\textwidth}{|H|H|H|}
\pic{graph10_3a} & \pic{graph10_3b} & \pic{graph10_311} \\ \hline
\end{tabularx} \offinterlineskip
\begin{tabularx}{\textwidth}{|M|M|}
10(3,2) & 12(2)
\end{tabularx} \offinterlineskip
\begin{tabularx}{\textwidth}{|W|W|}
\pic{graph10_32} & \includegraphics[scale=.58]{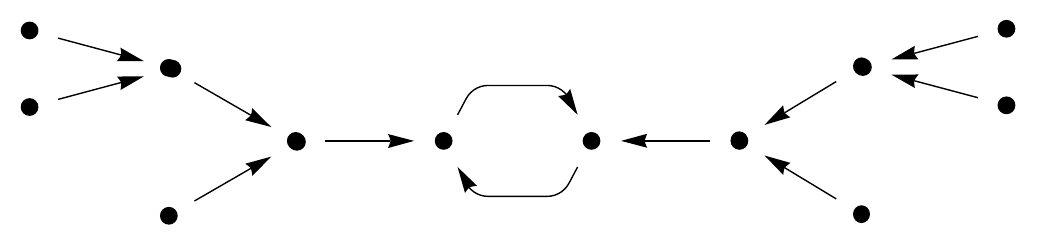} \\ \hline
\end{tabularx} \offinterlineskip
\newpage
\begin{tabularx}{\textwidth}{|M|M|} \hline
12(2,1,1)a & 12(2,1,1)b
\end{tabularx} \offinterlineskip
\begin{tabularx}{\textwidth}{|W|W|}
\pic{graph12_211a} & \pic{graph12_211b} \\ \hline
\end{tabularx} \offinterlineskip
\begin{tabularx}{\textwidth}{|L|L|L|}
12(3) & 12(4) & 12(4,2)
\end{tabularx} \offinterlineskip
\begin{tabularx}{\textwidth}{|H|H|H|}
\pic{graph12_3} & \pic{graph12_4} & \pic{graph12_42} \\ \hline
\end{tabularx} \offinterlineskip
\begin{tabularx}{\textwidth}{|M|M|}
12(6) & 14(2,1,1)
\end{tabularx} \offinterlineskip
\begin{tabularx}{\textwidth}{|W|W|}
\pic{graph12_6} & \pic{graph14_211} \\ \hline
\end{tabularx} \offinterlineskip
\begin{tabularx}{\textwidth}{|M|M|}
14(3,1,1) & 14(3,2)
\end{tabularx} \offinterlineskip
\begin{tabularx}{\textwidth}{|W|W|}
\pic{graph14_311} & \pic{graph14_32} \\ \hline
\end{tabularx} \offinterlineskip

\end{center}

\subsection[Representative data]{Representative data}
We give here a representative set of data for each graph in \textsection \ref{sub:known_graphs}. Each item in the list below includes the following information: 
\[
	K,\ p(t),\ c,\ \PrePer(f_c,K)'.
\]
Here $K = \bbQ(\sqrt{D})$ is a quadratic field over which this preperiodic structure was observed; $p(t)$ is a defining polynomial for $K$ with a root $g \in K$; $c$ is an element of $K$ such that the set $\PrePer(f_c,K) \setminus \{\infty\}$, when endowed with the structure of a directed graph, is isomorphic to the given graph; and $\PrePer(f_c,K)'$ is an abbreviated form of the full set of finite $K$-rational preperiodic points for $f_c$: since $x \in \PrePer(f_c, K)$ if and only if $-x \in \PrePer(f_c,K)$, we list only one of $x$ and $-x$ in the set $\PrePer(f_c,K)'$. We do not make explicit the correspondence between individual elements of this set and vertices of the graph. If a particular graph was observed over both real and imaginary quadratic fields, we give a representative set of data for each case.
\newpage
\begin{longtable}[l]{l l}
{\bf 0:} & $\bbQ(\sqrt{5}), \;\; t^2 - t - 1,\;\; 1, \;\; \emptyset$\\
\vspace{4mm}
	& $\bbQ(\sqrt{-3}), \;\; t^2 - t + 1,\;\; 2, \;\; \emptyset$\\
{\bf 2(1):} & $\bbQ(\sqrt{5}), \;\; t^2 - t - 1,\;\; \frac{1}{4}, \;\; \left\{\frac{1}{2}\right\}$\\
\vspace{4mm}
	& $\bbQ(\sqrt{-7}), \;\; t^2 - t + 2,\;\; \frac{1}{4}, \;\; \left\{\frac{1}{2}\right\}$
\\
{\bf 3(1,1):}
&  $\bbQ(\sqrt{5}), \;\; t^2 - t - 1,\;\; 0, \;\; \left\{0, 1\right\}$ \\
\vspace{4mm}
&  $\bbQ(\sqrt{-7}), \;\; t^2 - t + 2,\;\; 0, \;\; \left\{0, 1\right\}$
\\
{\bf 3(2):} &  $\bbQ(\sqrt{3}), \;\; t^2 - 3,\;\; -1, \;\; \left\{0, 1\right\}$\\
\vspace{4mm}
& $\bbQ(\sqrt{-3}), \;\; t^2 - t + 1,\;\; -1, \;\; \left\{0, 1\right\}$
\\
\vspace{4mm}
{\bf 4(1):} & $\bbQ(\sqrt{-3}), \;\;t^2 - t + 1,\;\; \frac{1}{4}, \;\; \left\{ \frac{1}{2},  g - \frac{1}{2}\right\}$
\\
{\bf 4(1,1):}
&  $\bbQ(\sqrt{5}), \;\; t^2 - t - 1,\;\; \frac{1}{5}, \;\; \left\{\frac{1}{5} g + \frac{2}{5}, \frac{1}{5} g - \frac{3}{5}\right\}$ \\
\vspace{4mm}
&  $\bbQ(\sqrt{-3}), \;\;t^2 - t + 1, \;\;1, \;\; \left\{ g,  g - 1\right\}$
\\
{\bf 4(2):}
&  $\bbQ(\sqrt{5}), \;\; t^2 - t - 1,\;\; -\frac{4}{5}, \;\; \left\{\frac{1}{5} g + \frac{2}{5}, \frac{1}{5} g - \frac{3}{5}\right\}$\\
\vspace{4mm}
&  $\bbQ(\sqrt{-3}), \;\;t^2 - t + 1, \;\;-\frac{2}{3}, \;\; \left\{\frac{1}{3} g - \frac{2}{3}, \frac{1}{3} g + \frac{1}{3}\right\}$
\\
{\bf 5(1,1)a:}
&  $\bbQ(\sqrt{13}), \;\; t^2 - t - 3,\;\; -2, \;\; \left\{0, 2, 1\right\}$ \\
\vspace{4mm}
&  $\bbQ(\sqrt{-3}), \;\; t^2 - t + 1,\;\; -2, \;\; \left\{0, 2, 1\right\}$
\\
\vspace{4mm}
{\bf 5(1,1)b:} & $\bbQ(\sqrt{-1}), \;\;t^2 + 1, \;\;0, \;\; \left\{0,  1,  g\right\}$
\\
\vspace{4mm}
{\bf 5(2)a:} & $\bbQ(\sqrt{-1}), \;\;t^2 + 1, \;\;g, \;\;\left\{0,  g, g - 1\right\}$
\\
\vspace{4mm}
{\bf 5(2)b:} & $\bbQ(\sqrt{2}), \;\;t^2 - 2, \;\;-1, \;\; \left\{0,  1,  g\right\}$
\\
{\bf 6(1,1):}
&  $\bbQ(\sqrt{5}), \;\; t^2 - t - 1,\;\; -\frac{3}{4}, \;\; \left\{\frac{1}{2},  g - \frac{1}{2}, \frac{3}{2}\right\}$ \\
\vspace{4mm}
&  $\bbQ(\sqrt{-3}), \;\; t^2 - t + 1,\;\; -\frac{3}{4}, \;\; \left\{\frac{1}{2}, \frac{3}{2}, g - \frac{1}{2}\right\}$
\\
{\bf 6(2):}
&  $\bbQ(\sqrt{5}), \;\;t^2 - t - 1, \;\;-3, \;\; \left\{ 1,  2, 2 g - 1\right\}$ \\
\vspace{4mm}
&  $\bbQ(\sqrt{-3}), \;\; t^2 - t + 1,\;\; -\frac{13}{9}, \;\; \left\{\frac{1}{3}, \frac{4}{3}, \frac{5}{3}\right\}$
\\
\vspace{4mm}
{\bf 6(2,1):} & $\bbQ(\sqrt{-1}), \;\;t^2 + 1, \;\;\frac{1}{4}, \;\;\left\{ \frac{1}{2},  g - \frac{1}{2}, g + \frac{1}{2}\right\}$
\\
{\bf 6(3):}
&  $\bbQ(\sqrt{33}), \;\; t^2 - t - 8,\;\; -\frac{301}{144}, \;\; \left\{\frac{5}{12}, \frac{19}{12}, \frac{23}{12}\right\}$\\
\vspace{4mm}
&  $\bbQ(\sqrt{-67}), \;\; t^2 - t + 17,\;\; -\frac{301}{144}, \;\; \left\{\frac{5}{12}, \frac{19}{12}, \frac{23}{12}\right\}$
\\
\vspace{4mm}
{\bf 7(1,1)a:} & $\bbQ(\sqrt{2}), \;\;t^2 - 2, \;\;-2, \;\; \left\{0,  1, 2,  g\right\}$
\\
\vspace{4mm}
{\bf 7(1,1)b:} & $\bbQ(\sqrt{3}), \;\;t^2 - 3, \;\;-2, \;\; \left\{0,  1, 2, g\right\}$
\\
\vspace{4mm}
{\bf 7(2,1,1)a:} & $\bbQ(\sqrt{-3}), \;\;t^2 - t + 1,\;\;0, \;\; \left\{0,  1,  g, g - 1\right\}$
\\
\vspace{4mm}
{\bf 7(2,1,1)b:} & $\bbQ(\sqrt{5}), \;\;t^2 - t - 1, \;\;-1, \;\; \left\{0,  1,  g, g - 1\right\}$
\\
{\bf 8(1,1)a:}
&  $\bbQ(\sqrt{13}), \;\; t^2 - t - 3,\;\; -\frac{289}{144}, \;\; \left\{\frac{5}{6} g + \frac{1}{12}, \frac{1}{2} g - \frac{13}{12}, \frac{1}{2} g + \frac{7}{12}, \frac{5}{6} g - \frac{11}{12}\right\}$\\
\vspace{4mm}
&  $\bbQ(\sqrt{-15}), \;\; t^2 - t + 4,\;\; -\frac{5}{16}, \;\; \left\{\frac{1}{4}, \frac{3}{4}, \frac{5}{4}, \frac{1}{2} g - \frac{1}{4}\right\}$
\\
{\bf 8(1,1)b:}
&  $\bbQ(\sqrt{13}), \;\; t^2 - t - 3,\;\; -\frac{40}{9}, \;\; \left\{\frac{4}{3}, \frac{8}{3}, \frac{5}{3}, \frac{4}{3} g - \frac{2}{3}\right\}$\\
\vspace{4mm}
&  $\bbQ(\sqrt{-2}), \;\; t^2 + 2,\;\; -\frac{10}{9}, \;\; \left\{\frac{2}{3}, \frac{1}{3} g, \frac{4}{3}, \frac{5}{3}\right\}$
\\
{\bf 8(2)a:}
&  $\bbQ(\sqrt{10}), \;\; t^2 - 10,\;\; -\frac{13}{9}, \;\; \left\{\frac{1}{3}, \frac{1}{3} g, \frac{4}{3}, \frac{5}{3}\right\}$\\
\vspace{4mm}
&  $\bbQ(\sqrt{-3}), \;\;t^2 - t + 1,\;\; -\frac{5}{12}, \;\; \left\{ \frac{2}{3} g - \frac{5}{6},  \frac{2}{3} g + \frac{1}{6}, \frac{1}{3} g  + \frac{5}{6},   \frac{1}{3} g - \frac{7}{6} \right\}$
\\
{\bf 8(2)b:}
&  $\bbQ(\sqrt{13}), \;\; t^2 - t - 3,\;\; -\frac{37}{9}, \;\; \left\{\frac{4}{3}, \frac{5}{3}, \frac{7}{3}, \frac{4}{3} g - \frac{2}{3}\right\}$\\
\vspace{4mm}
&  $\bbQ(\sqrt{-7}), \;\; t^2 -t + 2,\;\; -\frac{13}{16}, \;\;  \left\{\frac{1}{4}, \frac{3}{4}, \frac{1}{2} g - \frac{1}{4}, \frac{5}{4}\right\}$
\\
{\bf 8(2,1,1):}
&  $\bbQ(\sqrt{5}), \;\; t^2 - t - 1,\;\; -12, \;\; \left\{3, 3 g - 1, 3 g - 2, 4\right\}$\\
\vspace{4mm}
&  $\bbQ(\sqrt{-3}), \;\; t^2 - t + 1,\;\; \frac{7}{12}, \;\; \left\{\frac{2}{3} g + \frac{1}{6}, \frac{2}{3} g - \frac{5}{6}, \frac{4}{3} g - \frac{7}{6}, \frac{4}{3} g - \frac{1}{6}\right\}$
\\
{\bf 8(3):}
&  $\bbQ(\sqrt{5}), \;\; t^2 - t - 1,\;\; -\frac{29}{16}, \;\; \left\{\frac{1}{4}, \frac{5}{4}, \frac{3}{4}, \frac{7}{4}\right\}$\\
\vspace{4mm}
&  $\bbQ(\sqrt{-3}), \;\; t^2 -t + 1,\;\; -\frac{29}{16}, \;\; \left\{\frac{1}{4}, \frac{5}{4}, \frac{3}{4}, \frac{7}{4}\right\}$
\\
{\bf 8(4):}
&  $\bbQ(\sqrt{10}), \;\;t^2 - 10,\;\; -\frac{155}{72}, \;\; \left\{ \frac{1}{4} g - \frac{1}{6},   \frac{1}{4} g + \frac{1}{6},  \frac{1}{12} g - \frac{3}{2},   \frac{1}{12} g + \frac{3}{2}\right\}$\\
\vspace{4mm}
&  $\bbQ(\sqrt{-455}), \;\;t^2 - t + 114,\;\; \frac{199}{720},\;\;\left\{ \frac{1}{10}g+\frac{17}{60}, \frac{1}{15}g - \frac{47}{60},\frac{1}{10}g-\frac{23}{60},\frac{1}{15}g+\frac{43}{60}\right\}$
\\
\vspace{4mm}
{\bf 9(2,1,1):} & $\bbQ(\sqrt{5}), \;\;t^2 - t - 1, \;\;-2, \;\; \left\{0,  1, 2,  g,  g - 1\right\}$
\\
\vspace{4mm}
{\bf 10(1,1)a:} & $\bbQ(\sqrt{-7}), \;\;t^2 - t + 2,\;\; \frac{3}{16}, \;\; \left\{\frac{1}{4}, \frac{1}{2} g + \frac{1}{4}, \frac{1}{2} g - \frac{1}{4}, \frac{1}{2} g - \frac{3}{4}, \frac{3}{4}\right\}$
\\
\vspace{4mm}
{\bf 10(1,1)b:} & $\bbQ(\sqrt{17}), \;\;t^2 - t - 4, \;\;-\frac{1}{2} g - \frac{13}{16}, \;\; \left\{\frac{1}{4}, \frac{1}{2} g + \frac{3}{4}, \frac{3}{4}, \frac{1}{2} g - \frac{1}{4}, \frac{1}{2} g + \frac{1}{4}\right\}$
\\
{\bf 10(2):}
&  $\bbQ(\sqrt{73}), \;\;t^2 - t - 18,\;\; \frac{1}{9} g - \frac{205}{144},$\\
& \hspace{10mm} $\;\; \left\{\frac{1}{6} g + \frac{1}{12}, \frac{1}{6} g - \frac{11}{12}, \frac{1}{6} g + \frac{7}{12}, \frac{1}{3} g - \frac{7}{12}, \frac{1}{3} g - \frac{1}{12} \right\}$\\
\vspace{4mm}
&  $\bbQ(\sqrt{-7}), \;\;t^2 - t + 2,\;\; -\frac{1}{2} g - \frac{5}{16}, \;\; \left\{\frac{1}{4}, \frac{1}{2} g - \frac{1}{4}, \frac{1}{2} g + \frac{1}{4}, \frac{3}{4}, \frac{1}{2} g + \frac{3}{4}\right\}$
\\
{\bf 10(2,1,1)a:}
&  $\bbQ(\sqrt{17}), \;\; t^2 - t - 4,\;\; -\frac{273}{64}, \;\; \left\{\frac{11}{8}, \frac{13}{8}, \frac{19}{8}, \frac{5}{4} g - \frac{5}{8}, \frac{21}{8}\right\}$\\
\vspace{4mm}
&  $\bbQ(\sqrt{-1}), \;\;t^2 + 1,\;\; \frac{3}{8} g - \frac{1}{4}, \;\; \left\{\frac{3}{4} g + \frac{1}{4},  \frac{3}{4} g - \frac{3}{4},   \frac{1}{4} g - \frac{1}{4},  \frac{1}{4} g + \frac{3}{4},   \frac{1}{4} g - \frac{5}{4}\right\}$
\\
{\bf 10(2,1,1)b:}
&  $\bbQ(\sqrt{13}), \;\;t^2 - t - 3, \;\;-\frac{10}{9}, \;\; \left\{\frac{2}{3}, \frac{4}{3}, \frac{5}{3}, \frac{1}{3} g - \frac{2}{3}, \frac{1}{3} g + \frac{1}{3} \right\}$\\
\vspace{4mm}
&  $\bbQ(\sqrt{-7}), \;\; t^2 -t + 2,\;\; -\frac{21}{16}, \;\; \left\{\frac{1}{4}, \frac{7}{4}, \frac{1}{2} g - \frac{1}{4}, \frac{3}{4}, \frac{5}{4}\right\}$
\\
\vspace{4mm}
{\bf 10(3)a:} & $\bbQ(\sqrt{41}), \;\;t^2 - t - 10,\;\; -\frac{29}{16}, \;\; \left\{\frac{1}{4}, \frac{5}{4}, \frac{3}{4}, \frac{1}{2} g - \frac{1}{4}, \frac{7}{4}\right\}$
\\
\vspace{4mm}
{\bf 10(3)b:} & $\bbQ(\sqrt{57}), \;\;t^2 - t - 14, \;\;-\frac{29}{16}, \;\; \left\{  \frac{1}{4}, \frac{3}{4}, \frac{5}{4},  \frac{7}{4},  \frac{1}{2} g - \frac{1}{4}\right\}$
\\
\vspace{4mm}
{\bf 10(3,1,1)} & $\bbQ(\sqrt{337}), \;\;t^2 - t - 84, \;\;-\frac{301}{144}, \;\; \left\{\frac{5}{12}, \frac{19}{12}, \frac{23}{12}, \frac{1}{6}g + \frac{5}{12}, \frac{1}{6}g - \frac{7}{12} \right\}$
\\
\vspace{4mm}
{\bf 10(3,2):} & $\bbQ(\sqrt{193}), \;\;t^2 - t - 48,\;\; -\frac{301}{144}, \;\; \left\{\frac{5}{12}, \frac{19}{12}, \frac{23}{12}, \frac{1}{6} g + \frac{5}{12}, \frac{1}{6} g - \frac{7}{12} \right\}$
\\
\vspace{4mm}
{\bf 12(2):} & $\bbQ(\sqrt{2}), \;\;t^2 - 2, \;\;-\frac{15}{8}, \;\; \left\{\frac{3}{4} g + \frac{1}{2}, \frac{3}{4} g - \frac{1}{2}, \frac{1}{4} g + \frac{1}{2}, \frac{1}{4} g - \frac{3}{2}, \frac{1}{4} g - \frac{1}{2}, \frac{1}{4} g + \frac{3}{2}\right\}$
\\
\vspace{4mm}
{\bf 12(2,1,1)a:} & $\bbQ(\sqrt{17}), \;\;t^2 - t - 4, \;\;-\frac{13}{16}, \;\; \left\{\frac{1}{4}, \frac{3}{4}, \frac{5}{4},  \frac{1}{2} g + \frac{1}{4}, \frac{1}{2} g - \frac{3}{4}, \frac{1}{2} g - \frac{1}{4}\right\}$
\\
{\bf 12(2,1,1)b:}
&  $\bbQ(\sqrt{33}), \;\; t^2 - t - 8,\;\; -\frac{45}{16}, \;\; \left\{\frac{3}{4}, \frac{9}{4},  \frac{5}{4}, \frac{1}{2} g - \frac{3}{4}, \frac{1}{2} g + \frac{1}{4}, \frac{1}{2} g - \frac{1}{4}\right\}$\\
\vspace{4mm}
&  $\bbQ(\sqrt{-7}), \;\;t^2 - t + 2, \;\;-\frac{5}{16}, \;\; \left\{\frac{1}{4}, \frac{3}{4}, \frac{5}{4},  \frac{1}{2} g + \frac{1}{4}, \frac{1}{2} g - \frac{3}{4}, \frac{1}{2} g - \frac{1}{4}\right\}$
\\
\vspace{4mm}
{\bf 12(3):} & $\bbQ(\sqrt{73}), \;\;t^2 - t - 18, \;\;-\frac{301}{144}, \;\; \left\{\frac{1}{6} g - \frac{1}{12},  \frac{5}{12}, \frac{19}{12}, \frac{1}{3} g + \frac{1}{12}, \frac{1}{3} g - \frac{5}{12}, \frac{23}{12}\right\}$
\\
{\bf 12(4):} & $\bbQ(\sqrt{105}), \;\;t^2 - t - 26,\;\; -\frac{95}{48}$, \\
\vspace{4mm}
& \hspace{10mm} $ \;\; \left\{\frac{1}{6} g - \frac{13}{12}, \frac{1}{6} g + \frac{11}{12}, \frac{1}{3} g - \frac{5}{12}, \frac{1}{6} g + \frac{5}{12}, \frac{1}{6} g - \frac{7}{12}, \frac{1}{3} g + \frac{1}{12}\right\}$
\\
{\bf 12(4,2):} & $\bbQ(\sqrt{-15}), \;\;t^2 - t + 4,\;\; -\frac{31}{48}$,\\
\vspace{4mm}
& \hspace{10mm} $ \;\; \left\{ \frac{1}{3} g + \frac{1}{12}, \frac{1}{6} g - \frac{13}{12}, \frac{1}{3} g - \frac{5}{12}, \frac{1}{6} g + \frac{5}{12}, \frac{1}{6} g - \frac{7}{12},  \frac{1}{6} g + \frac{11}{12}\right\}$
\\
{\bf 12(6):} & $\bbQ(\sqrt{33}), \;\;t^2 - t - 8, \;\;-\frac{71}{48}$, \\
\vspace{4mm}
& \hspace{10mm} $ \;\; \left\{\frac{1}{6} g - \frac{13}{12}, \frac{1}{6} g - \frac{7}{12}, \frac{1}{3} g - \frac{5}{12}, \frac{1}{6} g + \frac{5}{12}, \frac{1}{3} g + \frac{1}{12}, \frac{1}{6} g + \frac{11}{12}\right\}$
\\
\vspace{4mm}
{\bf 14(2,1,1):} & $\bbQ(\sqrt{17}), \;\;t^2 - t - 4,\;\; -\frac{21}{16}, \;\; \left\{\frac{1}{4},  \frac{3}{4}, \frac{5}{4}, \frac{7}{4}, \frac{1}{2} g - \frac{1}{4}, \frac{1}{2} g - \frac{3}{4}, \frac{1}{2} g + \frac{1}{4} \right\}$
\\
\vspace{4mm}
{\bf 14(3,1,1):} & $\bbQ(\sqrt{33}), \;\;t^2 - t - 8, \;\;-\frac{29}{16}, \;\; \left\{\frac{1}{4}, \frac{5}{4}, \frac{3}{4}, \frac{1}{2} g - \frac{3}{4}, \frac{1}{2} g + \frac{1}{4}, \frac{1}{2} g - \frac{1}{4}, \frac{7}{4}\right\}$
\\
{\bf 14(3,2):} & $\bbQ(\sqrt{17}), \;\;t^2 - t - 4,\;\; -\frac{29}{16}, \;\; \left\{\frac{1}{4}, \frac{5}{4}, \frac{3}{4}, \frac{1}{2} g - \frac{1}{4}, \frac{1}{2} g - \frac{3}{4}, \frac{1}{2} g + \frac{1}{4}, \frac{7}{4}\right\}$
\end{longtable}

\bibliography{C:/Dropbox/jdoyle}
\bibliographystyle{amsplain}

\end{document}